\documentclass[12pt]{amsart}
\usepackage{amsmath, amsthm, amssymb}
\usepackage{amsfonts}
\usepackage{fullpage}
\usepackage{color}
\usepackage{hyperref}
\usepackage{enumitem}
\usepackage{bm}
\usepackage{verbatim}

\usepackage[dvipsnames]{xcolor}
\usepackage{graphicx}  
\DeclareGraphicsExtensions{.pstex,.eps}

\usepackage{xcolor}

\renewcommand{\ell}{{l}}  %Because jmet was using this when it wasn't used elsewhere.  sorry.

\newcommand{\bX}{\mathbf{X}}
\newcommand{\bZ}{\mathbf{Z}}
\newcommand{\bY}{\mathbf{Y}}
\newcommand{\bH}{\mathbf{H}}
\newcommand{\bh}{\mathbf{h}}

\newcommand{\R}{{\mathbb{R}}}
\newcommand{\Z}{{\mathbb Z}}

\renewcommand{\S}{{\mathbb S}}

\newcommand{\N}{{\mathbb N}}

\renewcommand{\AA}{{\mathcal A}} 
\newcommand{\BB}{{\mathcal B}}
\newcommand{\CC}{{\mathcal C}}
\newcommand{\EE}{{\mathcal E}}

\newcommand{\bEE}{{\bm{ \mathcal E}}}
\newcommand{\MM}{{\mathcal M}}
\newcommand{\FF}{{\mathcal F}}

\newcommand{\GG}{{\mathcal G}}
\newcommand{\HH}{{\mathcal H}}
\newcommand{\PP}{{\mathcal P}}
\newcommand{\QQ}{{\mathcal Q}}
\renewcommand{\SS}{{\mathcal S}}
\newcommand{\VV}{{\mathcal V}}

\newcommand{\tF}{{\tilde{F}}}

\newcommand{\tih}{{\tilde h}}
\newcommand{\ts}{{\tilde s}}
\newcommand{\tR}{{\tilde{R}}}
\newcommand{\tg}{{\tilde g}}

\newcommand{\tr}{\operatorname{tr}}

\renewcommand{\Re}{\mathop{\rm Re}\nolimits}
\renewcommand{\Im}{\mathop{\rm Im}\nolimits}

\DeclareMathOperator{\Ric}{Ric}
\DeclareMathOperator{\tRic}{{\widetilde{Ric}}}

\theoremstyle{plain}

\newtheorem{thm}{Theorem}[section]
\newtheorem{prop}[thm]{Proposition}

\newtheorem{lemma}[thm]{Lemma}

\theoremstyle{definition}

\newtheorem{rem}{Remark}[thm]

\numberwithin{equation}{section}

\newcommand{\secref}[1]{Section~\ref{#1}}

\def\squarebox#1{\hbox to #1{\hfill\vbox to #1{\vfill}}}

\newcommand{\<}{\langle}
\renewcommand{\>}{\rangle}

\renewcommand{\d}{\partial}
\newcommand{\ep}{\epsilon}
\newcommand{\lV}{\lVert}
\newcommand{\rV}{\rVert}

\def\be{{\beta}}
\def\ga{\gamma}
\def\Ga{\Gamma}
\def\de{\delta}
\def\De{\Delta}
\def\ep{\epsilon}

\def\la{\lambda}
\def\La{\Lambda}
\def\si{\sigma}
\def\Si{\Sigma}
\def\om{\omega}

\def\nab{\nabla}

\def\al{\alpha}

\renewcommand{\dh}{\delta h}
\newcommand{\dA}{\delta A}
\newcommand{\dB}{\delta B}
\newcommand{\dV}{\delta V}
\newcommand{\dpsi}{\delta \psi}
\newcommand{\dla}{\delta \lambda}
\newcommand{\dSS}{\delta \SS}
\newcommand{\dK}{\delta \mathcal K}
\newcommand{\dP}{\delta \mathcal P}
\newcommand{\dC}{\delta \mathcal C}

\newcommand{\bmF}{\mathbf{F}}
\newcommand{\tbmF}{\tilde{\mathbf{F}}}

\title[Skew Mean Curvature Flow]
{Local well-posedness of Skew mean curvature flow for small data in $d\geq 4$ dimensions}

\author[J. Huang]
{Jiaxi Huang}

\author[D. Tataru]
{Daniel Tataru}

\address{Beijing International Center for Mathematical Research, Peking University, Beijing 100871, P. R. China}
\email{huangjiaxi@bicmr.pku.edu.cn}

\address{Department of Mathematics, University of California, Berkeley \\
 Berkeley, CA 94720, USA}
\email{tataru@math.berkeley.edu}



\thanks{J. Huang was partially supported by China Postdoctoral Science Foundation Grant 2021M690223 and the NSFC Grant No. 11771415, and was also sponsored by the China Scholarship Council (No. 201806340044) for one year at University of California, Berkeley. }

\thanks{D. Tataru was supported by the NSF grant DMS-1800294 as well as by a Simons Investigator grant from the Simons Foundation.   }

\begin{document}

\begin{abstract}
The skew mean curvature flow is an evolution equation for $d$ dimensional ma\-nifolds
embedded in $\R^{d+2}$ (or more generally, in a Riemannian manifold). It can be viewed as a Schr\"odinger analogue of the mean curvature 
flow, or alternatively  as a quasilinear version of the Schr\"odinger Map equation.
In this article, we prove small data local well-posedness in low-regularity Sobolev spaces for the skew mean curvature flow in dimension $d\geq 4$. 
\end{abstract}

\date{\today}
\maketitle

%\centerline{\today}

\setcounter{tocdepth}{1}
\pagenumbering{roman} \tableofcontents \newpage \pagenumbering{arabic}

\section{Introduction}
The skew mean curvature flow (SMCF) is a nonlinear Schr\"odinger type flow 
modeling the evolution of a $d$ dimensional oriented manifold  embedded into a fixed oriented $d+2$ dimensional manifold. It can be seen as a Schr\"odinger analogue of the well studied mean curvature flow. In this article, we consider the small data local well-posedness for the skew mean curvature flow in high dimensions $d \geq 4$, for low regularity initial data.

\subsection{The (SMCF) equations}
Let $\Sigma^d$ be a $d$-dimensional oriented manifold, and $(\mathcal{N}^{d+2},g_{\mathcal{N}})$ be a $d+2$-dimensional oriented Riemannian manifold. Let $I=[0,T]$ be an interval and $F:I\times \Sigma^d \rightarrow \mathcal{N}$ be a 
one parameter family of immersions. This induces a time dependent Riemannian structure 
on $\Sigma^d$. For each $t\in I$, we denote the submanifold by $\Sigma_t=F(t,\Sigma)$,  its tangent bundle by $T\Sigma_t$, and its normal bundle by $N\Sigma_t$ respectively. For an arbitrary vector $Z$ at $F$ we denote by $Z^\perp$
its orthogonal projection onto $N\Sigma_t$.
The mean curvature  $\mathbf{H}(F)$ of $\Sigma_t$ can be identified naturally with a section of the normal bundle $N\Sigma_t$.

  The normal bundle $N\Sigma_t$ is a rank two vector bundle with a naturally induced complex structure $J(F)$ which simply rotates a vector in the normal space by $\pi/2$ positively. Namely, for any point $y=F(t,x)\in \Sigma_t$ and any normal vector $\nu\in N_{y}\Sigma_t$, we define $J(F)\in N_{y}\Sigma_t$ as the unique vector with the same length
  so that 
  \[
  J(F)\nu\bot \nu, \qquad \omega(F_{\ast}(e_1), F_{\ast}(e_2),\cdots F_{\ast}(e_d), \nu,  J(F)\nu)>0,
  \]
  where $\om$ is the volume form of $\mathcal{N}$ and $\{e_1,\cdots,e_d\}$ is an oriented basis of $\Sigma^d$. The skew mean curvature flow (SMCF) is defined by the initial value problem
\begin{equation}           \label{Main-Sys}
\left\{\begin{aligned}
&(\d_t F)^{\perp}=J(F)\mathbf{H}(F),\\
&F(\cdot,0)=F_0,
\end{aligned}\right.
\end{equation}
which evolves a codimension two submanifold along its binormal direction with a speed given by its mean curvature.

The (SMCF) was derived both in physics and mathematics. 
The one-dimensional (SMCF) in the Euclidean space $\R^3$ is the well-known vortex filament equation (VFE)
\begin{align*}
\d_t \ga=\d_s \ga\times \d_s^2 \ga,
\end{align*}
where $\ga$ is a time-dependent space curve, $s$ is its arc-length parameter and $\times$ denotes the cross product in $\R^3$. The (VFE) was first discovered by Da Rios \cite{DaRios1906} in 1906 in the study of the free motion of a vortex filament.

The (SMCF) also arises in the study of asymptotic dynamics of vortices in the context of superfluidity and superconductivity. For the Gross-Pitaevskii equation, which models the wave function associated with a Bose-Einstein condensate, physics evidence indicates that the vortices would evolve along the (SMCF). An incomplete  verification was attempted  by T. Lin \cite{LinT00} for the vortex filaments in three space dimensions. For higher dimensions, Jerrard \cite{Je02} proved this conjecture when the initial singular set is a codimension two sphere with multiplicity one.

The other motivation is that the (SMCF) naturally arises in the study of the hydrodynamical Euler equation. A singular vortex in a fluid is called a vortex membrane in higher dimensions if it is supported on a codimension two subset. The law of locally induced motion of a vortex membrane can be deduced from the Euler equation by applying the Biot-Savart formula. Shashikanth \cite{Sh12} first investigated the motion of a vortex membrane in $\R^4$ and showed that it is governed by the two dimensional (SMCF), while Khesin \cite{Kh12} then generalized this conclusion to any dimensional vortex membranes in Euclidean spaces.

From a mathematical standpoint, the (SMCF) equation is a canonical geometric flow for codimension two submanifolds which can be viewed as the Schr\"odinger analogue of the well studied mean curvature flow. In fact, the infinite-dimensional space of codimension two immersions of a Riemannian manifold admits a generalized Marsden-Weinstein sympletic structure, and hence the Hamiltonian flow of the volume functional on this space is verified to be the (SMCF). Haller-Vizman \cite{HaVi04} noted this fact where they studied the nonlinear Grassmannians.
For a detailed mathematical derivation of these equations we refer the reader  to the article \cite[Section 2.1]{SoSun17}. 

The study of higher dimensional (SMCF) is still at its infancy compared with its one-dimensional case. For the $1$-d case, we refer the reader to the survey article of Vega~\cite{Ve14}. For the higher dimensional case, Song-Sun \cite{SoSun17} proved the local existence of (SMCF) with a smooth, compact oriented surface as the initial data in two dimensions, then Song \cite{So19} generalized this result to compact oriented manifolds for all $d\geq 2$ and also 
proved a corresponding uniqueness result. Song \cite{So17} also proved that the Gauss map of a $d$ dimensional (SMCF) in $\R^{d+2}$ satisfies a Schr\"odinger Map type equation but relative to the varying metric. We remark that in one space dimension this is exactly the classical 
Schr\"odinger Map type equation, provided that one chooses suitable coordinates, i.e. the arclength
parametrization.

As written above, the (SMCF) equations 
are independent of the choice of coordinates in $I \times \Sigma$; here we include 
the time interval $I$ to emphasize that coordinates may be chosen in a time dependent fashion. The manifold $\Sigma^d$ simply serves to provide a parametrization for the moving manifold $\Sigma_t$; it  determines the topology of $\Sigma_t$, but nothing else. 
Thus, the (SMCF) system written in the form \eqref{Main-Sys} should be seen as a geometric 
evolution, with a large gauge group, namely the group of time dependent changes 
of coordinates in $I \times \Sigma$.  In particular, interpreting
the equations \eqref{Main-Sys} as a nonlinear Schr\"odinger equation will require 
a good gauge choice. This is further discussed in \secref{gauge}.

In this article we will restrict ourselves to the case when $\Sigma^d = \R^d$, i.e.
where $\Sigma_t$ has a trivial topology. We will further restrict to the case 
when $\mathcal{N}^{d+2}$ is the Euclidean space $\R^{d+2}$. Thus, the reader should visualize $\Sigma_t$ as an asymptotically flat codimension two submanifold 
of $\R^{d+2}$.

\subsection{Scaling and function spaces}
To understand what are the natural thresholds for local well-posedness, it is interesting  
to consider the scaling properties of the solutions. As one might expect, a clean scaling law is obtained when $\Sigma^d = \R^d$ and $\mathcal{N}^{d+2} = \R^{d+2}$. Then we have the following

\begin{prop}[Scale invariance for (SMCF)]
	Assume that $F$ is a solution of (\ref{Main-Sys}) with initial data $F(0)=F_0$. If $\la>0$ then $\tilde{F}(t,x):=\lambda^{-1}F(\lambda^2 t,\lambda x)$ is a solution of (\ref{Main-Sys}) with initial data $\tilde{F}(0)=\lambda^{-1}F_0(\lambda x)$.
\end{prop}
\begin{proof}
	Since $(\d_t F)^{\perp}=Jg^{\al\be}(\d_{\al\be}^2F-\Ga_{\al\be}^{\ga}\d_{\ga}F)$, 
	\begin{align*}
	\tilde{g}_{\al\be}(t,x)=\<\d_{\al} \tilde{F},\d_{\be}\tilde{F}\>=g_{\al\be}(\lambda^2 t,\lambda x),
	\end{align*}
	and 
	\begin{align*}
	    \tilde{\Ga}_{\al\be}^{\ga}(t,x)=\la \Ga_{\al\be}^{\ga}(\la^2t,\la x).
	\end{align*}
	Then 
	\begin{align*}
	(\d_t \tilde{F})^{\perp}=&\lambda (\d_tF)^{\perp}(\lambda^2 t,\lambda x)=\lambda Jg^{\al\be}(\lambda^2 t,\lambda x)[(\d^2_{\al\be}F-\Ga_{\al\be}^{\ga}\d_{\ga}F)(\lambda^2 t,\lambda x)]\\
	=& J \tilde{g}^{\al\be}(\d^2_{\al\be}\tilde{F}-\tilde{\Ga}_{\al\be}^{\ga}\d_{\ga}\tilde{F})(t,x).
	\end{align*}
\end{proof}

The above scaling would suggest the critical Sobolev space for our moving surfaces 
$\Sigma_t$ to be $\dot H^{\frac{d}{2}+1}$. However, instead of working directly with the 
surfaces, it is far more convenient to track the regularity at the level of the curvature
$\mathbf{H}(\Sigma_t)$, which scales at the level of $\dot H^{\frac{d}2-1}$.

\subsection{The main result}
Our objective in this paper is to establish the local well-posedness of skew mean 
curvature flow for small data at low regularity. A key observation is that providing a rigorous description  of fractional Sobolev spaces for functions (tensors) on a rough manifold is a delicate matter, which a-priori requires both a good choice of coordinates on the manifold and a good 
frame on the vector bundle (the normal bundle in our case). This is done in the next section, where we fix the gauge and write the equation as a quasilinear Schr\"odinger evolution in a good 
gauge. At this point, we content ourselves with a less precise formulation of the main result:

\begin{thm}[Small data local well-posedness]   \label{LWP-thm}
	Let $s>\frac{d}{2}$, $d\geq 4$. Then there exists $\ep_0>0$ sufficiently small such that, for all initial data $\Sigma_0$ with metric $\|\partial_x(g_0-I)\|_{H^{s+1}}\leq \ep_0$ and mean curvature $\lV \mathbf{H}_0  \rV_{H^s(\Sigma_0)}\leq \ep_0$, the skew mean curvature flow \eqref{Main-Sys} for maps from $\R^d$ to the Euclidean space $(\R^{d+2},g_0)$ is locally well-posed on the time interval $I=[0,1]$ in a suitable gauge. 
\end{thm}

\begin{rem}
We remark on the necessity of having a smallness condition on both $g_0-I$ and the mean curvature 
$\mathbf{H}_0$.
The combined efforts of E. De Giorgi \cite{DGi-1961}, F. J. Almgren, Jr. \cite{Al-1966}, and J. Simons \cite{Si-1968}
led to the following theorem (see Theorem 4,2, \cite{CoMi-2006}):

\emph{``If $u:\R^{n-1}\rightarrow \R$ is an entire solution to the minimal surface
equation and $n\leq 8$, then $u$ is an affine function."}

\noindent However, in 1969 E. Bombieri, De Giorgi, and E. Giusti \cite{BDG} constructed entire non-affine solutions to the minimal surface equation in $\R^9$.   Hence the bound $\|\mathbf H_0\|_{H^s(\Si_0)}\leq \ep_0$ on the mean curvature does not necessarily imply that the sub-manifold is almost flat.

Here we only prove the small data local well-posedness, which means that the initial submanifold $\Sigma_0$ should be a perturbation of Euclidean plane $\R^d$.
Hence, the bound on metric $\|\partial_x(g-I)\|_{H^{s+1}}\leq \ep_0$ is also necessary in our main result, at least in very high dimension. This condition on metric will insure the existence of global harmonic coordinates (see Proposition \ref{Global-harmonic}). 
\end{rem}

Unlike any of the prior results, which prove only existence and uniqueness for smooth data,
here we consider rough data and provide a full, Hadamard style well-posedness result based on a more modern, frequency envelope approach and  using a paradifferential form for both the full and the linearized equations. For an overview 
of these ideas we refer the reader to the expository paper \cite{IT-primer}.
While, for technical reasons, this result is limited to dimensions $d \geq 4$, we expect the same strategy to also work in lower dimension; the lower dimensional case will be considered in forthcoming work.

The favourable gauge mentioned in the theorem, defined in the next section, will have two 
components:
\begin{itemize}
    \item The harmonic coordinates on the manifolds $\Sigma_t$.
    \item The Coulomb gauge for the orthonormal frame on the normal bundle.
\end{itemize}
In the next section we reformulate the (SMCF) equations as a quasilinear Schr\"odinger 
evolution for a good scalar complex variable $\psi$, which is exactly the mean curvature 
but represented in the good gauge.  There we provide an alternate formulation of the above result, 
as a well-posedness result for the $\psi$ equation. In the final section of the paper we 
close the circle and show that one can reconstruct the full (SMCF) flow starting from 
the good variable $\psi$.

One may compare our gauge choices with the prior work in \cite{SoSun17} and \cite{So19}.
There the tangential component of $\partial_t F$ in \eqref{Main-Sys} is omitted,
and the coordinates on the manifold $\Sigma_t$ are simply those transported from the initial time.
The difficulty with such a choice is that the regularity of the map $F$ is no longer
determined by the regularity of the second fundamental form, and instead there is a loss 
of derivatives which may only be avoided if the initial data is assumed to have extra regularity.
This loss is what prevents a complete low regularity theory in that approach.

Once our problem is rephrased as a nonlinear Schr\"odinger evolution,  one may compare its 
study with earlier results on general quasilinear Schr\"odinger evolutions. This story begins 
with the classical work of Kenig-Ponce-Vega \cite{KPV2,KPV3,KPV}, where local well-posedness is established for more regular and localized data. Lower regularity results in translation invariant Sobolev spaces were later established by Marzuola-Metcalfe-Tataru~\cite{MMT3,MMT4,MMT5}. 
The local energy decay properties of the Schr\"odinger equation, as developed earlier in \cite{CS,CKS,Doi,Doi1} play a key role in these results. While here we are using some of the ideas in the above papers, the present problem is both more complex and exhibits  additional structure. Because of this, new ideas and more work are required in order to close the estimates required for both the full problem and for its linearization.

\subsection{An overview of the paper}

Our first objective in this article will be to provide a self-contained formulation
of the (SMCF) flow, interpreted as a nonlinear Schr\"odinger equation for a single independent 
variable. This independent variable, denoted by $\psi$, represents the trace of the second fundamental form on $\Sigma_t$, in complex notation. In addition to the independent variables,
we will use several dependent variables, as follows:
\begin{itemize}
    \item The Riemannian metric $g$ on $\Sigma_t$.
    \item The (complex) second fundamental form $\lambda$ for $\Sigma_t$.
    \item The magnetic potential $A$, associated to the natural connection on the 
    normal bundle $N \Sigma_t$, and the corresponding temporal component $B$. 
    \item The advection vector field $V$, associated to the time dependence of our choice of coordinates.
\end{itemize}

These additional variables will be viewed as uniquely determined by our independent variable $\psi$, provided that a suitable gauge choice was made. The gauge choice involves two steps:

\begin{enumerate}
    \item[(i)] The choice of coordinates on $\Sigma_t$; here we use harmonic coordinates, with 
    suitable boundary conditions at infinity.
    \item[(ii)] The choice of the orthonormal frame on $N\Sigma_t$; here we use the Coulomb gauge,
    again assuming flatness at infinity.
\end{enumerate}

To begin this analysis, in the next section we describe the gauge choices, so that 
by the end we obtain 
\begin{enumerate}
    \item[(a)] a nonlinear Schr\"odinger equation for $\psi$, see \eqref{mdf-Shr-sys-2}.  
    \item[(b)] An elliptic fixed time system \eqref{ell-syst} for the dependent variables
    $\SS=(g,\lambda,V,A,B)$, together with suitable compatibility conditions (constraints).
\end{enumerate}
 
Setting the stage to solve these equations, in Section~\ref{Sec3} we describe the function spaces
for both $\psi$ and $\SS$. This is done at two levels, first at fixed time, which is useful
in solving the elliptic system \eqref{ell-syst}, and then using in the space-time setting, which is needed in order to solve the Schr\"odinger evolution. The fixed time spaces are classical Sobolev spaces,
with matched regularities for all the components. The space-time norms are the so called local energy spaces, as developed in \cite{MMT3,MMT4,MMT5}.

Using these spaces, in Section~\ref{Sec-Ell} we consider the solvability of the elliptic system
\eqref{ell-syst}. This is first considered and solved without reference to the constraint equations,
but then we prove that the constraints are indeed satisfied.

Finally, we turn our attention to the Schr\"odinger system \eqref{mdf-Shr-sys-2}, in several stages. 
In Section~\ref{Sec-mutilinear} we establish several multilinear and nonlinear estimates in our space-time
function spaces. These are then used in Section~\ref{Sec-LED} in order to prove local energy decay bounds
first for the linear paradifferential Schr\"odinger flow, and then for a full linear 
Schr\"odinger flow associated to the linearization of our main evolution. 
The analysis is completed in Section~\ref{Sec-LWP}, where we use the linear Schr\"odinger bounds 
in order to (i) construct solutions for the full nonlinear Schr\"odinger flow, and (ii) 
to prove the uniqueness and continuous dependence of the solutions. The analysis here broadly follows the ideas introduced in \cite{MMT3,MMT4,MMT5}, but a number of improvements are needed
which allow us to take better advantage of the structure of the (SMCF) equations.

Last but not least, in the last section we prove that the full set of variables $(g,\lambda,V,A,B)$
suffice in order to uniquely reconstruct the defining function $F$ for the evolving surfaces $\Sigma_t$,
as $H^{s+2}_{loc}$ manifolds. More precisely, with respect to the parametrization provided by  our chosen gauge, $F$ has regularity 
\[
\partial_t F,\ \partial_x^2 F \in C[0,1;H^s]. 
\]

\section{The differentiated equations and the gauge choice}
\label{gauge}

The goal of this section is to introduce our main independent variable $\psi$, which represents 
the trace of the second fundamental form in complex notation, as well as the 
following auxiliary variables: the metric $g$, the second fundamental form $\lambda$,
the connection coefficients $A,B$ for the  normal bundle as well as the advection vector field $V$.
For $\psi$ we start with \eqref{Main-Sys} and derive a nonlinear Sch\"odinger type system \eqref{mdf-Shr-sys-2}, with coefficients depending on  $\SS=(\lambda,h,V,A,B)$, where $h=g-I_d$. Under suitable gauge conditions, the auxiliary variables $\SS$ are  shown to satisfy an elliptic system \eqref{ell-syst}, as well as a natural set of constraints. We conclude the section with a gauge formulation of our main result, see Theorem~\ref{LWP-MSS-thm}.

We remark that H. Gomez (\cite[Chapter 4]{Go}) introduced the language of gauge fields as an appropriate framework for presenting the structural properties of the surface and the evolution equations of its geometric quantities, and showed that the complex mean curvature of the evolving surface satisfies a nonlinear Schr\"{o}dinger-type equation. Here we will further derive the self-contained modified Schr\"{o}dinger system under harmonic coordinate conditions and Coulomb gauge.

\subsection{The Riemannian metric \texorpdfstring{$g$}{}.} Let $(\Sigma^d,g)$ be a $d$-dimensional oriented manifold and let $(\R^{d+2},g_0)$ be $(d+2)$-dimensional Euclidean space. Let $\al,\be,\ga,\cdots \in\{1,2,\cdots,d\}$ and $k\in\{1,2,\cdots,d+2\}$. Considering the immersion $F:\Sigma\rightarrow (\R^{d+2},g_0)$, we obtain the induced metric $g$ in $\Sigma$,
\begin{equation}        \label{g_metric}
	g_{\al\be}=\d_{x_{\al}} F\cdot \d_{x_{\be}} F.
\end{equation}
We denote the inverse of the matrix $g_{\al\be}$ by $g^{\al\be}$, i.e.
\begin{equation*}
g^{\al\be}:=(g_{\al\be})^{-1},\quad g_{\al\ga}g^{\ga\be}=\delta_{\al}^{\be}.
\end{equation*}

Let $\nab$ be the cannonical Levi-Civita connection in $\Si$ associated with the induced metric $g$.
A direct computation shows that on the Riemannian manifold $(\Si,g)$ we have the Christoffel symbols
\begin{align*}
	\Gamma^{\ga}_{\al\be}=\ \frac{1}{2}g^{\ga\si}(\d_{\be}g_{\al\si}+\d_{\al}g_{\be\si}-\d_{\si}g_{\al\be})
	=\ g^{\ga\si}\d^2_{\al\be}F\cdot\d_\si F.
\end{align*}
Hence, the Laplace-Beltrami operator $\Delta_g$ can be written in the form
\begin{align*}
	\Delta_g f=&\ \tr\nab^2 f=g^{\al\be}(\d_{\al\be}^2f-\Gamma^{\ga}_{\al\be}\d_{\ga} f)\\
	=&\ g^{\al\be}[\d_{\al\be}^2f-g^{\ga\si}(\d_{\al\be}^2F\cdot \d_{\si} F)\d_{\ga} f],
\end{align*}
for any twice differentiable function $f:\Sigma\rightarrow \R$. The curvature tensor $R$ on the Riemannian manifold $(\Sigma,g)$ is given by 
\begin{align*}
R(\d_{\al},\d_{\be})\d_{\ga}
=&\ (\d_{\al} \Gamma_{\be\ga}^{\si} +\Gamma_{\be\ga}^m\Gamma_{\al m}^{\si} -\d_{\be} \Gamma_{\al\ga}^{\si} -\Gamma_{\al\ga}^m\Gamma_{\be m}^{\si} )\d_{\si}.
\end{align*}
Hence, we have
\begin{align}             \label{R}
R_{\ga\al\be}^{\si}=\d_{\al} \Gamma_{\be\ga}^{\si}-\d_{\be} \Gamma_{\al\ga}^{\si} +\Gamma_{\be\ga}^m\Gamma_{\al m}^{\si}  -\Gamma_{\al\ga}^m\Gamma_{\be m}^{\si}.
\end{align}
By $R(X,Y,Z,W)=\<R(Z,W)Y,X\>$ and $R_{\al\be\ga\si}=R(\d_{\al},\d_{\be},\d_{\ga},\d_{\si})$, we get
\begin{equation*}
R_{\al\be\ga\si}=\<R(\d_{\ga},\d_{\si})\d_{\be},\d_{\al}\>=\<R_{\be\ga\si}^{\mu} \d_{\mu},\d_{\al}\>=g_{\mu\al}R_{\be\ga\si}^{\mu},
\end{equation*}
We will also use the Ricci curvature
\begin{equation*}
\Ric_{\al\be}=R^{\si}_{\al\si\be}=g^{\si\ga}R_{\ga\al\si\be}.
\end{equation*}

\subsection{The second fundamental form}
Let $\bar{\nab}$ be the Levi-Civita connection in $(\R^{d+2},g_0)$ and let $\bh$ be the second fundamental form for $\Sigma$ as an embedded manifold. For any vector fields $u,v\in T_{\ast}\Sigma$, the Gauss relation is 
\begin{equation*}
	\bar{\nab}_u F_{\ast}v=F_{\ast}(\nab_u v)+ \bh(u,v).
\end{equation*}
Then we have
\begin{align*}
	\bh_{\al\be}=&\bh(\d_{\al},\d_{\be})=\bar{\nab}_{\d_{\al}} \d_{\be} F-F_{\ast}(\nab_{\d_{\al}}\d_{\be})\\
	=&\d_{\al\be}^2 F+\bar{\Gamma}_{kl}\d_{\al} F^k \d_{\be} F^l-\Gamma_{\al\be}^{\ga} \d_{\ga} F.
\end{align*}
By $\bar{\Gamma}_{kl}^j=0$, this gives the mean curvature $\bH$ at $F(x)$,
\begin{equation*}
	\bH=\tr_g \bh=g^{\al\be}\bh_{\al\be}=g^{\al\be}(\d^2_{\al\be}F-\Gamma^{\ga}_{\al\be}\d_{\ga} F)=\Delta_g F.
\end{equation*}
Hence, the $F$-equation in (\ref{Main-Sys}) is rewritten as
\begin{equation*}  %\label{F-eqn}
	(\d_t F)^{\perp}=J(F)\Delta_g F=J(F)g^{\al\be}(\d^2_{\al\be}F-\Gamma^{\ga}_{\al\be}\d_{\ga} F).
\end{equation*}
This equation is still independent of the choice of coordinates 
in $\Sigma^d$, which at this point are allowed to fully depend on $t$.

\subsection{The complex structure equations} 
Here we introduce a complex structure on the normal bundle $N\Sigma_t$. This is achieved 
by choosing $\{\nu_1,\nu_2\}$ to be an orthonormal basis of $N\Si_t$ such that
\[
J\nu_1=\nu_2,\quad J\nu_2=-\nu_1.  
\]
Such a choice is not unique; in making it we introduce a second component to our gauge group,  
namely the group of sections of an $SU(1)$ bundle over $I \times \R^d$.  

The vectors $\{ F_1,\cdots,F_d,\nu_1,\nu_2\}$ form a frame at each point on the manifold $(\Sigma,g)$, where $F_{\al}$ for $\al\in\{1,\cdots,d\}$ are defined as
\[
F_{\al}=\d_{\al}F.
\]  
If we differentiate the frame, we obtain a set of structure equations of the following type 
\begin{equation}          \label{strsys}
\left\{\begin{aligned}
&\d_{\al} F_{\be}=\Gamma^{\ga}_{\al\be}F_{\ga}+\kappa_{\al\be}\nu_1+\tau_{\al\be}\nu_2,\\
&\d_{\al} \nu_1=-\kappa^{\ga}_{\al} F_{\ga}+A_{\al} \nu_2,\\
&\d_{\al} \nu_2=-\tau^{\ga}_{\al} F_{\ga}-A_{\al} \nu_1,
\end{aligned}\right.
\end{equation}
where the tensors $\kappa_{\al\be},\tau_{\al\be}$ and the connection coefficients $A_{\al}$ are defined by
\[
\kappa_{\al\be}:=\d_{\al} F_{\be}\cdot\nu_1,\quad \tau_{\al\be}:=\d_{\al} F_{\be}\cdot\nu_2,\quad A_{\al}=\d_{\al}\nu_1\cdot\nu_2.
\]
The mean curvature $\bH$ can be expressed in term of $\kappa_{\al\be}$ and $\tau_{\al\be}$, i.e.
\[
\bH=g^{\al\be}(\kappa_{\al\be}\nu_1+\tau_{\al\be}\nu_2).
\]

Next, we complexify the structure equations (\ref{strsys}) as follows. We define the complex vector $m$ and the complex second fundamental form tensor $\la_{\al\be}$ to be 
\begin{equation*}
	m=\nu_1+i\nu_2,\quad \lambda_{\al\be}=\kappa_{\al\be}+i\tau_{\al\be}.
\end{equation*}
Then we define the \emph{complex scalar mean curvature} $\psi$ as the trace of the second fundamental form,
\begin{equation}\label{csmc}
\psi:=\tr\la=g^{\al\be}\la_{\al\be}.
\end{equation}
Our objective for the rest of this section will be to interpret the (SMCF) equation as 
a nonlinear Schr\"odinger evolution for $\psi$, by making suitable gauge choices.

We remark that the action of sections of the $SU(1)$ bundle is given by 
\begin{equation}    \label{gauge-A}
\psi \to e^{i \theta }\psi , \quad \lambda \to e^{i \theta }\lambda, \quad m \to e^{i\theta} m,
\quad A_\alpha \to A_\alpha - \partial_\alpha \theta.
\end{equation}
for a real valued function $\theta$.

We use the convention for the inner product of two complex vectors, say $a$ and $b$, as
\begin{equation*}
	\<a,b\>=\sum_{j=1}^{d+2}a_{j} \bar{b}_{j},
\end{equation*}
where $a_{j}$ and $b_{j}$ are the complex components of $a$ and $b$ respectively. Then we get the following relations for the complex vector $m$,
\begin{align*}
	\<m,m\>=|\nu_1|^2+|\nu_2|^2=2,\quad \<m,\bar{m}\>=\<\bar{m},m\>=|\nu_1|^2-|\nu_2|^2=0.
\end{align*}
From these relations we obtain
\begin{align*}
	\kappa_{\al\be}\nu_1+\tau_{\al\be}\nu_2=&\frac{1}{2}(\lambda_{\al\be}+\bar{\lambda}_{\al\be})\frac{1}{2}(m+\bar{m})+\frac{1}{2i}(\lambda_{\al\be}-\bar{\lambda}_{\al\be})\frac{1}{2i}(m-\bar{m})\\
	=&\frac{1}{2}(\lambda_{\al\be}\bar{m}+\bar{\lambda}_{\al\be}m)
	=\Re(\lambda_{\al\be}\bar{m}).
\end{align*}
Then the structure equations (\ref{strsys}) are rewritten as 
\begin{equation}               \label{strsys-cpf}
\left\{\begin{aligned}
&\d_{\al}F_{\be}=\Gamma^{\ga}_{\al\be}F_{\ga}+\Re(\lambda_{\al\be}\bar{m}),\\
&\d_{\al}^A m=-\lambda^{\ga}_{\al} F_{\ga},
\end{aligned}\right.
\end{equation}
where 
\[
\d_{\al}^A=\d_{\al}+iA_{\al}.
\]

\subsection{The Gauss and Codazzi relations} The Gauss and Codazzi equations are derived from the equality of second derivatives $\d_{\al}\d_{\be}F_{\ga}=\d_{\be}\d_{\al}F_{\ga}$ for the tangent vectors on the submanifold $\Sigma$ and for the normal vectors respectively.  Here we use the Gauss and Codazzi relations to derive the Riemannian curvature, the first compatibility condition and a symmetry. 

By the structure equations (\ref{strsys-cpf}), we get
\begin{equation}\label{d^2 F_ga}
\begin{aligned}           
\d^2_{\al\be}F_{\ga}=&\d_{\al}(\Gamma^{\si}_{\be\ga}F_{\si}+\Re(\lambda_{\be\ga}\bar{m}))\\
=&\d_{\al}\Ga_{\be\ga}^{\si} F_{\si}+\Ga_{\be\ga}^{\si} (\Gamma^{\mu}_{\al\si}F_{\mu}+\Re(\lambda_{\al\si}\bar{m}))+\Re(\d_{\al}\la_{\be\ga}\bar{m}+\la_{\be\ga}(iA_{\al}\bar{m}-\bar{\la}_{\al}^{\mu}F_{\mu}))\\
=&(\d_{\al}\Ga_{\be\ga}^{\si}+\Ga_{\be\ga}^{\mu}\Gamma^{\si}_{\al\mu}-\Re(\la_{\be\ga}\bar{\la}_{\al}^{\si}))F_{\si}+\Re[(\d_{\al}^A\la_{\be\ga}+\Ga^{\si}_{\be\ga}\la_{\al\si})\bar{m}].
\end{aligned}
\end{equation}
Then in view of $\d_{\al}\d_{\be} F_{\ga}=\d_{\be}\d_{\al} F_{\ga}$ and equating the coefficients of the tangent vectors, we obtain
\begin{equation*}
\d_{\al}\Ga_{\be\ga}^{\si}+\Ga_{\be\ga}^{\mu}\Gamma^{\si}_{\al\mu}-\d_{\be}\Ga_{\al\ga}^{\si}-\Ga_{\al\ga}^{\mu}\Gamma^{\si}_{\be\mu}=\Re(\la_{\be\ga}\bar{\la}_{\al}^{\si}-\la_{\al\ga}\bar{\la}_{\be}^{\si}).
\end{equation*}
This gives the Riemannian curvature
\begin{align}          \label{R-la}
R_{\si\ga\al\be}=\<R^{\mu}_{\ga\al\be}F_{\mu},F_{\si}\>=\<R(\d_{\al},\d_{\be})F_{\ga},F_{\si}\>=\Re(\lambda_{\be\ga}\bar{\la}_{\al\si}-\la_{\al\ga}\bar{\la}_{\be\si}),
\end{align}
which is a complex formulation of the Gauss equation.
Correspondingly we obtain the the Ricci curvature
\begin{equation}\label{Ric}
\Ric_{\ga\be}=\Re (\la_{\ga\be}\bar{\psi}-\la_{\ga\al}\bar{\la}_{\be}^{\al}).
\end{equation}
After equating the coefficients of the vector $m$ in (\ref{d^2 F_ga}), we obtain 
\begin{equation*}
\d^A_{\al}\la_{\be\ga}+\Ga_{\be\ga}^{\si}\la_{\al\si}=\d^A_{\be}\la_{\al\ga}+\Ga_{\al\ga}^{\si}\la_{\be\si},
\end{equation*}
By the definition of covariant derivatives, i.e.
\begin{equation*}
\nab_{\al} \la_{\be\ga}=\d_{\al}\la_{\be\ga}-\Ga_{\al\be}^{\si} \la_{\si\ga}-\Ga_{\al\ga}^{\si}\la_{\be\si},
\end{equation*}
we obtain 
\begin{equation*}
\d^A_{\al}\la_{\be\ga}-\Ga_{\al\ga}^{\si}\la_{\be\si}-\Ga_{\al\be}^{\si}\la_{\si\ga}=\d^A_{\be}\la_{\al\ga}-\Ga_{\be\ga}^{\si}\la_{\al\si}-\Ga_{\al\be}^{\si}\la_{\si\ga}.
\end{equation*}
This implies the complex formulation of the Codazzi equation, namely
\begin{equation}         \label{la-commu}
\nab^A_{\al} \la_{\be\ga}=\nab^A_{\be}\la_{\al\ga}.
\end{equation}
As a consequence of this equality, we obtain
\begin{lemma}
The second fundamental form $\lambda$ satisfies the Codazzi relations
	\begin{equation}       \label{comm-la}
	\nab^A_{\al} \la^{\ga}_{\be}=\nab^A_{\be} \la^{\ga}_{\al}=\nab^{A,\ga}\la_{\al\be}.
	\end{equation}
\end{lemma}
\begin{proof}
	Here we prove the last equality. By $\nab_{\be}g^{\ga\si}=0$ and (\ref{la-commu}) we have
	\begin{align*}
	\nab^A_{\be}\la^{\ga}_{\al}=g^{\ga\si}\nab_{\be}^A \la_{\si\al}=g^{\ga\si}\nab^A_{\si}\la_{\be\al}=\nab^{A,\ga}\la_{\al\be}.
	\end{align*}
	The first equality can be proved similarly.
\end{proof}

Next, we use the relation $\d_{\al}\d_{\be}m=\d_{\be}\d_{\al}m$ in order to derive a compatibility condition between the connection $A$ in the normal bundle and the second fundamental form.
Indeed, from $\d_{\al}\d_{\be}m=\d_{\be}\d_{\al}m$ we obtain the commutation relation
\begin{equation} \label{comp-cond_m-pre}
[\d^A_{\al},\d^A_{\be}]m=i(\d_{\al} A_{\be}-\d_{\be} A_{\al})m.
\end{equation}
By (\ref{strsys-cpf}) we have
\begin{align*}
\d^A_{\al}\d^A_{\be} m=&-\d^A_{\al}(\la^{\ga}_{\be} F_{\ga})=-(\d^A_{\al}\la^{\si}_{\be}+\la^{\ga}_{\be}\Ga^{\si}_{\al\ga})F_{\si}-\la^{\ga}_{\be}\Re(\la_{\al\ga}\bar{m}).
\end{align*}
Then multiplying (\ref{comp-cond_m-pre}) by $m$ yields
\begin{align*}
2i(\d_{\al} A_{\be}-\d_{\be} A_{\al})=&\<[-\la^{\ga}_{\be}\Re(\la_{\al\ga}\bar{m})+\la^{\ga}_{\al}\Re(\la_{\be\ga}\bar{m})],m\>\\
=&-\la^{\ga}_{\be} \bar{\la}_{\al\ga}+\la^{\ga}_{\al}\bar{\la}_{\be\ga}=2i\Im(\la^{\ga}_{\al}\bar{\la}_{\be\ga}).
\end{align*}
This gives the compatibility condition for the curvature of $A$,
\begin{equation}          \label{cpt-AiAj}
\d_{\al} A_{\be}-\d_{\be} A_{\al}=\Im(\la^{\ga}_{\al}\bar{\la}_{\be\ga}).
\end{equation}
Using covariant differentiation, this can be written as
\begin{equation}          \label{cpt-AiAj-2}
\nab_{\al} A_{\be}-\nab_{\be} A_{\al}=\Im(\la^{\ga}_{\al}\bar{\la}_{\be\ga}),
\end{equation}
which can be seen as the complex form of the Ricci equations.

\medskip

We remark that, by equating the coefficients of the tangent vectors in (\ref{comp-cond_m-pre}), we also obtain
\begin{equation*}
\d^A_{\al}\la^{\si}_{\be}+\la^{\ga}_{\be}\Ga^{\si}_{\al\ga}=\d^A_{\be}\la^{\si}_{\al}+\la^{\ga}_{\al}\Ga^{\si}_{\be\ga},
\end{equation*}
and hence 
\begin{equation*}
\nab^A_{\al}\la^{\si}_{\be}=\nab^A_{\be}\la^{\si}_{\al},
\end{equation*}
which is the same as (\ref{comm-la}).

Next, we state an elliptic system for the second fundamental form $\la_{\al\be}$ in terms of $\psi$, using the Codazzi relations (\ref{comm-la}).

\begin{lemma}[Div-curl system for $\la$] The second fundamental form $\lambda$ satisfies
	\begin{equation}\label{la-eq}
	\left\{\begin{aligned}
	     & \nab^A_\al \la_{\be\ga}-\nab^A_\be \la_{\al\ga}=0,\\
	     & \nab^{A,\al}\la_{\al\be}=\nab^A_\be \psi.
	\end{aligned}\right.
	\end{equation}
\end{lemma}
We remark that a-priori solutions $\lambda$ to the above system are not guaranteed to be symmetric, so we record this as a separate property:
\begin{equation}\label{lambda-sim}
\la_{\al\be} = \la_{\be\al}.    
\end{equation}

Finally, we turn our attention to the connection $A$, for which we have the curvature relations
\eqref{cpt-AiAj-2} together with the gauge group \eqref{gauge-A}. In order to both fix the 
gauge and obtain an elliptic system for $A$, we impose the Coulomb gauge condition
\begin{equation}\label{Coulomb}
    \nab^\al A_\al=0.
\end{equation}

Next, we derive the elliptic $A$-equations from the Ricci equations (\ref{cpt-AiAj-2}).
\begin{lemma}[Elliptic equations for $A$] Under the Coulomb gauge condition, the connection $A$ solves
	\begin{align} \label{Ellp-A-pre}
	\nab^{\ga}\nab_{\ga} A_{\al}=\Re(\la_\al^\si\bar{\psi}-\la_{\al}^{\ga}\bar{\la}_{\ga}^{\si})A_{\si}+\nab^{\ga}\Im(\la^{\si}_{\ga}\bar{\la}_{\al\si}).
	\end{align}
\end{lemma}
\begin{proof}
	Applying $\nab^{\be}$ to \eqref{cpt-AiAj-2}, by curvature and \eqref{Coulomb} we obtain
	\begin{align*}
	\nab^{\be}\nab_{\be} A_{\al}=\Ric _{\al\de}A^{\de}+\nab^{\be}\Im(\la^{\si}_{\be}\bar{\la}_{\al\si}).
	\end{align*}
	Then the equation \eqref{Ellp-A-pre} for $A$ is obtained from \eqref{Ric}.
\end{proof}

\subsection{The elliptic equation for the metric \texorpdfstring{$g$}{} in harmonic coordinates}    % \label{Sec-gauge}
Here we take the next step towards fixing the gauge, by choosing to work in harmonic coordinates. Precisely, we will require the coordinate functions $\{x_{\al},\al=1,\cdots,d\}$ to be globally Lipschitz solutions of the elliptic equations 
\begin{equation} \label{h-gauge}
\Delta_g x_{\al}=0.
\end{equation}
This determines the coordinates uniquely modulo time dependent affine transformations.
This remaining ambiguity will be removed later on by imposing suitable boundary conditions at infinity. After this, the only remaining degrees of freedom in the choice of coordinates will be 
given by \emph{time independent} translations and
rigid rotations. Thus, once a choice is made 
at the initial time, the coordinates will be 
uniquely determined later on (see also Remark \ref{Rmk-coordinate}).

Here we will interpret the above harmonic coordinate condition at fixed time as an elliptic equation for the metric $g$ (see e.g. \cite{Fo}, \cite[P161]{We}).
The equations \eqref{h-gauge} may be expressed in terms of the Christoffel symbols $\Ga$, which  must satisfy the condition
\begin{equation}           \label{hm-coord}
	g^{\al\be}\Ga^{\ga}_{\al\be}=0,\quad {\rm for}\ \ga=1,\cdots,d.
\end{equation}
This implies
\begin{equation}         \label{hm-coord2}
	g^{\al\be}\d_{\al} g_{\be\ga}=\frac{1}{2}g^{\al\be}\d_{\ga} g_{\al\be},\quad \d_{\al} g^{\al\ga}=\frac{1}{2}g_{\al\be}g^{\ga\si}\d_{\si} g^{\al\be}.
\end{equation}
Let
\begin{equation}   \label{Ga_albega}
	\Ga_{\al\be,\ga}=\frac{1}{2}(\d_{\al} g_{\be\ga}+\d_{\be} g_{\al\ga}-\d_{\ga} g_{\al\be})=g_{\ga\si}\Ga^{\si}_{\al\be}.
\end{equation}
Then we also have
\begin{equation*}
	g^{\al\be}\Ga_{\al\be,\ga}=g^{\al\be}g_{\ga\si}\Ga^{\si}_{\al\be}=0,
\end{equation*}
and 
\begin{equation*}
	R_{\al\ga\be\si}=\d_{\be} \Ga_{\ga\si,\al}-\d_{\si}\Ga_{\be\ga,\al}+\Ga_{\si\al,\nu}\Ga^{\nu}_{\be\ga}-\Ga_{\be\al,\nu}\Ga^{\nu}_{\ga\si}.
\end{equation*}
This leads to an equation for the metric $g$:

\begin{lemma} [Elliptic equations of $g$] In harmonic coordinates, the metric $g$ satisfies
	\begin{equation}\label{g-eq-original}
	\begin{aligned}      
		g^{\al\be}\d^2_{\al\be}g_{\ga\si}=&\ [-\d_{\ga} g^{\al\be}\d_{\be} g_{\al\si}-\d_{\si} g^{\al\be}\d_{\be} g_{\al\ga}+\d_{\ga} g_{\al\be}\d_{\si} g^{\al\be}]\\
		&+2g^{\al\be}\Ga_{\si\al,\nu}\Ga^{\nu}_{\be\ga}-2\Re (\la_{\ga\si}\bar{\psi}-\la_{\al\ga}\bar{\la}_{\si}^{\al}).
	\end{aligned}
	\end{equation}
\end{lemma}
\begin{proof}
	By the definition of Ricci curvature, (\ref{R}) and (\ref{hm-coord}), we have
	\begin{align*}
		\Ric_{\ga\si}=&g^{\al\be}R_{\al\ga\be\si}=g^{\al\be}(\d_{\be} \Ga_{\ga\si,\al}-\d_{\si}\Ga_{\be\ga,\al})+g^{\al\be}\Ga_{\si\al,\nu}\Ga^{\nu}_{\be\ga}-g^{\al\be}\Ga_{\be\al,\nu}\Ga^{\nu}_{\ga\si}\\
		=&g^{\al\be}(\d_{\be} \Ga_{\ga\si,\al}-\d_{\si}\Ga_{\be\ga,\al})+g^{\al\be}\Ga_{\si\al,\nu}\Ga^{\nu}_{\be\ga}\\
		=&I+II.
	\end{align*}
	
	We compute the first term $I$. By the definition of $\Ga_{\al\be,\ga}$ in (\ref{Ga_albega}), we have
	\begin{align*}
		I=&\frac{1}{2}g^{\al\be}[\d_{\be}(\d_{\ga} g_{\si\al}+\d_{\si} g_{\ga\al}-\d_{\al} g_{\ga\si})-\d_{\si}(\d_{\be} g_{\ga\al}+\d_{\ga}g_{\be\al}-\d_{\al} g_{\be\ga})]\\
		=&-\frac{1}{2}g^{\al\be}\d^2_{\al\be}g_{\ga\si}+\frac{1}{2}g^{\al\be}(\d^2_{\be\ga}g_{\al\si}+\d^2_{\al\si}g_{\be\ga}-\d^2_{\ga\si}g_{\al\be})
	\end{align*}
	Since, by (\ref{hm-coord2}) we have
	\begin{equation*}
		g^{\al\be}(\d^2_{\ga\be}g_{\al\si}-\frac{1}{2}\d^2_{\ga\si}g_{\al\be})=-\d_{\ga} g^{\al\be}(\d_{\be} g_{\al\si}-\frac{1}{2}\d_{\si} g_{\al\be}).
	\end{equation*}
	Then 
	\begin{align*}
		I=&-\frac{1}{2}g^{\al\be}\d^2_{\al\be}g_{\ga\si}+\frac{1}{2}[-\d_{\ga} g^{\al\be}(\d_{\be} g_{\al\si}-\frac{1}{2}\d_{\si} g_{\al\be})-\d_{\si} g^{\al\be}(\d_{\be} g_{\al\ga}-\frac{1}{2}\d_{\ga} g_{\al\be})]\\
		=&-\frac{1}{2}g^{\al\be}\d^2_{\al\be}g_{\ga\si}+\frac{1}{2}[-\d_{\ga} g^{\al\be}\d_{\be} g_{\al\si}-\d_{\si} g^{\al\be}\d_{\be} g_{\al\ga}+\d_{\ga} g_{\al\be}\d_{\si} g^{\al\be}].
	\end{align*}
	Hence, 
	\begin{equation*}
		\Ric_{\ga\si}=-\frac{1}{2}g^{\al\be}\d^2_{\al\be}g_{\ga\si}+\frac{1}{2}[-\d_{\ga} g^{\al\be}\d_{\be} g_{\al\si}-\d_{\si} g^{\al\be}\d_{\be} g_{\al\ga}+\d_{\ga} g_{\al\be}\d_{\si} g^{\al\be}]+g^{\al\be}\Ga_{\si\al,\nu}\Ga^{\nu}_{\be\ga}.
	\end{equation*}
	By (\ref{Ric}) this concludes the proof of the Lemma.
\end{proof}

\subsection{The motion of the frame \texorpdfstring{$\{F_1,\cdots,F_d,m\}$}{} under (SMCF)} Here we derive the equations of motion for the frame, assuming that the immersion $F$ satisfying (\ref{Main-Sys}). 

We begin by rewriting  the SMCF equations in the form 
\begin{equation}    \label{SMCF-general-version}
\d_t F=J(F)\bH(F)+V^{\ga} F_{\ga},
\end{equation}  
where $V^{\ga}$ is a vector field on the manifold $\Sigma$, which 
in general depends on the choice of coordinates.

By the definition of $m$ and $\la_{\al\be}$, we get
\begin{align*}
J(F)\bH(F)=J(F) \Re(\psi \bar{m})=\Re i(\psi\bar{m})=-\Im (\psi\bar{m}).
\end{align*}
Hence, the above $F$-equation (\ref{SMCF-general-version}) is rewritten as
\begin{equation}          \label{sys-cpf}
\d_t F=-\Im (\psi\bar{m})+V^{\ga} F_{\ga}.
\end{equation}

Then we use this to derive the equations of motion for the frame. Applying $\d_{\al}$ to (\ref{sys-cpf}), by the structure equations (\ref{strsys-cpf}) we obtain
\begin{align*}
\d_t F_{\al}=&\ \d_{\al} F_t=\d_{\al}[-\Im (\psi\bar{m})+V^{\ga} F_{\ga}]\\
=&-\Im ((\d_{\al}+iA_{\al})\psi \bar{m}+\psi\overline{(\d_{\al}+iA_{\al})}\bar{m})+\d_{\al} V^{\ga} F_{\ga}+V^{\ga} (\Ga^{\si}_{\al\ga}F_{\si}+\Re(\la_{\al\ga}\bar{m}))\\
=&-\Im (\d^A_{\al} \psi \bar{m}-\psi \bar{\la}^{\ga}_{\al} F_{\ga})+\d_{\al} V^{\ga} F_{\ga}+V^{\ga} (\Ga^{\si}_{\al\ga}F_{\si}+\Re(\la_{\al\ga}\bar{m}))\\
=&-\Im (\d^A_{\al} \psi \bar{m})+\Re(\la_{\al\ga}V^{\ga} \bar{m})+[\Im(\psi\bar{\la}^{\ga}_{\al})+\nab_{\al} V^{\ga}]F_{\ga}\\
=&-\Im (\d^A_{\al} \psi \bar{m}-i\la_{\al\ga}V^{\ga} \bar{m})+[\Im(\psi\bar{\la}^{\ga}_{\al})+\nab_{\al} V^{\ga} ]F_{\ga}.
\end{align*}
By the orthogonality relation  $m\bot F_{\al}=0$, this implies
\begin{align*}
\<\d_t m,F_{\al}\>=&\ \d_t\<m,F_{\al}\>-\<m,\d_t F_{\al}\>\\
=&\ -\<m, -\Im(\d^A_{\al} \psi \bar{m}-i\la_{\al\ga}V^{\ga} \bar{m})\>\\
=&\ \<m,\frac{i}{2}(\overline{\d^A_{\al} \psi -i\la_{\al\ga}V^{\ga}})m\>\\
=&\ -i(\d^A_{\al} \psi -i\la_{\al\ga}V^{\ga} ).
\end{align*}

In order to describe the normal component of the time derivative of $m$, we also need the temporal
component of the connection in the normal bundle. This is defined by
\[
B=\<\d_t \nu_1,\nu_2\>.
\]
We have
\begin{align*}
(\d_t m)^{\perp}=(\d_t(\nu_1+i\nu_2))^{\perp}=B\nu_2-iB\nu_1=-iB(\nu_1+i\nu_2)=-iB m.
\end{align*}
Then we get
\begin{equation*}
\d_t m=-i(\d^{A,\al} \psi -i\la^{\al}_{\ga}V^{\ga} )F_{\al}-iB m,
\end{equation*}
which can be further rewritten as
\begin{equation*}
\d^{B}_t m=-i(\d^{A,\al} \psi -i\la^{\al}_{\ga}V^{\ga} )F_{\al}.
\end{equation*}
Therefore, we obtain the following equations of motion for the frame
\begin{equation}              \label{mo-frame}
\left\{\begin{aligned}
&\d_t F_{\al}=-\Im (\d^A_{\al} \psi \bar{m}-i\la_{\al\ga}V^{\ga} \bar{m})+[\Im(\psi\bar{\la}^{\ga}_{\al})+\nab_{\al} V^{\ga}]F_{\ga},\\
&\d^{B}_t m=-i(\d^{A,\al} \psi -i\la^{\al}_{\ga}V^{\ga} )F_{\al}.
\end{aligned}\right.
\end{equation}

From this we obtain the evolution equation for the metric $g$. By the definition of the induced metric $g$ (\ref{g_metric}) and (\ref{mo-frame}), we have
\begin{align*}
\d_t g_{\al\be}=&\ \d_t\<F_{\al},F_{\be}\>=\<\d_t F_{\al},F_{\be}\>+\<F_{\al},\d_tF_{\be}\>\\
=&\ \<-\Im (\d^A_{\al} \psi \bar{m}-i\la_{\al\ga}V^{\ga} \bar{m})+[\Im(\psi\bar{\la}^{\ga}_{\al})+\nab_{\al} V^{\ga}]F_{\ga},F_{\be}\>\\
&\ +\<F_{\al},-\Im (\d^A_{\be} \psi \bar{m}-i\la_{\be\ga}V^{\ga} \bar{m})+[\Im(\psi\bar{\la}^{\ga}_{\be})+\nab_{\be} V^{\ga}]F_{\ga}\>\\
=&\ g_{\ga\be}(\Im(\psi\bar{\la}^{\ga}_{\al})+\nab_{\al} V^{\ga})+g_{\al\ga}(\Im(\psi\bar{\la}^{\ga}_{\be})+\nab_{\be} V^{\ga})\\
=&\ 2\Im(\psi\bar{\la}_{\al\be})+\nab_{\al}V_{\be}+\nab_{\be}V_{\al},
\end{align*}
and hence,
\begin{align}           \label{g_dt}
\d_t g^{\al\be}=-2\Im(\psi\bar{\la}^{\al\be})-\nab^{\al}V^{\be}-\nab^{\be}V^{\al},\\   \label{Ga_dt}
\d_t \Ga_{\al\be}^\ga=\nab_\al G_\be^\ga+\nab_\be G_\al^\ga-\nab^\ga G_{\al\be},
\end{align}
where $G_{\al\be}$ are defined by
\begin{equation*}
    G_{\al\be}=\Im(\psi\bar{\la}_{\al\be})+\frac{1}{2}(\nab_{\al}V_{\be}+\nab_{\be}V_{\al}).
\end{equation*}

So far, the choice of $V$ has been unspecified; it depends on the choice of coordinates 
on our manifold as the time varies. However, once the latter is fixed via the harmonic coordinate condition (\ref{hm-coord}), we can also derive an elliptic equation for the advection field $V$:

\begin{lemma}[Elliptic equation for the vector field $V$] Under the harmonic coordinate condition~\eqref{hm-coord}, the advection field $V$ solves 
	\begin{equation}\label{Ellp-X}
	\begin{aligned}
	\nab^{\al}\nab_{\al}V^{\ga}=&\ 2\Im (\nab^{A,\ga}\psi\bar{\psi}-\nab^A_{\al}\psi\bar{\la}^{\al\ga})-\Re (\la^{\ga}_{\si}\bar{\psi}-\la_{\al\si}\bar{\la}^{\al\ga})V^{\si}\\
	&\ +2(\Im(\psi\bar{\la}^{\al\be})+\nab^{\al}V^{\be})\Ga^{\ga}_{\al\be}.
	\end{aligned}
	\end{equation}
\end{lemma}
\begin{proof}
    Applying $\d_t$ to $g^{\al\be}\Ga_{\ga\be}^\ga$, by \eqref{g_dt} and \eqref{Ga_dt} we have
    \begin{align*}
        \d_t (g^{\al\be}\Ga_{\ga\be}^\ga)=&-2G^{\al\be}\Ga_{\al\be}^\ga+g^{\al\be}(2\nab_\al G_\be^\ga-\nab^\ga G_{\al\be})\\
        =& -2G^{\al\be}\Ga_{\al\be}^\ga+2\nab_\al\Im(\psi\bar{\la}^{\al\ga})+\De_g V^\ga+[\nab_\al,\nab^\ga] V^\al.
    \end{align*}
    Since 
    \begin{equation*}
        [\nab_\al,\nab^\ga] V^\al=\Ric^\ga_\si V^\si=\Re(\la^\ga_\si \bar{\psi}-\la_{\al\si}\bar{\la}^{\al\ga})V^\si.
    \end{equation*}
    By the harmonic coordinate condition \eqref{hm-coord} and \eqref{comm-la}, the above two equalities give the $V$-equations \eqref{Ellp-X}. 
\end{proof}

\begin{rem}\label{Rmk-coordinate}
Consider an arbitrary choice of  coordinates
(parametrization) $\{x_1,\cdots,x_d\}$ for the time evolving manifolds
$\Sigma_t$ for $t \in [0,T]$. This yields a representation of $\Sigma_t$ as the 
image of a map
\[
F:\R^d\times[0,T]\rightarrow \R^{d+2},
\]
restricted to time $t$. If $\Sigma_t$ moves along the (SMCF) flow \eqref{SMCF-general-version},
then we have the relation
\begin{equation*}
\d_t (g^{\al\be}\Ga_{\al\be}^\ga)=(V\ equation).
\end{equation*}
Here we uniquely determine the evolution of the coordinates as the time varies by choosing the advection vector field $V$, precisely so that it satisfies the $V$-equation \eqref{Ellp-X}. For this choice we obtain $\d_t(g^{\al\be}\Ga_{\al\be}^\ga)=0$.
This implies that $g^{\al\be}\Ga_{\al\be}^\ga$ is conserved for any $x\in \R^d$, and thus the harmonic gauge condition is propagated in time.
\end{rem}

\subsection{Derivation of the modified Schr\"{o}dinger system from SMCF}
Here we derive the main Schr\"{o}dinger equation and the second compatibility condition. We consider the commutation relation
\begin{equation} \label{com-m}
[\d^{B}_t,\d^A_{\al}]m=i(\d_t A_{\al}-\d_{\al} B)m.
\end{equation}
In order, for the left-hand side, by (\ref{strsys-cpf}) and (\ref{mo-frame}) we have
\begin{align*}
\d^{B}_t\d^A_{\al} m=&-\d^{B}_t(\la^{\ga}_{\al} F_{\ga})=-\d^{B}_t\la^{\ga}_{\al} \cdot F_{\ga}-\la^{\ga}_{\al}\cdot\d_t F_{\ga}\\
=&-[\d^{B}_t\la^{\si}_{\al}+\la^{\ga}_{\al}(\Im(\psi\bar{\la}^{\si}_{\ga})+\nab_{\ga}V^{\si})]F_{\si}+\la^{\ga}_{\al}\Im(\d^A_{\ga}\psi \bar{m}-i\la_{\ga\si}V^{\si}\bar{m}),
\end{align*}
and 
\begin{align*}
\d^A_{\al}\d^{B}_t m=&-i\d^A_{\al}[(\d^{A,\si} \psi -i\la^{\si}_{\ga}V^{\ga} )F_{\si}]\\
=&-i\d^A_{\al}(\d^{A,\si} \psi -i\la^{\si}_{\ga}V^{\ga} )F_{\si}-i(\d^{A,\si} \psi -i\la^{\si}_{\ga}V^{\ga} )[\Ga^{\mu}_{\al\si}F_{\mu}+\Re(\la_{\al\si}\bar{m})]\\
=&-i\nab^A_{\al}(\d^{A,\si} \psi -i\la^{\si}_{\ga}V^{\ga} )F_{\si}-i(\d^{A,\si} \psi -i\la^{\si}_{\ga}V^{\ga} )\Re(\la_{\al\ga}\bar{m}).
\end{align*}
Then by the above three equalities, equating the coefficients of the tangent vectors and the normal vector $m$, we obtain the evolution equation for $\la$
\begin{equation}\label{main-eq-abst}
\d^{B}_t\la^{\si}_{\al}+\la^{\ga}_{\al}(\Im(\psi\bar{\la}^{\si}_{\ga})+\nab_{\ga} V^{\si})=i\nab^A_{\al}(\d^{A,\si} \psi -i\la^{\si}_{\ga}V^{\ga} ),
\end{equation}
as well as the compatibility condition (curvature relation)
\begin{equation*}
\begin{aligned}
\d_t A_{\al}-\d_{\al} B=&\ \frac{1}{2i}\<\la^{\ga}_{\al}\Im(\d^A_{\ga}\psi \bar{m}-i\la_{\ga\si}V^{\si}\bar{m})+i(\d^{A,\si} \psi-i\la^{\si}_{\ga}V^{\ga} )\Re(\la_{\al\si}\bar{m}),\bar{m}\>\\
=&\ \frac{1}{2}\la_{\al}^{\ga}(\bar{\d}^A_{\ga}\bar{\psi}+i\bar{\la}_{\ga\si}V^{\si})+\frac{1}{2}(\d^{A,\si} \psi-i\la^{\si}_{\ga}V^{\ga} )\bar{\la}_{\al\si}\\
=&\ \frac{1}{2}[\la_{\al}^{\ga}(\bar{\d}^A_{\ga}\bar{\psi}+i\bar{\la}_{\ga\si}V^{\si})+\bar{\la}_{\al}^{\ga}(\d^A_{\ga}\psi-i\la_{\ga\si}V^{\si})]\\
=&\ \Re(\la_{\al}^{\ga}\bar{\d}^A_{\ga}\bar{\psi})-\Im (\la^\ga_\al \bar{\la}_{\ga\si})V^\si.
\end{aligned}
\end{equation*}
which we record for later reference:
\begin{equation}\label{Cpt-A&B}
\d_t A_{\al}-\d_{\al} B = \Re(\la_{\al}^{\ga}\bar{\d}^A_{\ga}\bar{\psi})-\Im (\la^\ga_\al \bar{\la}_{\ga\si})V^\si.
\end{equation}

This in turn allows us to  use the Coulomb gauge condition \eqref{Coulomb} in order to obtain 
an elliptic equation for $B$:

\begin{lemma}[Elliptic equation of $B$] The temporal connection coefficient $B$ solves
	\begin{equation}                \label{Ellip-B}
	\nab^{\ga}\nab_{\ga}B=-\nab^{\ga}[\Re(\la_{\ga}^{\si}\bar{\d}^A_{\si}\bar{\psi})-\Im (\la^\si_\ga \bar{\la}_{\si\be})V^\be]+(2\Im(\psi\bar{\la}^{\be\ga})+\nab^{\be}V^{\ga}+\nab^{\ga}V^{\be})\d_{\be}A_{\ga}.
	\end{equation}
\end{lemma} 
\begin{proof}
Applying $\nab^{\al}$ to (\ref{Cpt-A&B}) yields
\begin{align*}
\nab^{\ga}\nab_{\ga} B=\nab^{\ga}\d_tA_{\ga}-\nab^{\ga}\Re [\la_{\ga}^{\si}(\bar{\d}^A_{\si}\bar{\psi}+i\bar{\la}_{\si\be}V^{\be})].
\end{align*}
By the harmonic coordinates condition (\ref{hm-coord}), (\ref{g_dt}) and the Coulomb gauge condition \eqref{Coulomb} the first term in the right hand side is written as
\begin{align*}
\nab^{\ga}\d_t A_{\ga}=&g^{\be\ga}\nab_{\be}\d_t A_{\ga}=g^{\be\ga}(\d_{\be}\d_tA_{\ga}-\Ga^{\si}_{\be\ga}\d_tA_{\si})=g^{\be\ga}\d_{\be}\d_tA_{\ga}\\
=&\d_t(g^{\be\ga}\d_{\be}A_{\ga})-\d_t g^{\be\ga}\cdot\d_{\be}A_{\ga}\\
=&\d_t\nab^{\ga}A_{\ga}+(2\Im(\psi\bar{\la}^{\be\ga})+\nab^{\be}V^{\ga}+\nab^{\ga}V^{\be})\d_{\be}A_{\ga}\\
=&(2\Im(\psi\bar{\la}^{\be\ga})+\nab^{\be}V^{\ga}+\nab^{\ga}V^{\be})\d_{\be}A_{\ga}.
\end{align*}
We then obtain the $B$-equation.
\end{proof}

Next, we use (\ref{main-eq-abst}) to derive the main equation, i.e. the Schr\"odinger equation for $\psi$. By (\ref{la-commu}), the right-hand side of (\ref{main-eq-abst}) is rewritten as
\begin{align*}
\nab^A_{\al}(\d^{A,\si} \psi -i\la^{\si}_{\ga}V^{\ga} )
=\nab^A_{\al} \d^{A,\si}\psi-i\nab^A_{\ga} \la^{\si}_{\al}V^{\ga}-i\la^{\si}_{\ga}\nab_{\al} V^{\ga}.
\end{align*}
Hence, we have
\begin{equation*}
(\d^{B}_t-V^{\ga}\nab^A_{\ga})\la^{\si}_{\al}+\la^{\ga}_{\al}\Im(\psi\bar{\la}^{\si}_{\ga})+(\la^{\ga}_{\al}\nab_{\ga} V^{\si}-\la^{\si}_{\ga} \nab_{\al} V^{\ga})=i\nab^A_{\al}\nab^{A,\si}\psi,
\end{equation*}
and then contracting this yields
\begin{equation*}    
i(\d^{B}_t-V^{\ga}\nab^A_{\ga})\psi+\nab^A_{\al}\nab^{A,\al}\psi=-i\la^{\ga}_{\si}\Im(\psi\bar{\la}^{\si}_{\ga}).
\end{equation*}
This can be further written as 
\begin{align*}
i(\d_t+iB-V^{\ga}\nab^A_{\ga})\psi+(\nab_{\al}+iA_{\al})(\nab^{\al}+iA^{\al})\psi=-i\la_{\si}^{\ga}\Im(\psi\bar{\la}^{\si}_{\ga}).
\end{align*}
Hence, under the harmonic coordinates condition (\ref{hm-coord}) and the Coulomb gauge 
condition \eqref{Coulomb} we obtain the main Schr\"{o}dinger equation
\begin{equation}\label{main-eq-final}
\begin{aligned}
i\d_t\psi+g^{\al\be}\d_{\al}\d_{\be}\psi=&\ iV^{\ga}\nab^A_{\ga}\psi-2iA_{\al}\nab^{\al}\psi+(B+A_{\al}A^{\al}-i\nab_{\al}A^{\al})\psi-i\la_{\si}^{\ga}\Im(\psi\bar{\la}^{\si}_{\ga})\\
=&\ iV^{\ga}\nab^A_{\ga}\psi-2iA_{\al}\nab^{\al}\psi+(B+A_{\al}A^{\al})\psi-i\la_{\si}^{\ga}\Im(\psi\bar{\la}^{\si}_{\ga}).
\end{aligned}
\end{equation}

In conclusion, under the Coulomb gauge condition  $\nab^{\al}A_{\al}=0$ and the harmonic coordinate condition $g^{\al\be}\Ga^{\ga}_{\al\be}=0$, by (\ref{main-eq-final}), (\ref{la-eq}), (\ref{g-eq-original}), (\ref{Ellp-X}), (\ref{Ellp-A-pre}) and (\ref{Ellip-B}), we obtain the Schr\"{o}dinger equation for the complex mean curvature $\psi$
\begin{equation}        \label{mdf-Shr-sys-2}
\left\{
\begin{aligned}
    & i\d_t\psi+g^{\al\be}\d_{\al}\d_{\be}\psi=i(V-2A)_{\al}\nab^{\al}\psi+(B+A_{\al}A^{\al}-V_{\al}A^{\al})\psi-i\la_{\si}^{\ga}\Im(\psi\bar{\la}^{\si}_{\ga}),
    \\
    & \psi(0) = \psi_0,
    \end{aligned}
\right.    
\end{equation}
where the metric $g$, curvature tensor $\la$, the advection field $V$, connection coefficients $A$ and $B$ are determined at fixed time in an elliptic fashion via the following equations
\begin{equation}           \label{ell-syst}
	\left\{\begin{aligned}
	    &\nab^A_\al \la_{\be\ga}-\nab^A_\be \la_{\al\ga}=0,\quad
	      \nab^{A,\al}\la_{\al\be}=\nab^A_\be \psi,\\
		&\begin{aligned}
		g^{\al\be}\d^2_{\al\be}g_{\ga\si}=&\ [-\d_{\ga} g^{\al\be}\d_{\be} g_{\al\si}-\d_{\si} g^{\al\be}\d_{\be} g_{\al\ga}+\d_{\ga} g_{\al\be}\d_{\si} g^{\al\be}]\\
		&+2g^{\al\be}\Ga_{\si\al,\nu}\Ga^{\nu}_{\be\ga}-2\Re (\la_{\ga\si}\bar{\psi}-\la_{\al\ga}\bar{\la}_{\si}^{\al}),
		\end{aligned}\\
		&\begin{aligned}
	    \nab^{\al}\nab_{\al}V^{\ga}=&\ 2\Im (\nab^{A,\ga}\psi\bar{\psi}-\nab^A_{\al}\psi\bar{\la}^{\al\ga})-\Re (\la^{\ga}_{\si}\bar{\psi}-\la_{\al\si}\bar{\la}^{\al\ga})V^{\si}\\
	    &+2(\Im(\psi\bar{\la}^{\al\be})+\nab^{\al}V^{\be})\Ga^{\ga}_{\al\be},
	    \end{aligned}\\
	    &\nab^{\ga}\nab_{\ga} A_{\al}=\Re(\psi\bar{\la}_{\al}^{\si}-\la_{\al}^{\be}\bar{\la}_{\be}^{\si})A_{\si}+\nab^{\ga}\Im(\la^{\si}_{\ga}\bar{\la}_{\al\si}),\\
		&\begin{aligned}\nab^{\ga}\nab_{\ga}B=&-\nab^{\ga}[\Re(\la_{\ga}^{\si}\bar{\d}^A_{\si}\bar{\psi})-\Im (\la^\si_\ga \bar{\la}_{\si\be})V^\be]\\
		&+(2\Im(\psi\bar{\la}^{\be\ga})+\nab^{\be}V^{\ga}+\nab^{\ga}V^{\be})\d_{\be}A_{\ga}.\end{aligned}
	\end{aligned}\right.
\end{equation}
Fixing the remaining degrees of freedom (i.e. the affine group for the choice of the 
coordinates as well as the time dependence of the $SU(1)$ connection) 
 we can assume that the  following conditions hold at infinity in an averaged sense:
 \begin{equation*}  
\la(\infty)=0,\quad g(\infty) = I_d, \quad V(\infty) = 0,\quad A(\infty) = 0, \quad B(\infty) = 0     
 \end{equation*}
These are needed to insure  the unique solvability of the above elliptic equations
in a suitable class of functions. For the metric $g$ it will be useful to use the representation 
\begin{equation*}
g = I_d + h    
\end{equation*}
so that $h$ vanishes at infinity.

Finally, we note that the above elliptic system \eqref{mdf-Shr-sys-2} is accompanied by
a large family of compatibility conditions as follows:
\begin{enumerate}[label=(\roman*)]
    \item The trace relation \eqref{csmc}. 
    \item The Gauss equations \eqref{R-la} connecting the curvature $R$ of $g$ and $\lambda$.
    \item The symmetry property \eqref{lambda-sim}.
    \item The Ricci equations \eqref{cpt-AiAj-2} for the curvature of $A$.
    \item The Coulomb gauge condition \eqref{Coulomb} for $A$.
    \item The harmonic coordinates condition \eqref{hm-coord} for $g$.
\end{enumerate}
These conditions will all be shown to be satisfied for small solutions to the nonlinear elliptic system \eqref{mdf-Shr-sys-2}.

Now we can restate here the small data local well-posedness result for the (SMCF) system in 
Theorem~\ref{LWP-thm} in terms of the above system:

\begin{thm}[Small data local well-posedness in the good gauge]   \label{LWP-MSS-thm}
	Let $s>\frac{d}{2}$, $d\geq 4$. Then there exists $\ep_0>0$ sufficiently small such that, for all initial data $\psi_0$ with 
\begin{equation*} 
	\lV \psi_0\rV_{H^s}\leq \ep_0,
\end{equation*}
the modified Schr\"odinger system (\ref{mdf-Shr-sys-2}), with $(\la,h,V,A,B)$ determined via the 
elliptic system \eqref{ell-syst}, is locally well-posed in $H^s$ on the time interval $I=[0,1]$. Moreover, the mean curvature satisfies
the bounds
	\begin{equation}\label{psi-full-reg}
	\lV \psi\rV_{l^2 \bX^s} +  \lV (\la,h,V,A,B)\rV_{\bEE^s}\lesssim \lV \psi_0\rV_{H^s}.
	\end{equation}
In addition, the auxiliary functions $(\la,h,V,A,B)$ satisfy 
	the constraints  \eqref{csmc}, \eqref{R-la}, \eqref{lambda-sim}, \eqref{cpt-AiAj-2}, \eqref{Coulomb} and \eqref{hm-coord}.
\end{thm}

Here the solution $\psi$ satisfies in particular the expected bounds
\[
\| \psi \|_{C[0,1;H^s]} \lesssim \|\psi_0\|_{H^s}.
\]
The spaces $l^2 \bX^s$ and  $\bEE^s$, defined in the next section, contain a more complete description of the  full set of variables $\psi,\la,h,V,A,B$, which includes both Sobolev regularity and local energy bounds.

In the above theorem,  by well-posedness we mean a full Hadamard-type well-posedness, including the following properties:
\begin{enumerate}[label=\roman*)]
    \item Existence of solutions $\psi \in C[0,1;H^s]$, with the additional regularity 
    properties \eqref{psi-full-reg}.
    \item Uniqueness in the same class.
    \item Continuous dependence of solutions with respect to the initial data 
    in the strong $H^s$ topology.
    \item Weak Lipschitz dependence  of solutions with respect to the initial data 
    in the weaker $L^2$ topology.
    \item Energy bounds and propagation of higher regularity.
\end{enumerate}

\section{Function spaces and notations}   \label{Sec3}
The goal of this section is to define the function spaces where we aim to solve 
the (SMCF) system  in the good gauge, given by \eqref{mdf-Shr-sys-2}.
Both the spaces and the notation presented in this section are similar to those introduced in \cite{MMT3,MMT4,MMT5}. All the function spaces
described below will be used with respect to harmonic coordinates determined by our gauge choices described in the previous section. We neither attempt nor need to transfer these spaces 
to other coordinate frames.

For a function $u(t,x)$ or $u(x)$, let $\hat{u}=\FF u$ denote the Fourier transform in the spatial variable $x$. Fix a smooth radial function $\varphi:\R^d \rightarrow [0,1] $ supported in $[-2,2]$ and equal to 1 in $[-1,1]$, and for any $i\in \Z$, let
\begin{equation*}
	\varphi_i(x):=\varphi(x/2^i)-\varphi(x/2^{i-1}).
\end{equation*}
We then have the spatial Littlewood-Paley decomposition,
\begin{equation*}
	\sum_{i=-\infty}^{\infty}P_i (D)=1, \quad \sum_{i=0}^{\infty}S_i (D)=1,
\end{equation*}
where $P_i$ localizes to frequency $2^i$ for $i\in \Z$, i.e,
\begin{equation*}
	 \FF(P_iu)=\varphi_i(\xi)\hat{u}(\xi),
\end{equation*}
and 
\[S_0(D)=\sum_{i\leq 0}P_i(D),\quad S_i(D)=P_i(D),\ \text{ for}\ i>0.\]
For simplicity of notation, we set
\[
u_j=S_j u,\quad u_{\leq j}=\sum_{i=0}^j S_i u,\quad u_{\geq j}=\sum_{i=j}^{\infty} S_i u,\quad \text{for }j\geq 0.
\]

For each $j\in\N$, let $\QQ_j$ denote a partition of $\R^d$ into cubes of side length $2^j$, and let $\{\chi_Q\}$ denote an associated partition of unity. For a translation-invariant Sobolev-type space $U$, set $l^p_j U$ to be the Banach space with associated norm 
\begin{equation*}
	\lV u\rV_{l^p_j U}^p=\sum_{Q\in\QQ_j}\lV \chi_Q u\rV_U^p
\end{equation*}
with the obvious modification for $p=\infty$.

Next we define the $l^2\bX^s$ and $l^2N^s$ spaces, which will be used for the primary variable 
$\psi$, respectively for the source term in the Schr\"odinger equation for $\psi$.
Following \cite{MMT3,MMT4,MMT5}, we first define the $X$-norm as
\begin{equation*}
	\lV u\rV_{X}=\sup_{l \in \N} \sup_{Q\in\QQ_l} 2^{-\frac{l}{2}}\lV u\rV_{L^2L^2([0,1]\times Q)}.
\end{equation*}
Here and throughout, $L^pL^q$ represents $L^p_tL^q_x$. To measure the source term, we use an atomic space $N$ satisfying $X=N^{\ast}$. A function $a$ is an atom in $N$ if there is a $j\geq 0$ and a $Q\in \QQ_j$ such that $a$ is supported in $[0,1]\times Q$ and 
\begin{equation*}
	\lV a\rV_{L^2([0,1]\times Q)}\lesssim 2^{-\frac{j}{2}}.
\end{equation*}
Then we define $N$ as linear combinations of the form 
\begin{equation*}
	f=\sum_k c_k a_k,\ \ \sum_k|c_k|<\infty,\ \ a_k\ {\rm atom},
\end{equation*}
with norm 
\begin{equation*}
	\lV f\rV_N=\inf\big\{\sum_k |c_k|: f=\sum_k c_k a_k,\ a_k\ {\rm atoms}\big\}.
\end{equation*}

For solutions which are localized to frequency $2^j$ with $j \geq 0$, we will work in the space
\begin{equation*}
	X_j=2^{-\frac{j}{2}}X\cap L^{\infty}L^2,
\end{equation*} 
with norm 
\begin{equation*}
	\lV u\rV_{X_j}=2^{\frac{j}{2}}\lV u\rV_X+\lV u\rV_{L^{\infty}L^2}.
\end{equation*}
One way to assemble the $X_j$ norms is via the $X^s$ space
\begin{equation*}
\lV u\rV_{X^s}^2=\sum_{j\geq 0} 2^{2js}\lV S_j u\rV_{X_j}^2.
\end{equation*}
But we will also add the $l^p$ spatial summation on the $2^j$ scale to $X_j$,
in order to obtain the space $l^p_j X_j$ with norm
\[
\lV u\rV_{l^p_j X_j} =(\sum_{Q\in \QQ_j} \lV \chi_Q u\rV_{X_j}^p)^{1/p}.
\] 
We then define the space $l^p X^s$ by 
\begin{equation*}
	\lV u\rV_{l^p X^s}^2=\sum_{j\geq 0}2^{2js}\lV S_j u\rV_{l^p_j X_j}^2.
\end{equation*}
For the solutions of Schr\"{o}dinger equation in (\ref{mdf-Shr-sys-2}), we will be working primarily in $l^2 \bX^s$, which is defined by
\[
\lV u\rV_{l^2\bX^s}=\lV u\rV_{l^2X^s}+\lV \d_t u\rV_{L^2H^{s-2}}.
\]
We note that the second component, introduced here for the first time, serves the purpose 
of providing better bounds at low frequencies $j \leq 0$.

We analogously define 
\begin{equation*}
	N_j=2^{\frac{j}{2}}N+L^1L^2,
\end{equation*}
which has norm 
\begin{equation*}
	\lV f\rV_{N_j}=\inf_{f=2^{\frac{j}{2}}f_1+f_2} \big(\lV f_1\rV_N+\lV f_2\rV_{L^1L^2}\big),
\end{equation*}
and 
\begin{equation*}
	\lV f\rV_{l^pN^s}^2=\sum_{j\geq 0}2^{2js}\lV S_j f\rV_{l^p_j N_j}^2.
\end{equation*}
Here we shall be working primarily with $l^2N^s$.

We also note that for any $j\in\N$, we have
\begin{equation*}
	\sup_{Q\in\QQ_j} 2^{-\frac{j}{2}}\lV u\rV_{L^2L^2([0,1]\times Q)}\leq \lV u\rV_{X},
\end{equation*}
hence
\begin{equation*}
	\lV u\rV_N\lesssim 2^{j/2}\lV u\rV_{l^1_jL^2L^2}.
\end{equation*}
This bound will come in handy at several places later on.

For the elliptic system \eqref{ell-syst}, at a fixed time we define the $\mathcal{H}^s$ norm,
\begin{equation*}
    \lV (\la,h,V,A,B)\rV_{\mathcal H^s}=\lV\la\rV_{H^{s}}+\lV|D|h\rV_{H^{s+1}} +\lV|D|V\rV_{H^{s}}+\lV|D|A\rV_{H^{s}}+\lV|D|B\rV_{H^{s-1}}.
\end{equation*}

In addition to the fixed time norms, for the study of the Schr\"odinger
equation for $\psi$ we will also need to bound time dependent norms $\EE^s$ and $\bEE^s$ for the elliptic system \eqref{ell-syst}, in terms of similar norms for $\psi$. 
For simplicity of notation, we define 
\begin{align*}
\lV u\rV_{Z^{\si,s}}=\lV |D|^{\si}S_0 u\rV_{l^2_0L^{\infty}L^2}^2+\sum_{j>0} 2^{2sj} \lV S_ju\rV_{l^2_jL^{\infty}L^2}^2.
\end{align*}
Then the $\bZ^{\si,s}$ spaces are defined by
\begin{equation*}
\lV u\rV_{\bZ^{\si,s}}=\lV u\rV_{Z^{\si,s}}+\lV|D|^{\si}\d_t u\rV_{L^2H^{s-\si-2}}.
\end{equation*}
For the $\lambda$, $V$, $A$ and $B$-equations in (\ref{ell-syst}), we will be working primarily in $\bZ^{0,s}$, $\bZ^{1,s+1}$, $\bZ^{1,s+1}$ and $\bZ^{1,s}$, respectively. 

On the other hand, for the metric component $h=g-I_d$ we need to introduce some additional structure
which is associated to spatial scales larger than the frequency.
Precisely, to measure the portion of $h$ which is localized to frequency $2^j$, $j\in\Z$, we decompose $P_j h$ as an atomic summation of components $h_{j,l}$ associated to spatial scales $2^l$
with $l\geq |j|$, where $h_{j,l}$ still localizes to frequency $2^j$, i.e.,
\begin{equation*}
P_jh=\sum_{l\geq |j|}h_{j,l}.
\end{equation*}  
Then we define the $Y_j$-norm by
\begin{equation*}
\lV P_j h \rV_{Y_j} =\inf_{P_jh=\sum_{l\geq |j|}h_{j,l}} \sum_{l\geq |j|} 2^{l-|j|} \lV h_{j,l}\rV_{l^1_lL^{\infty}L^2}.
\end{equation*}
Assembling together the dyadic pieces in an $l^2$ Besov fashion, we obtain  the $Y^{\si,s}$
space with norm given by
\begin{gather*}
\lV h\rV_{Y^{\si,s}}^2=\sum_{j\leq \Z} 2^{2(\si j^-+sj^+)}\lV P_jh\rV_{Y_j}^2.
\end{gather*}
Then for $h$-equation in (\ref{mdf-Shr-sys-2}), we will be working primarily in $\bY^{s+2}$, whose norm is defined by
\begin{align*}
\lV h\rV_{\bY^{s+2}}=\lV h\rV_{Y^{s+2}}+\lV \nab \d_t h\rV_{L^2H^{s-1}}=&\lV h\rV_{Y^{\frac{d}{2}-1-\de,s+2}}+\lV h\rV_{\bZ^{1,s+2}},
\end{align*}
where the space $Y^s=Y^{\frac{d}{2}-1-\de,s}\cap Z^{1,s}$.
Collecting all the components defined above, for the elliptic system \eqref{ell-syst}, we define the $\EE^s$ norm as
\begin{equation*}
    \lV (\la,h,V,A,B)\rV_{\EE^s}=\lV \la\rV_{Z^{0,s}}+\lV h\rV_{Y^{s+2}} +\lV V\rV_{Z^{1,s+1}}+\lV A\rV_{Z^{1,s+1}}+\lV B\rV_{Z^{1,s}},
\end{equation*}
and the $\bEE^s$ norm as
\begin{equation*}
    \lV (\la,h,V,A,B)\rV_{\bEE^s}=\lV (\la,h,V,A,B)\rV_{\EE^s}+\lV \d_t(\la,h,V,A,B)\rV_{L^2\HH^{s-2}}.
\end{equation*}

Since we often use Littlewood-Paley decompositions, the next lemma is a convenient tool 
to see that our function spaces are invariant under the action of some standard classes of multipliers:

\begin{lemma}
	For any Schwartz function $f\in\SS$, multiplier $m(D)$ with $\lV \FF^{-1}(m(\xi))\rV_{L^1}<\infty$, and translation-invariant Sobolev-type space $U$, we have
	\begin{equation*}
		\lV m(D)f\rV_{U}\lesssim \lV \FF^{-1}(m(\xi))\rV_{L^1}\lV f\rV_U.
	\end{equation*}
\end{lemma}

We will also need the following Bernstein-type inequality:

\begin{lemma}[Bernstein-type inequality]
	For any $j,k\in\Z$ with $j+k\geq 0$, $1\leq r<\infty$ and $1\leq q\leq p\leq \infty$, we have
	\begin{align}     \label{Bern-ineq}
		&\lV P_k f\rV_{l^r_jL^p}\lesssim 2^{kd(\frac{1}{q}-\frac{1}{p})}\lV P_k f\rV_{l^r_jL^q},\\    \label{Convo-est}
		&\lV \<x\>^{\al-d}\ast f_{\leq 0}\rV_{l_0^pL^{\infty}L^2}\lesssim \lV f_{\leq 0}\rV_{l^1_0L^{\infty}L^2},\ \ {\rm for }\ p>\frac{d}{d-\al}.
	\end{align}
\end{lemma}
\begin{proof}
	We begin with the Bernstein-type inequality (\ref{Bern-ineq}). Using the properies of the Fourier transform, $P_kf$ is rewritten as
	\begin{equation*}     
		P_k f=\int_{\R^d} (\FF^{-1}{\varphi}_k) (x-y)P_kf(y)dy=2^{kd}\int_{\R^d} K(2^k(x-y))S_k f(y)dy, 
	\end{equation*}
	where $K(x)=\FF^{-1}\varphi(x)$. Then 
	\begin{align*}
		\lV P_kf\rV_{l^r_jL^p}=& \ 2^{kd}(\sum_{Q\in\QQ_j}\lV \chi_Q(x)\int_{\R^d} K(2^k(x-y))P_k f(y)dy\rV_{L^p}^r)^{1/r}\\
		\leq & \ 2^{kd}(\sum_{Q\in\QQ_j}\lV \chi_Q(x)\int_{\R^d} K(2^k(x-y)) \mathbf{1}_{<M}(2^k(x-y)) P_k f(y)dy\rV_{L^p}^r)^{1/r}\\
		&+2^{kd} \lV K(2^kx)\mathbf{1}_{>M}(2^kx)\ast P_k f\rV_{l^r_j L^p} \\
		:=&\ I+II,
	\end{align*}
	where $d(Q,\tilde{Q})=\inf\{|x-y|:x\in Q,y\in\tilde{Q}\}$ and $M$ is a large constant. Since $j+k\geq 0$, for any fixed $Q\in\QQ_j$ there are only finite many $\tilde{Q}\in\QQ_j$ such that $d(Q,\tilde{Q})\leq 2^{-k} M$. Then from Young's inequality we can bound $I$ by
	\begin{align*}
		I\lesssim 2^{kd}(\sum_{Q\in\QQ_j} \sum_{d(Q,\tilde{Q})\leq 2^{-k}M,\tilde{Q}\in\QQ_j}\lV K(2^kx)\rV_{L^{\tilde{q}}}^r\lV \chi_{\tilde{Q}}P_k f\rV_{L^q}^r)^{1/r}
		\lesssim  2^{kd(\frac{1}{q}-\frac{1}{p})}\lV P_kf\rV_{l^r_jL^q}.
	\end{align*}

	On the other hand, since $|K(x)|\lesssim \<x\>^{-N}$ for any large $N$, for $II$ we have
	\begin{align*}
		II\lesssim & \ 2^{k(d-N)} \lV|2^kx|^{-N}\mathbf{1}_{>M}(2^kx)\rV_{L^1}\lV S_k f\rV_{l^r_j L^p}\\
		\lesssim & \ M^{-N+d} \lV P_k f\rV_{l^r_jL^q},
	\end{align*}
	which can be absorbed by the term on the left. These imply the bound (\ref{Bern-ineq}).
	\medskip
	
	Next, we prove the estimate (\ref{Convo-est}). The left hand side of (\ref{Convo-est}) is decomposed as 
	\begin{align*}
		\lV \<x\>^{\al-d}\ast f_{\leq 0}\rV_{l_0^pL^{\infty}L^2}^p
		\lesssim & \  \sum_{Q\in\QQ_0}\lV \chi_Q(x)\int \<y\>^{\al-d}f_{\leq 0}(x-y)dy\rV_{L^{\infty}L^{\infty}}^p\\
		\lesssim & \sum_{Q\in\QQ_0}\lV \chi_Q(x)\int_{|y|\leq 1} \<y\>^{\al-d}f_{\leq 0}(x-y)dy\rV_{L^{\infty}L^{\infty}}^p\\
		&+\sum_{Q\in\QQ_0}\lV \chi_Q(x)\int_{|y|>1} \<y\>^{\al-d}\sum_{\tilde{Q}\in\QQ_0}\chi_{\tilde{Q}}f_{\leq 0}(x-y)dy\rV_{L^{\infty}L^{\infty}}^p\\
		=&I_1^p +I_2^p.
	\end{align*}
	Then by (\ref{Bern-ineq}) we bound $I_1$ by
	\begin{equation*}
	I_1\lesssim \lV f_{\leq 0}\rV_{l^p_0L^{\infty}L^{\infty}}\lesssim \lV f_{\leq 0}\rV_{l^p_0L^{\infty}L^2}.
	\end{equation*}
    On the other hand, by H\"{o}lder's inequality  and (\ref{Bern-ineq}), we bound $I_2$ by
	\begin{align*}
	I_2\lesssim & \sum_{\tilde{Q}\in\QQ_0}(\sum_{Q\in\QQ_0}\lV \chi_Q(x)\int_{|y|>1} \<y\>^{\al-d}\chi_{\tilde{Q}}f_{\leq 0}(x-y)dy\rV_{L^{\infty}L^{\infty}}^p)^{1/p}\\
	\lesssim &  \sum_{\tilde{Q}\in\QQ_0}(\sum_{Q\in\QQ_0}\int_{|y|>1}\chi_Q\<y\>^{(\al-d)p}dy \lV \chi_{\tilde{Q}}f_{\leq 0}\rV_{L^{\infty}L^q}^p)^{1/p}\\
	\lesssim &  \lV f_{\leq 0}\rV_{l^1_0L^{\infty}L^q}(\int \<y\>^{(\al-d)p}dy)^{1/p}\\
	\lesssim & \lV f_{\leq 0}\rV_{l^1_0L^{\infty}L^2},
	\end{align*}
	which gives the bound (\ref{Convo-est}), and thus completes the proof of the lemma.
\end{proof}

Finally, we define the frequency envelopes as in \cite{MMT3,MMT4,MMT5} which will be used in multilinear estimates. Consider a Sobolev-type space $U$ for which we have
\begin{equation*}
\lV u\rV_U^2=\sum_{k=0}^{\infty} \lV S_k u\rV_U^2.
\end{equation*}
A frequency envelope for a function $u\in U$ is a positive $l^2$-sequence, $\{ a_j\}$, with 
\begin{equation*}
\lV S_j u\rV_U\leq a_j.
\end{equation*}
We shall only permit slowly varying frequency envelopes. Thus, we require $a_0\approx \lV u\rV_U$ and 
\begin{equation*}
a_j\leq 2^{\delta|j-k|} a_k,\quad j,k\geq 0,\ 0<\delta\ll 1.
\end{equation*}
The constant $\delta$ shall be chosen later and only depends on $s$ and the dimension $d$. Such frequency envelopes always exist. For example, one may choose
\begin{equation}         \label{Freq-envelope}
a_j=2^{-\de j}\lV u\rV_U+\max_k 2^{-\de|j-k|} \lV S_k u\rV_U.
\end{equation}

\section{Elliptic estimates} \label{Sec-Ell}\

Here we consider the solvability of the elliptic system \eqref{ell-syst}, together with the constraints  \eqref{lambda-sim}, \eqref{hm-coord} and \eqref{Coulomb}. We will do this in two steps. First  we prove that this system is solvable in Sobolev spaces at fixed time. Then we prove  space-time bounds in local energy spaces; the latter will be needed in the study of the Schr\"odinger evolution \eqref{mdf-Shr-sys-2}. 

For simplicity of notations, we define the set of elliptic variables by
\[
\SS=(\la,h,V,A,B),
\] 
Later when we compare two solutions for \eqref{ell-syst}, we will denote the differences 
of two solutions or the linearized variable by
\[
\dSS=(\dla,\dh,\dV,\dA,\dB).
\]
Our fixed time result is as follows:
\begin{thm}\label{t:ell-fixed-time}
a) Assume that $\psi$ is small in $H^s$ for $s > d/2$ and $d \geq 4$. Then the elliptic system \eqref{ell-syst} admits a unique small solution $\SS= (\la,h,V,A,B)$ in $\mathcal H^s$,
with
\begin{equation}   \label{t:ell-fixed-time: bd1}
\|\SS\|_{\mathcal H^s} \lesssim \| \psi\|_{H^s}.    
\end{equation}
In addition 
this solution has a smooth dependence on $\psi$ in $H^s$ and  satisfies the constraints 
\eqref{csmc}, \eqref{R-la},  \eqref{lambda-sim}, \eqref{cpt-AiAj-2}, \eqref{hm-coord} and \eqref{Coulomb}. 

b) Let $\psi$ and $(\la, h,V,A,B)= \SS(\psi)$ be as above. Then for the linearization of the solution map above we also have the bound:
\begin{equation} \label{t:ell-fixed-time: dlt}
\|D\SS(\delta \psi)\|_{\mathcal H^{\si}} \lesssim \| \dpsi\|_{H^{\si}}, \qquad  \si\in (d/2-3,s]. 
\end{equation}
Moreover, assume that $p_k$ is an admissible frequency envelope for $\dpsi$ in $H^{\si}$. 
Then we have 
\begin{equation}    \label{t:ell-fixed-time: fre-env}
    \| S_k\dSS \|_{\HH^{\si}}\lesssim p_k.
\end{equation}

c) We also have a similar bound for the Hessian of the solution map,
\begin{equation} \label{t:ell-fixed-time: ddlt}
\|D^2\SS(\delta_1 \psi,\delta_2 \psi)\|_{\mathcal H^{\si}} 
\lesssim  \| \delta_1 \psi\|_{H^{\si_1}} \| \delta_2 \psi\|_{H^{\si_2}},  
\end{equation}
with $\si,\si_1,\si_2\in (d/2-3,s], \si_1+\si_2 = \si+s$.
\end{thm}

\begin{rem}
    Here we solve the elliptic system \eqref{ell-syst} in the function space $\HH^s$ for $s>d/2$, which is more suitable for the nonlinear estimates of $\psi$-equation. 
    Nevertheless, this system can be solved in a similar fashion for the full range  
    of indices $s$ above scaling, namely $s>d/2-1$. However, in the additional range  
    $d/2-1 < s \leq d/2$ one needs to replace the  above solution space $\mathcal H^s$ with a slightly larger one,
    \[
    \lV \SS\rV_{\tilde{\mathcal H}^s}=\lV\la\rV_{H^{s}}+\lV|D|h\rV_{H^{\sigma+1}} +\lV|D|V\rV_{H^{\sigma}}+\lV|D|A\rV_{H^{\sigma}}+\lV|D|B\rV_{H^{\sigma-1}},
    \]
    where $\sigma = 2s-d/2+1$. Then the elliptic system \eqref{ell-syst} admits a unique small solution $\SS$ in $\tilde{\mathcal H}^s$ with
    \[
    \lV \SS\rV_{\tilde{\mathcal H}^s}\lesssim \lV \psi\rV_{H^s}.
    \]
\end{rem}

\begin{proof}[Proof of Theorem \ref{t:ell-fixed-time}]

a)  The proof is based on a perturbative argument.	We rewrite the system \eqref{ell-syst} in the form
\begin{equation}    \label{ell-syst_2}
	\left\{\begin{aligned}
	&\d_\al \la_{\al\be}=\d_\be \psi+H_{1\la},\\
	&\d_\al \la_{\be\ga}-\d_\be \la_{\al\ga}=H_{2\la},\\
	&\De g_{\ga\si}=H_g,\\
	&\De V^{\ga}=H_V,\\
	&\De A_{\al}=H_A,\\
	&\De B=H_B,
	\end{aligned}\right.
\end{equation}	
where the nonlinear source terms are given by
	\begin{align*}
	&H_{1\la}= iA_\be \psi-h^{\al\mu}\d_\mu \la_{\al\be}+\Ga_{\al\be,\si}\la^{\al\si},\\ 
	&H_{2\la}=  -iA_\al \la_{\be\ga}+iA_\be \la_{\al\ga}+\Ga_{\al\ga,\si}\la_\be^\si -\Ga_{\be\ga,\si}\la_\al^\si,\\
	&\begin{aligned}
	H_g=&-h^{\al\be}\d^2_{\al\be}g_{\ga\si}-\d_{\ga} g^{\al\be}\d_{\be} g_{\al\si}-\d_{\si} g^{\al\be}\d_{\be} g_{\al\ga}+\d_{\ga} g_{\al\be}\d_{\si} g^{\al\be}\\
	&+2g^{\al\be}\Ga_{\si\al,\nu}\Ga^{\nu}_{\be\ga}-2\Re (\la_{\ga\si}\bar{\psi}-\la_{\al\ga}\bar{\la}_{\si}^{\al}),
	\end{aligned}\\
	&\begin{aligned}
	H_V=&-\nab^{\al}\nab_{\al}V^{\ga}+\De V^{\ga}+2\Im (\nab^{A,\ga}\psi\bar{\psi}-\nab^A_{\al}\psi\bar{\la}^{\al\ga})-\Re (\la^{\ga}_{\si}\bar{\psi}-\la_{\al\si}\bar{\la}^{\al\ga})V^{\si}\\
	&+2(\Im(\psi\bar{\la}^{\al\be})+\nab^{\al}V^{\be})\Ga^{\ga}_{\al\be},
	\end{aligned}\\
	&H_A=-\nab^{\ga}\nab_{\ga} A_{\al}+\De A_{\al}+\Re(\psi\bar{\la}_{\al}^{\si}-\la_{\al}^{\be}\bar{\la}_{\be}^{\si})A_{\si}+\nab^{\ga}\Im(\la^{\si}_{\ga}\bar{\la}_{\al\si}),\\
	&
	\begin{aligned}
	H_B=&-\nab^{\ga}\nab_{\ga}B+\De B-\nab^{\ga}\Re [\la_{\ga}^{\si}(\bar{\d}^A_{\si}\bar{\psi}+i\bar{\la}_{\si\be}V^{\be})]\\
	&+(2\Im(\psi\bar{\la}^{\be\ga})+\nab^{\be}V^{\ga}+\nab^{\ga}V^{\be})\d_{\be}A_{\ga}.
	\end{aligned}
	\end{align*}

	In order to prove the existence of solutions to \eqref{ell-syst_2} at a fixed time for small $\psi\in H^s$, we construct solutions to \eqref{ell-syst_2} iteratively.
	We define the sets of elliptic variables 
	\[
	\SS^{(n)}=(\la^{(n)},h^{(n)},V^{(n)},A^{(n)},B^{(n)}),
	\]
	at each step, based on the scheme
	\begin{equation}        \label{ell-iteration}
	\left\{\begin{aligned}
	&\d_\al \la^{(n+1)}_{\al\be}=\d_\be \psi+H^{(n)}_{1\la},\\
	&\d_\al \la^{(n+1)}_{\be\ga}-\d_\be \la^{(n+1)}_{\al\ga}=H^{(n)}_{2\la},\\
	&\De g^{(n+1)}_{\ga\si}=H^{(n)}_g,\\
	&\De V^{(n+1)\ga}=H^{(n)}_V,\\
	&\De A^{(n+1)}_{\al}=H^{(n)}_A,\\
	&\De B^{(n+1)}=H^{(n)}_B,
	\end{aligned}\right.
	\end{equation}
	with the trivial initialization
	\begin{equation*}
	\SS^{(0)}=(0,0,0,0,0), \quad g^{(0)}=h^{(0)}+I,
	\end{equation*}
	where $H^{(n)}_{1\la}$, $H^{(n)}_{2\la}$, $H^{(n)}_{g}$, $H^{(n)}_{V}$, $H^{(n)}_{A}$ and $H^{(n)}_{B}$ are defined as $H_{1\la}$, $H_{2\la}$, $H_{g}$, $H_{V}$, $H_{A}$ and $H_{B}$ with
	\[
	\SS=\SS^{(n)}.
	\]
	
	We will inductively show that 
    \[
	\lV \SS^{(n)}\rV_{\HH^s}\leq C\lV\psi\rV_{H^s},
    \]
    with a large universal constant $C$. This trivially holds for our initialization. Then using a standard Littlewood-Paley decomposition, Bernstein's inequality and 
    the smallness of our data $\psi\in H^s$ in order to estimate the source terms 
    $H^{(n)}_{1\la}$, $H^{(n)}_{2\la}$, $H^{(n)}_{g}$, $H^{(n)}_{V}$, $H^{(n)}_{A}$ and $ H^{(n)}_{B}$,
    we obtain
	\begin{align*}
	\lV \SS^{(n+1)}\rV_{\HH^s}
	\lesssim \lV\psi\rV_{H^s}+\lV\SS^{(n)}\rV_{\HH^s}^2(1+\lV\SS^{(n)}\rV_{\HH^s})^N
	\lesssim \lV\psi\rV_{H^s}.
	\end{align*}
	From the iterative scheme \eqref{ell-iteration} and $\psi\in H^s$ small,
	we can repeat the same analysis for successive differences in  order to obtain
	a small Lipschitz constant,
	\begin{align*}
	\lV \SS^{(n+1)}-\SS^{(n)}\rV_{\HH^s}\ll\lV\SS^{(n)}-\SS^{(n-1)}\rV_{\HH^s}.
	\end{align*}
	Hence the elliptic system \eqref{ell-syst} admits a small solution 
	\[
	\SS=\lim_{n\rightarrow \infty} \SS^{(n)}\in \HH^s.
	\]
	The uniqueness and the Lipschitz dependence of the 
	solution on $\psi$ are easily obtained by similar elliptic estimates. 
\bigskip 

Next, we prove the solution satisfies the constraints \eqref{csmc}, \eqref{lambda-sim}, \eqref{cpt-AiAj-2}, \eqref{Coulomb}, \eqref{hm-coord} and \eqref{R-la}. To get started, let us summarize the compatibility conditions we need to verify:
\begin{align*}
\psi = g^{\alpha \beta} \lambda_{\alpha\beta},
\\
\lambda_{\al\beta} = \lambda_{\be\alpha }, 
\\
\nabla_\alpha A_\beta - \nabla_\beta A_\alpha = \Im( \lambda_\alpha^{\ \gamma} \bar \lambda_{\ga\beta}),\\
\nabla^\alpha A_\alpha = 0,\\
g^{\alpha\beta} \Gamma_{\alpha\beta,\delta} = 0,
\\
\Ric_{\ga\be}=\Re(\lambda_{\ga\be}\bar{\psi}-\la_{\ga\al}\bar{\la}_{\ \be}^\al),\\
R_{\si\ga\al\be}=
    \Re(\lambda_{\ga\be}\bar{\la}_{\si\al}-\la_{\ga\al}\bar{\la}_{\si\be}).
\end{align*}
We need to show that these constraints are satisfied for solutions to
the elliptic system \eqref{ell-syst}. We can disregard the $B$ and $V$
equations, which are unneeded here.

To shorten the notations, we define 
\begin{equation*}  
\begin{aligned}
    &C^1=\psi- g^{\alpha \beta} \lambda_{\alpha\beta},\\
    &C^2_{\al\be}=\lambda_{\al\beta} -\lambda_{\be\alpha},\\
    &C^3_{\al\be}=\nabla_\alpha A_\beta - \nabla_\beta A_\alpha - \Im( \lambda_{\alpha\gamma} \bar \lambda^\ga_{\ \beta }),\\
    &C^4=\nabla^\alpha A_\alpha,\\
    &C^5_\de=g^{\alpha\beta} \Gamma_{\alpha\beta,\delta},\\
    &C^6_{\ga\be}=\Ric_{\ga\be}-\Re(\lambda_{\ga\be}\bar{\psi}-\la_{\ga\al}{\bar{\la}^\al}_{\ \be}), \\
    &C^7_{\si\ga\al\be}=R_{\si\ga\al\be}-
    \Re(\lambda_{\ga\be}\bar{\la}_{\si\al}-\la_{\ga\al}\bar{\la}_{\si\be}).
\end{aligned}
\end{equation*}
Here $C^2$ and $C^3$ are antisymmetric, $C^6$ is symmetric and $C^7$ inherits all the linear symmetries of the 
curvature tensor.

Our goal is to show that all these functions vanish. We will prove this by showing 
that they solve a coupled linear homogeneous elliptic system of the form
\begin{equation*} 
\begin{aligned}
\nab^A_\be C^1 =&\ \nab^{A,\al}C^2_{\al\be}, \\
\Delta^A_g C^2 =&\ (\la+\psi)(C^3+C^6+C^7)+(\la^2+\la\psi)C^2, \\
\Delta_g C^3 = &\  R C^3+ \nabla(  C^6 A) + \nab( \la \nab C^2+\nab \la C^2)  ,\\
\Delta_g C^4 = &\  \nab( C^6 A) + R C^3 +\nab^2(\la C^2), \\
\Delta_g C^5 = &\ R C^5 + \nabla ( C^1 \psi)+\la \nab C^2+\nab \la C^2  ,\\
C^6_{\ga\si}=&\  \frac{1}{2}(\nab_\ga C^5_\si+\nab_\si C^5_\ga), \\
\nab_\delta C^7_{\si\ga\al\be}& \ + \nab_{\si} C^7_{\ga\de\al\be} 
 + \nab_{\ga} C^7_{\de\si\al\be} =0, \\
\nab^\si C^7_{\si\ga\al\be} =& \  \nabla_{\al} C^6_{\ga\be} - \nabla_{\be} C^6_{\ga\al}
+ \nab(\la C^1+\la C^2).
\end{aligned}
\end{equation*}
Here the covariant Laplace operators $\Delta_g$, respectively $\Delta_g^A$ are 
symmetric and coercive in $\dot H^1$. We consider these equations as a system in the space
\[
(C^1,C^2,C^3,C^4,C^5,C^6,C^7) \in \dot H^1 \times \dot H^1 \times \dot H^1 \times \dot H^1 
\times \dot H^1 \times L^2 \times L^2
\]
using $\dot H^1$ bounds for the Laplace operator in the second to fifth equations, and 
interpreting the last two equations as an elliptic div-curl system in $L^2$, with an $\dot H^{-1}$ source term.
Since the coefficients are all small, the right hand side terms are perturbative
and $0$ is the unique solution for this system. The details are left for the reader, as they  only involve Sobolev embeddings and H\"older's inequality.

To complete the argument, we now successively derive the equations in the above system.
In the computations below, it is convenient to introduce several auxiliary notations.
The curvature of the connection $A$ acting on complex valued functions is denoted by
\[
\bmF_{\alpha \beta} = \partial_\alpha A_\beta - \partial_\beta A_\alpha
\]
so that we have
\[
[\nabla^A_\alpha,\nabla^A_\beta]\psi = i \bmF_{\alpha \beta} \psi.
\]
We also set 
\[
C^7_{\si\ga\al\be}=R_{\si\ga\al\be}-\tR_{\si\ga\al\be},\qquad \tR_{\si\ga\al\be}:= \Re(\lambda_{\ga\be}\bar{\la}_{\si\al}-\la_{\ga\al}\bar{\la}_{\si\be}), 
\]
respectively
\[
C^6_{\ga\be}=\Ric_{\ga\be}-\tRic_{\ga\be},\qquad   \tRic_{\ga\be}:=\Re(\lambda_{\ga\be}\bar{\psi}-\la_{\ga\al}\bar{\la}^\al_{\ \be}),\quad \tR:=g^{\ga\be}\tRic_{\ga\be},
\]
and
\[
C^3_{\alpha\beta} = \bmF_{\alpha\beta} - \tbmF_{\alpha\beta}, \qquad  \tbmF_{\alpha\beta} := \Im( \lambda_{\alpha\gamma} \bar \lambda^\ga_{\ \beta }).
\]

\bigskip

\begin{proof}[The equation for $C^1$] This equation has the exact form
\begin{equation*}   
\nab^A_\be C^1 = \nab^{A,\al}C^2_{\al\be}.
\end{equation*}
This is obtained by \eqref{la-eq} directly.
\end{proof}

\medskip

\begin{proof}[The equation for $C^2$] The full system for $C^2$ has the form
	\begin{equation}       \label{CompCond-2_eq}
	\begin{aligned}
	\Delta^A_g C^2_{\al\be} =&(\la+\psi)(C^3+C^6+C^7)+(\la^2+\la\psi)C^2.
	\end{aligned}
	\end{equation}
	By $\la$-equation \eqref{la-eq} we have
	\begin{align*}
	    \nab^{A,\ga}\nab^A_\ga \la_{\al\be}=& [\nab^{A,\ga},\nab^A_\al]\la_{\ga\be}+\nab^A_\al \nab^A_\be \psi\\
	    =&\Ric_{\al\mu}{\la^\mu}_\be+R_{\si\al\be\mu}\la^{\si\mu}+iC^3_{\ga\al}{\la^\ga}_\be+i\Im({\la_\ga}^\mu\bar{\la}_{\mu\al}){\la^\ga}_\be+\nab^A_\al \nab^A_\be \psi.
	\end{align*}
	Then we use $C^6$, $C^7$ and $C^3$ to give
	\begin{align*}
	    \De_g^A C^2_{\al\be}=&C^6_{\al\mu}{\la^\mu}_\be-C^6_{\be\mu}{\la^\mu}_\al+C^7_{\si\al\be\mu}\la^{\si\mu}-C^7_{\si\be\al\mu}\la^{\si\mu}\\
	    &+iC^3_{\ga\al}{\la^\ga}_\be-iC^3_{\ga\be}{\la^\ga}_\al+iC^3_{\al\be}\psi+C^2(\la^2+\la\psi).
	\end{align*}
	Hence, the $C^2$-equation \eqref{CompCond-2_eq} follows.
\end{proof}

\medskip

\begin{proof}[The equation for $C^3$]  
This has the form 
\begin{equation}   \label{CompCond-3_eq}
    \begin{aligned} 
  \Delta_g C^3_{\al\be} =& \
 \nabla_\beta ( C^6_{\alpha\delta} A^\delta) - \nabla_\alpha (C^6_{\beta\delta}  A^\delta)  
 + R_{\beta\alpha\sigma\delta} C^{3,\sigma\delta} 
 + \Ric_{\alpha \delta} C^{3,\delta\beta} -  \Ric_{\beta \delta} C^{3,\delta \alpha}\\
 &\ + \nabla^\ga \Im(\la_\ga^{\ \si}(\overline{\nab^A_\al C^2_{\si\be}}-\overline{\nab^A_\be C^2_{\si\al}})+\nab^A_\ga \la^\si_{\ \be}\bar{C^2}_{\al\si}).
\end{aligned}
\end{equation}
To prove this, it is convenient to separate the left hand side into two terms,
\[
 \Delta_g C^3_{\al\be} = ( [\Delta_g, \nabla_\alpha ] A_\beta - [\Delta_g, \nabla_\beta ] A_\alpha)
 + (\nabla_\alpha \Delta_g A_\beta - \nabla_\beta \Delta_g A_\alpha - \Delta_g \tbmF_{\alpha\beta}) := I + II .
\]
For the commutator we use the Bianchi identities to compute
\begin{align*}
 I  = &\ [\nabla^\sigma \nabla_\sigma,\nabla_{\alpha}] A_\beta - [ \nabla^\sigma \nabla_\sigma,   \nabla_\beta] A_\alpha  \\
 = & \ \nabla^\sigma (R_{\sigma \alpha \beta\delta} A^\delta - R_{\sigma \beta \alpha\delta} A^\delta)
 + (R_{\sigma \alpha \beta\delta} - R_{\sigma \beta \alpha\delta}) \nabla^\sigma A^\delta +
 {R^\sigma}_{\alpha \sigma \delta} \nabla^\delta A_\beta -{R^\sigma}_{\beta \sigma \delta} \nabla^\delta A_\alpha
 \\
 = & \ \nabla^\sigma R_{\beta\alpha\sigma\delta} A^\delta + 2 R_{\beta\alpha\sigma\delta} \nabla^\sigma A^\delta
 + \Ric_{\alpha \delta} \nabla^{\delta} A_\beta -  \Ric_{\beta \delta} \nabla^{\delta} A_\alpha
 \\
 = & \ (\nabla_\beta \Ric_{\alpha\delta} - \nabla_\alpha R_{\beta\delta} ) A^\delta  
 + R_{\beta\alpha\sigma\delta} \bmF^{\sigma\delta}+ \Ric_{\alpha \delta} ({\bmF^\delta}_\beta+\nabla_{\beta} A^\delta) -  \Ric_{\beta \delta} ({\bmF^\delta}_\alpha+ \nabla_{\alpha} A^\delta)
 \\
 = & \ \nabla_\beta ( \Ric_{\alpha\delta} A^\delta) - \nabla_\alpha (\Ric_{\beta\delta}  A^\delta)  
 + R_{\beta\alpha\sigma\delta} \bmF^{\sigma\delta}+ \Ric_{\alpha \delta} {\bmF^\delta}_\beta-  \Ric_{\beta \delta} {\bmF^\delta}_\alpha.
 \end{align*}
On the other hand for the second term we use the $A$ equation in \eqref{ell-syst} to write
 \begin{align*}
        II=&\  \nab_{\al}[\tRic_{\beta\si}A^{\si}]-\nab_{\be}[\tRic_{\alpha\si} A^{\si}]
        \\ & \
        +\nab_{\al}\nab^{\ga}\Im(\la_{\ga\si}\bar{\la}^\si_{\ \be})-\nab_{\be}\nab^{\ga}\Im(\la_{\ga\si}\bar{\la}^\si_{\ \al})-\nab_{\ga}\nab^{\ga}\Im(\la_{\al\si}\bar{\la}^\si_{\ \be})\\
        :=& II_1+II_2.
\end{align*}
The first term $II_1$ combines directly with the first two terms in $I$. For the second we
commute
\begin{align*}
II_2 =& \  R_{\alpha \gamma \gamma \delta} \tbmF^{\delta}_\beta +
R_{\alpha \gamma \beta \delta} \tbmF^\delta_{\gamma} -  R_{\beta \gamma \gamma \delta} \tbmF^{\delta}_\alpha -
R_{\beta \gamma \alpha \delta} \tbmF^\delta_{\gamma} + \\ & \ + \nab^{\ga}(\nab_{\al}\Im(\la_{\ga\si}\bar{\la}^\si_{\ \be})-\nab_{\be}\Im(\la_{\ga\si}\bar{\la}^\si_{\ \al})-\nab_{\ga}\Im(\la_{\al\si}\bar{\la}^\si_{\ \be}))
\\ 
= & \ - \Ric_{\alpha \delta} {\tbmF^\delta}_{\ \beta} + \Ric_{\beta\delta} {\tbmF^\delta}_{\ \al} - R_{\beta\alpha\sigma\delta} \tbmF^{\sigma\delta}\\
&+ \nabla^\ga \Im(\la_\ga^{\ \si}(\overline{\nab^A_\al C^2_{\si\be}}-\overline{\nab^A_\be C^2_{\si\al}})+\nab^A_\ga \la^\si_{\ \be}\bar{C^2}_{\al\si}).
\end{align*}
Summing up the expressions for $I$ and $II$ we obtain \eqref{CompCond-3_eq}.
\end{proof}

\medskip

\begin{proof}[The equation for $C^4$] This has the form
\begin{equation}   \label{CompCond-4-eq}
\begin{aligned}
     \De_g C^4=&-\nab^{\si}(C^6_{\mu\si}A^{\ga})-\frac{1}{2}[\nab^{\al},\nab^{\ga}]C^3_{\ga\al}-\frac{1}{2}\nab^\ga\nab^\al\Im (C^2_{\si\ga}\bar{\la}_\al^{\ \si}+\la_{\ga}^{\ \si}\bar{C^2}_{\al\si}) .
\end{aligned}
\end{equation}
To prove it we commute $\De_g$ with $\nab^{\al}$
    \begin{equation*}
    \begin{aligned}
        \De_g \nab^{\al} A_{\al}= & \ \nab_{\si}[\nab^{\si},\nab^{\al}]A_{\al}+[\nab_{\si},\nab^{\al}]\nab^{\si}A_{\al}+\nab^{\al}\De_g A_{\al} \\
    = & \ - \nabla^\sigma( \Ric_{\mu\si} A^\mu) + \frac12   [\nab^{\si},\nab^{\al}] \bmF_{\si\al}
    + \nabla^\al (\tRic_{\al\si} A^\si) + \nabla^\al \nabla^\gamma \tbmF_{\al\ga}
\end{aligned}        
    \end{equation*}
In the last term we can symmetrize in $\al$ and $\ga$, and the desired equation \eqref{CompCond-4-eq}
follows.
\end{proof}

\medskip

\begin{proof}[The equation for $C^5$]
Here we compute 
\begin{equation}   \label{CompCond-5-eq}
\begin{aligned}
    \De_g C^5_\be =&-[\nab^\al,\nab_\be]C^5_\al-\Re (\nab_\be (C^1 \bar{\psi})-2\bar{\la}^{\al\si}\nab^A_\al C^2_{\si\be}+\nab_\be (\la^{\al\si}C^2_{\si\al})  ).
\end{aligned}
\end{equation}

We can rewrite the $g$ equation \eqref{g-eq-original} as
\[
\Ric_{\alpha\beta} = \tRic_{\alpha\beta} + \frac{1}{2}(\nabla_\alpha C^5_\beta + \nabla_\beta C^5_\alpha)
\]
which by contraction yields 
\[
R = \tR + \nabla^\alpha C^5_\alpha.
\]
To get to $\Delta_g C^5$, by the above two equalities we write
\[
\begin{aligned}
\frac{1}{2}\Delta_g C^5_\beta = & \ \nabla^\alpha (\Ric_{\alpha\beta} - \tRic_{\alpha\beta})
- \frac12 [\nabla^\alpha,\nabla_\beta] C^5_\alpha - \frac12 \nabla_\beta (R-\tR) 
\\
= & \ (\nabla^\alpha R_{\alpha\beta} - \frac12 \nabla_\beta R) 
- \frac12 [\nabla^\alpha,\nabla_\beta] C^5_\alpha - 
(\nabla^\alpha \tR_{\alpha \beta}
-\frac12 \nabla_\beta \tR) .
\end{aligned}
\] 
The first term drops by twice contracted Bianchi,
\begin{equation*}
    g^{\mu\nu}g^{\ga\al}(\nab_\ga R_{\nu\be\mu\al}+\nab_\nu R_{\be\ga\mu\al}+\nab_\be R_{\ga\nu\mu\al})=0,
\end{equation*}
and the last one is quadratic in $\lambda$
and yields $C^1$ and $C^2$ terms,
\begin{align*}
    (\nabla^\alpha \tR_{\alpha \beta}-\frac12 \nabla_\beta \tR) =&
    \Re (\frac{1}{2}\nab_\be (C^1 \bar{\psi})-\bar{\la}^{\al\si}\nab^A_\al C^2_{\si\be}+\frac{1}{2}\nab_\be (\la^{\al\si}C^2_{\si\al})  ).
\end{align*}
This completes the derivation of \eqref{CompCond-5-eq}.
\end{proof}

\medskip

\begin{proof}[The equation for $C^6$]
This has the form
\begin{equation}    \label{CompCond-6-eq}
    C^6_{\ga\si}=\frac{1}{2}(\nab_\ga C^5_\si+\nab_\si C^5_\ga).
\end{equation}
Indeed, by the $g$-equation in \eqref{ell-syst} and its proof, we recover the Ricci curvature
    \begin{align*}
        \Re(\lambda_{\ga\si}\bar{\psi}-\la_{\ga\al}\bar{\la}^\al_{\ \si})=\Ric_{\ga\si}-\frac{1}{2}(\d_\ga C^5_\si+\d_\si C^5_\ga)+\Ga^\nu_{\ga\si}C^5_\nu.
    \end{align*}
    This implies the relation \eqref{CompCond-6-eq} immediately. 
\end{proof}

\medskip

\begin{proof}[The equation for $C^7$]
By the second Bianchi identities of Riemannian curvature and the following equality
\begin{align*}    
    &\nab_\de \Re(\la_{\ga\be}\bar{\la}_{\si\al}-\la_{\ga\al}\bar{\la}_{\si\be})+\nab_\si \Re(\la_{\de\be}\bar{\la}_{\ga\al}-\la_{\de\al}\bar{\la}_{\ga\be})+\nab_\ga \Re(\la_{\si\be}\bar{\la}_{\de\al}-\la_{\si\al}\bar{\la}_{\de\be})=0,
\end{align*}
we have the counterpart of the second Bianchi identities 
\begin{equation*}
 \nab_\delta C^7_{\si\ga\al\be} + \nab_{\si} C^7_{\ga\de\al\be} 
 + \nab_{\ga} C^7_{\de\si\al\be} =0,
\end{equation*}
which combine with the algebraic symmetries of the same tensor to yield 
an elliptic system for $C^7$. Precisely, 
using the above relation we have
\begin{equation*}
\nab^{\si}C^7_{\si\ga\al\be} = \nabla_{\al} C^6_{\ga\be} - \nabla_{\be} C^6_{\ga\al}
+  \nab(\la C^1+\la C^2),
\end{equation*}
which combined with the previous one yields the desired elliptic system, 
with $C^6$ viewed as a source term.
\end{proof}

\bigskip

 b) Assume that $\ts_k$ is an admissible frequency envelope for $\dSS$ in  $\HH^{\si}$. In view of the bound \eqref{t:ell-fixed-time: bd1} and of the smallness of $\|\psi\|_{H^s}$, it suffices to prove the difference or linearized estimate
     \begin{equation}   \label{t:ell-fixed-time: dlt-key}
	 \| S_k\dSS\|_{\mathcal H^{\si}} \lesssim \| S_k\dpsi\|_{H^{\si}}+\ts_k\|\psi\|_{H^s}  (1+\|\psi\|_{H^s})^N.
	 \end{equation}
	 If this is true, then the bound \eqref{t:ell-fixed-time: dlt} follows. Thus, by the definition of frequency envelope \eqref{Freq-envelope}, \eqref{t:ell-fixed-time: dlt} and the smallness of $\psi\in H^s$, we also obtain the bound \eqref{t:ell-fixed-time: fre-env}.
	
	As an intermediate step in the proof of \eqref{t:ell-fixed-time: dlt}, we collect in the next Lemma several bilinear estimates. The proof of this Lemma is standard by Littlewood-Paley decompositions and Bernstein inequality.

	\begin{lemma}    \label{equivalent-fixed-time}
	    Let $d/2-3<\si\leq s$, $d\geq 3$, then we have
	    \begin{gather*}       \label{equivalent_h_fixed-time}
	    \lV\nab\de(\tilde{h}h)\rV_{H^{\si}}\lesssim \lV\nab\de\tilde{h}\rV_{H^{\si}}\lV \nab h\rV_{H^{s}}+\lV\nab\tilde{h}\rV_{H^{s}}\lV\nab \dh\rV_{H^{\si}},\\\label{equivalent_lambda_fixed-time}
	    \lV \de(\lambda h)\rV_{H^{\si}}\lesssim \lV \de\lambda\rV_{H^{\si}}\lV \nab h\rV_{H^{s}}+\lV \lambda\rV_{H^{s}}\lV \nab \dh\rV_{H^{\si}},\\\label{equivalent_A_fixed-time}
	    \lV\nab \de(Ah)\rV_{H^{\si-1}}\lesssim \lV  \nab\dA\rV_{H^{\si-1}}\lV \nab h\rV_{H^{s}}+\lV \nab A\rV_{H^{s}}\lV \nab \dh\rV_{H^{\si}}.
	    \end{gather*}
	\end{lemma}
	
	Now we turn our attention to the proof of \eqref{t:ell-fixed-time: dlt-key}. Here we first prove the estimates for $\dla$. By $\la$-equations in \eqref{ell-syst_2} it suffices to consider the following form
	\begin{align*}
	    &\d_\al \dla_{\al\be}=\d_\be \dpsi+ \dA \psi+A\dpsi+\dh\nab \la+h \nab \dla+\nab \dh \la+\nab h\dla,\\ 
	    &\d_\al \dla_{\be\ga}-\d_\be \dla_{\al\ga}=  \dA \la+A \dla+\nab \dh \la+\nab h\dla.
	\end{align*}
	By the relation
	\begin{equation}       \label{decomp-vector}
	\widehat{\la}(\xi)=|\xi|^{-2}(\widehat{\la}\cdot \xi)\xi+|\xi|^{-2}(\widehat{\la}\xi^{\top}-\xi \widehat{\la}^{\top})\cdot \xi,
	\end{equation}
	we obtain
	\begin{align*}
	    \| S_k\dla\|_{H^\si}\lesssim &\| S_k\dpsi\|_{H^\si}+\| |D|^{-1}S_k[\dA (\la+\psi)+A(\dla+\dpsi)+\dh\nab \la+h \nab \dla\\
	    &\qquad\qquad\qquad\qquad+\nab \dh \la+\nab h\dla]\|_{H^\si}\\
	    \lesssim & \|S_k\dpsi\|_{H^\si}+\ts_k\| \psi \|_{H^s}.
	\end{align*}

	Next we provide the estimate for $\dA$; the other estimates can be proved similarly. By $A$-equation in \eqref{ell-syst_2} and Lemma \ref{equivalent-fixed-time}, it suffices to consider the following form
	\begin{align*}
	    \De \dA=&\dh\nab^2 A+h\nab^2\dA+\nab \dh\nab A+\nab h\nab \dA+\nab \dh\nab h A+(\nab h)^2\dA\\
	    &+\dla \la (A+\nab h)+\la^2 (\dA+\nab\dh)+\nab\la\dla+\la\nab\dla.
	\end{align*}
	
	Using Littlewood-Paley trichotomy and Bernstein inequality, we bound all the nonlinearities except $\nab\la \dla$ and $\la\nab\dla$ by 
	\begin{align*}
	    &\| |D|^{-1}S_k (\dh\nab^2 A+h\nab^2\dA+\nab \dh\nab A+\nab h\nab \dA+\nab \dh\nab h A+(\nab h)^2\dA) \|_{H^\si}\\
	    &+\| |D|^{-1}S_k (\dla \la (A+\nab h)+\la^2 (\dA+\nab\dh)) \|_{H^\si}\\
	    \lesssim & \ts_k \| \psi\|_{H^s}(1+\| \psi\|_{H^s}).
	\end{align*}
	For the remainder terms, we also have
	\begin{align*}
	    \| |D|^{-1}S_0(\nab \la\dla+\la\nab\dla)\|_{L^2}
	    \lesssim & \| S_0(\nab \la\dla+\la\nab\dla)\|_{L^1}
	    \lesssim  \ts_0 \| \la\|_{H^s}
	\end{align*}
	and for $k>0$
	\begin{align*}
	    \| |D|^{-1}S_k(\nab \la\dla+\la\nab\dla)\|_{H^\si}
	    \lesssim  \ts_k \| \la\|_{H^s}.
	\end{align*}
	This completes the proof of \eqref{t:ell-fixed-time: dlt}.

\bigskip

	c) Using the similar argument to \emph{b)}, we have
	\[
	\| D^2\SS(\de_1\psi,\de_2\psi)\|_{\mathcal H^{\si}} \lesssim \|\de_1\SS\|_{\HH^{\si}}\|\de_2\SS\|_{\HH^{s}}   (1+\|\psi\|_{H^s})^N,
	\]
	and
	\[
	\| D^2\SS(\de_1\psi,\de_2\psi)\|_{\mathcal H^{\si}} \lesssim \|\de_1\SS\|_{\HH^{s}}\|\de_2\SS\|_{\HH^{\si}}   (1+\|\psi\|_{H^s})^N.
	\]
	Then by the smallness of $\psi\in H^s$, \eqref{t:ell-fixed-time: dlt} and interpolation, the above two bounds imply
	\[
	\| D^2\SS(\de_1\psi,\de_2\psi)\|_{\mathcal H^{\si}} \lesssim \|\de_1\psi\|_{H^{\si_1}}\|\de_2\psi\|_{H^{\si_2}}.
	\]
    This completes the proof of \eqref{t:ell-fixed-time: ddlt}.
\end{proof}

Next we establish bounds for the above solutions in space-time local energy spaces:
\begin{thm}\label{t:ell-time-dep}
a) Assume that $\psi$ is small in $l^2\bX^s$ for $s>d/2$, $d\geq 4$. Then the solution $(\la,h,A,V,B)$ 
for the elliptic system \eqref{ell-syst} given by Theorem~\ref{t:ell-fixed-time} belongs to $\EE^s$ and satisfies the bounds 
\begin{equation}  \label{t:ell-time-dep_bd1}
\|   \SS \|_{\bEE^s} \lesssim  \|\psi\|_{l^2 \bX^s},
\end{equation}
with Lipschitz dependence on the initial data in these topologies. Moreover, assume that $p_k$ is an admissible frequency envelope for $\psi\in l^2\bX^{s}$, we have the frequency envelope version 
\begin{equation}   \label{t:ell-time-dep_freq-envelope}
	  \| \SS_k \|_{\bEE^s} \lesssim p_k   .
\end{equation}

b) In addition, for the linearization of the  elliptic system \eqref{ell-syst} we have 
the bounds
\begin{equation}         \label{t:ell-time-dep_delta}
\|   \dSS \|_{\bEE^\sigma} \lesssim  \|\dpsi\|_{l^2 \bX^\sigma},
\end{equation}
for $\sigma \in (d/2-1,s]$.

\end{thm}

\begin{proof}[Proof of Theorem \ref{t:ell-time-dep}]
    For the elliptic system \eqref{ell-syst_2}, we will prove the bound for differences $\dSS$
    \begin{equation}  \label{t:ell-time-dep_key2}
         \|  \dSS  \|_{\bEE^{\si}} 
         \lesssim  \|\dpsi\|_{l^2 \bX^{\si}}+\|\dSS \|_{\bEE^{\si}}\|   \SS \|_{\bEE^{s}}\big(1+\|   \SS \|_{\bEE^{s}}\big)^N.
    \end{equation}
    If this is true, by a continuity argument the bounds \eqref{t:ell-time-dep_bd1} and \eqref{t:ell-time-dep_delta} follow.

Assume that $s_k$ is an admissible frequency envelope for $\dSS\in \bEE^{\si}$. We can separate the bound \eqref{t:ell-time-dep_key2} into two parts, 
namely 
\begin{equation*}
    \|  \d_t\dSS  \|_{L^2\HH^{\si-2}} 
         \lesssim  \|\dpsi\|_{l^2 \bX^{\si}}(1+\| \d_t\psi\|_{L^2 H^{s-2}})
\end{equation*}
respectively
\begin{equation}\label{dS-nodt}
    \|  S_k\dSS  \|_{\EE^{\si}} 
         \lesssim  p_k+s_k\|   \SS \|_{\bEE^{s}}\big(1+\|   \SS \|_{\bEE^{s}}\big)^N. 
\end{equation}
Here one can think of the first bound as a fixed time bound for the linearization of the elliptic system \eqref{ell-syst}, square integrated in time. As such, this is a direct consequence
of the bound \eqref{t:ell-fixed-time: dlt} with argument $\d_t\de\psi$ and regularity index $\si-2$, and the bound \eqref{t:ell-fixed-time: ddlt} with $\de_1=\d_t,\de_2=\de,\si_1=s-2,\si_2=\si$ in Theorem~\ref{t:ell-fixed-time}. So it remains to prove
\eqref{dS-nodt}.

If the bound \eqref{dS-nodt} holds, then by the bound \eqref{t:ell-fixed-time: fre-env} with $\de=\d_t,\si=s-2$ and \eqref{dS-nodt} with $\de=Id,\si=s$, the bound \eqref{t:ell-time-dep_freq-envelope} follows.

As an intermediate step in the proof of \eqref{dS-nodt},
we  collect in the next Lemma several bilinear estimates and equivalent relations.

\begin{lemma}[Bilinear estimates]\label{equivalent_lemma}
	Let $s>d/2$, $0< \si\leq s$, $d\geq 4$, assume that $ h\in \bY^{s}$, then we have
	\begin{gather}       \label{equivalent_h_keybd}
	\lV\tilde{h}h\rV_{Y^{\si}}\lesssim \lV\tilde{h}\rV_{\bY^{\si}}\lV h\rV_{\bY^{s}},\\\label{equivalent_lambda_keybd}
	\lV \lambda h\rV_{Z^{0,\si}}\lesssim \lV \lambda\rV_{\bZ^{0,\si}}\lV h\rV_{\bY^{s}},\\\label{equivalent_A_keybd}
	\lV(Ah)\rV_{Z^{1,\si}}\lesssim \lV  A\rV_{\bZ^{1,\si}}\lV h\rV_{\bY^{s}}.
	\end{gather}
	As consequences of these bounds, for $h^{\al\be}=g^{\al\be}-I,h_{\al\be}=g_{\al\be}-I$, $\lambda^{\al\be}=g^{\al\ga}\lambda_{\ga}^{\be}, \lambda^{\be}_{\ga}=g^{\be\nu}\lambda_{\gamma\nu}$, $V^{\al}=g^{\al\be}V_{\be}$ and $A^{\al}=g^{\al\be}A_{\be}$, assume that $\lV  h_{\al\be}\rV_{\bY^{\si+1}}\ll 1$, we have
	\begin{gather*}   \label{eqivalent_h_ab&h^ab}
	\lV h_{\al\be}\rV_{\bY^{\si+1}}\approx \lV  h^{\al\be}\rV_{\bY^{\si+1}},\\      \label{eqivalent_lambda}
	\lV \lambda^{\al\be}\rV_{\bZ^{0,\si}}\approx \lV \lambda^{\be}_{\al}\rV_{\bZ^{0,\si}}\approx \lV \lambda_{\al\be}\rV_{\bZ^{0,\si}},\\
	\label{equivalent_V}
	\lV  V_{\al}\rV_{\bZ^{1,\si}}\approx \lV  V^{\al}\rV_{\bZ^{1,\si}},\\\label{equivalent_A}
	\lV  A_{\al}\rV_{\bZ^{1,\si}}\approx \lV  A^{\al}\rV_{\bZ^{1,\si}}.
	\end{gather*}
\end{lemma}
\begin{proof}[Proof of Lemma \ref{eqivalent_lambda}]
We do this in several steps:

\medskip
	\emph{Proof of the bound (\ref{equivalent_h_keybd}).} First, we consider the $Y$-norm estimates. For the high-low interaction, for any decomposition $P_j\tih=\sum_{l\geq |j|} \tih_{j,l}$, we have
	\begin{align*}
	\lV \sum_{l\geq |j|}( \tih_{j,l} h_{\leq j})\rV_{Y_j}\lesssim \sum_{l\geq |j|} 2^{l-|j|}\lV( \tih_{j,l} h_{\leq j})\rV_{l^1_lL^{\infty}L^2}
	\lesssim  \sum_{l\geq |j|} 2^{l-|j|}\lV \tih_{j,l}\rV_{l^1_lL^{\infty}L^2}\lV h_{\leq j}\rV_{L^{\infty}L^{\infty}}.
	\end{align*}
	Taking the infimum over the decomposition of $\tih_j$ yields
	\begin{equation*}
	\lV \sum_{l\geq |j|}( \tih_{j,l} h_{\leq j})\rV_{Y_j}\lesssim \lV P_j \tih\rV_{Y_j} \lV  h\rV_{Z^{1,s}},
	\end{equation*}
	which is acceptable. Similarly, for the low-high interaction, we have
	\begin{align*}
	\lV \sum_{l\geq |j|}( P_{\leq j}\tih  h_{j,l})\rV_{Y^{\frac{d}{2}-1-\de,\si}}\lesssim & 2^{(\frac{d}{2}-1-\de)j^-+\si j^+} \sum_{k\leq j}2^{dl/2}\lV P_{k} \tih\rV_{L^{\infty}L^2} \lV P_j h\rV_{Y_j}\\
	\lesssim &\lV \nab \tih\rV_{L^{\infty}H^{\si-1}}\lV P_j h\rV_{Y^{d/2-\de,s}},
	\end{align*}
	which is acceptable.
	
	Next, for the high-high interaction, when $j<0$ we rewrite it as
	\[
	\sum_{j<j_1<-j} P_j (P_{j_1}\tih P_{j_1}h)+\sum_{-j\leq j_1} P_j (P_{j_1}\tih P_{j_1}h).
	\]
	Then we bound the first term by
	\begin{align*}
	&2^{(d/2-1-\delta)j}\lV\sum_{j<j_1<-j}P_j(P_{j_1} \tih  P_{j_1}h)\rV_{Y_j}\\
	\lesssim & 2^{(d-1-\delta)j}\sum_{j<j_1<-j}\lV P_{j_1} \tih  P_{j_1}h\rV_{l^1_{|j|}L^{\infty}L^1}\\
	\lesssim & 2^{(d-1-\delta)j}\sum_{j<j_1<-j}\lV P_{j_1}\tih\rV_{l^2_{|j|}L^{\infty}L^2}\lV P_{j_1}h\rV_{l^2_{|j|}L^{\infty}L^2}\\
	\lesssim & 2^{(d-3-2\delta)j}\lV \nab \tih_{\leq 0}\rV_{l^2_0L^{\infty}L^2}\lV \nab h_{\leq 0}\rV_{l^2_0L^{\infty}L^2}+2^{(d-1-\de)j}\lV \tih\rV_{Z^{1,0}}\lV h\rV_{Z^{1,0}}.
	\end{align*}
	We bound the second term by
	\begin{align*}
	2^{(d/2-1-\delta)j}\lV\sum_{-j\leq j_1}P_j(P_{j_1} h P_{j_1} h)\rV_{Y_j}
	\lesssim & \sum_{-j\leq j_1}2^{(d-\delta)j+j_1}\lV( P_{j_1}h P_{j_1}h)\rV_{l^1_{j_1}L^{\infty}L^1}\\
	\lesssim & \sum_{-j\leq j_1}2^{(d-\delta)j+j_1}\lV P_{j_1}h\rV_{l^2_{j_1}L^{\infty}L^2}^2\\
	\lesssim & 2^{(d-\delta)j}\lV  h\rV_{Z^{1,0}}\lV  h\rV_{Z^{1,1}}.
	\end{align*}
	When $j\geq 0$, we have
	\begin{align*}
	&2^{\si j}\lV\sum_{j_1>j}P_j(P_{j_1} h P_{j_1} h)\rV_{Y_j}\\
	\lesssim & \sum_{j_1>j}2^{(\si-1+d/2)j+j_1}\lV(P_{j_1} h P_{j_1} h)_j\rV_{l^1_{j_1}L^{\infty}L^1}\\
	\lesssim & \sum_{j_1>j}2^{(\si-1+d/2)(j-j_1)}\lV P_{j_1}h\rV_{Z^{1,\si}}\lV P_{j_1}h\rV_{Z^{1,s}},
	\end{align*}
	which is acceptable.
	
	Secondly, we consider the $Z^{1,\si+1}$-norm estimates. For the low-frequency part, we have
	\begin{align*}
	\lV \nab(h^2)_{\leq 0}\rV_{l^2_0L^{\infty}L^2}
	\lesssim & \lV \nab h_{\leq 0}\rV_{l^2_0L^{\infty}L^2}\lV h_{\leq 0}\rV_{L^{\infty}L^{\infty}}+\sum_{j>0}\lV (h_j h_j)_{\leq 0}\rV_{l^2_0L^{\infty}L^1}\\
	\lesssim & \lV h\rV_{Z^{1,0}}\lV h\rV_{Z^{1,s}}.
	\end{align*}
    For the high frequency part, by Littlewood-Paley dichotomy, we have
	\begin{align*}
	&2^{\si j}\lV (\tih h)_j\rV_{l^2_jL^{\infty}L^2}\\
	\lesssim & 2^{\si j}\lV h_j\rV_{l^2_jL^{\infty}L^2}\lV h_{\leq j}\rV_{L^{\infty}L^{\infty}}+\sum_{l\geq j}2^{(\si+d/2) j}\lV (h_l h_l)_j\rV_{l^2_jL^{\infty}L^1}\\
	\lesssim & 2^{\si j}\lV h_j\rV_{l^2_jL^{\infty}L^2}\lV\nab h\rV_{L^{\infty}H^{s-1}}+\sum_{l\geq j}2^{\si(j-l)}2^{(\si+d/2)l}\lV h_l \rV_{l^2_lL^{\infty}L^2}\lV h_l\rV_{L^{\infty}L^2},
	\end{align*}
	which is acceptable. This completes the proof of (\ref{equivalent_h_keybd}).
	\medskip
	
	\emph{Proof of the bound (\ref{equivalent_lambda_keybd}).} First we consider the $Z^{\de,\si}$-norm estimates. For the low-frequency part we have
	\begin{align*}
	\lV (h\lambda)_{\leq 0}\rV_{l^2_0L^{\infty}L^2}\lesssim & \lV h_{\leq 0}\|_{L^\infty L^\infty}\|\lambda_{\leq 0}\rV_{l^2_0L^{\infty}L^2}+\sum_{j>0} 2^{dj/2}\| h_j\|_{L^\infty L^2}\lV \la_j\rV_{l^2_jL^{\infty}L^2}\\
	\lesssim & \lV h\rV_{Z^{1,s}}\lV \lambda\rV_{Z^{0,\si}}.
	\end{align*}
	For the high-frequency part, by the Littlewood-Paley dichotomy, we have
	\begin{align*}
	&2^{\si j}\lV (\lambda h)_j\rV_{l^2_jL^{\infty}L^2}\\
	\lesssim & \sum_{l<j}2^{\si j+dl/2}\lV \lambda\rV_{L^{\infty}L^2}\lV h_j\rV_{l^2_jL^{\infty}L^2}+2^{\si j}\lV \lambda_j\rV_{l^2_jL^{\infty}L^2}\lV h\rV_{Z^{1,s}}\\
	&+\sum_{l>j}2^{\si(j-l)}2^{(\si+d/2)l}\lV h_l\rV_{l^2_lL^{\infty}L^2}\lV \lambda_l\rV_{L^{\infty}L^2},
	\end{align*}
	which implies
	\begin{align*}
	(\sum_{j>0}2^{2\si j}\lV (h\lambda)_j\rV_{l^2_jL^{\infty}L^2}^2)^{1/2}\lesssim \lV h\rV_{Z^{1,s}}\lV \lambda\rV_{Z^{\de,\si}}.
	\end{align*}
	This completes the proof of (\ref{equivalent_lambda_keybd}).

	\medskip

	\emph{Proof of the bound (\ref{equivalent_A_keybd}).}
	For the low-frequency part, by Bernstein's inequality we have
	\begin{align*}
	\lV \nab (Ah)_{\leq 0}\rV_{l^2_0L^{\infty}L^2}\lesssim &\lV \nab (A_{\leq 0}h_{\leq 0})\rV_{l^2_0L^{\infty}L^2}+\sum_{j>0}\lV \nab (A_jh_j)_{\leq 0}\rV_{l^2_0L^{\infty}L^2}\\
	\lesssim & \lV \nab A_{\leq 0} \rV_{l^2_0L^{\infty}L^2}\lV \nab h_{\leq 0}\rV_{L^{\infty}L^2}+\lV \nab A_{\leq 0}\rV_{L^{\infty}L^2}\lV \nab h_{\leq 0}\rV_{l^2_0L^{\infty}L^2}\\
	&+\sum_{j>0} 2^{dj/2}\lV A_j\rV_{l^2_jL^{\infty}L^2} \lV h_j\rV_{L^{\infty}L^2}\\
	\lesssim & \lV  A \rV_{Z^{1,0}}\lV  h\rV_{Z^{1,s}}.
	\end{align*}
	For the high-frequency part, by Littlewood-Paley dichotomy we bound the high-low and low-high interactions by
	\begin{align*}
	&2^{\si k}\lV  S_k(A_k h_{<k}+A_{<k}h_k)\rV_{l^2_kL^{\infty}L^2}\\
	\lesssim & 2^{\si k}(\lV A_k\rV_{l^2_kL^{\infty}L^2}\lV h_{<k}\rV_{L^{\infty}L^{\infty}}+\lV A_{<k}\rV_{L^{\infty}L^{\infty}}\lV h_k\rV_{l^2_kL^{\infty}L^2})\\
	\lesssim & \lV  A_k\rV_{Z^{1,\si}}\lV  h\rV_{Z^{1,s}}+\lV  A\rV_{Z^{1,\si}}\lV  h_k\rV_{Z^{1,s}},
	\end{align*}
	which is acceptable. We bound the high-high interaction by 
	\begin{align*}
	&2^{\si k}\sum_{j>k}\lV  S_k(A_j h_{j})\rV_{l^2_kL^{\infty}L^2}\\
	\lesssim & \sum_{j>k} 2^{(\si+d/2)k}\lV A_j h_{j}\rV_{l^2_kL^{\infty}L^1}\\
	\lesssim & \sum_{j>k} 2^{\si(k-j)}2^{(\si+d/2)j}\lV A_j \rV_{l^2_jL^{\infty}L^2}\lV h_j \rV_{L^{\infty}L^2},
	\end{align*}
	which is also acceptable. 
	 Hence, we conclude the proof of the bound (\ref{equivalent_A_keybd}).
\end{proof}

    We now turn our attention to the proof of \eqref{dS-nodt}.
   
   \medskip
   
   \emph{Step 1. Proof of the elliptic estimates for $\la$ equations.} By the $\la$-equations and Proposition \ref{equivalent_lemma}, it suffices to consider the following simplified form
   of the equations:
	\begin{align*}
	&\d_\al \dla_{\al\be}= \d_\be \dpsi+ \dA \psi+A \dpsi+\dh\nab\la+h\nab\dla+\nab \dh\la+\nab h\dla,\\ 
	&\d_\al \dla_{\be\ga}-\d_\be \dla_{\al\ga}=  \dA\la+A\dla+\nab \dh \la+\nab h \dla.
	\end{align*}
	By the relation \eqref{decomp-vector} we have for any $k>0$
	\begin{align*}
	    \| S_k\dla\|_{Z^{0,\si}}\lesssim & \ \| S_k\mathcal{R}\dpsi\|_{Z^{0,\si}}+\| S_k|D|^{-1}[\dA (\psi+\la)+A (\dpsi+\dla)\\
	    &+\dh\nab\la+h\nab\dla+\nab \dh\la+\nab h\dla] \|_{Z^{0,\si}}\\
	    \lesssim & \  p_k+s_k \|\SS\|_{\EE^s}.
	\end{align*}
	In order to bound the low frequency part $k=0$, we use the relation
	\begin{equation}\label{f(t)f(0)relation}
	f(t)=f(0)+\int_0^t \d_s f(s)ds.
	\end{equation}
	Then we have 
	\[
	\lV f\rV_{l_0^2 L^{\infty}L^2}\lesssim \lV f(0)\rV_{L^2}+\lV \d_t f\rV_{L^2L^2}.
	\]
	Using this idea, by Sobolev embeddings we have
	\begin{align*}
	    \| S_0 \dla\|_{l^2_0L^{\infty}L^2}\lesssim & \ \| S_0\mathcal{R}\dpsi\|_{l^2_0L^{\infty}L^2}
	    +\| S_0|D|^{-1}[\dA (\psi+\la)+A (\dpsi+\dla)\\
	    &\qquad\qquad\qquad\qquad+\dh\nab\la+h\nab\dla+\nab \dh\la+\nab h\dla] \|_{l^2_0L^{\infty}L^2}\\
	    \lesssim & \  \| S_0\dpsi\|_{\bZ^{0,\si}}\\ & \ +\| S_0|D|^{-1}[\dA (\psi+\la)+A (\dpsi+\dla) +\dh\nab\la\\
	    &\qquad+h\nab\dla+\nab \dh\la+\nab h\dla](0) \|_{L^2}\\
	    &\ +\| S_0|D|^{-1}\partial_t [\dA (\psi+\la)+A (\dpsi+\dla)+\dh\nab\la\\
	    &\qquad+h\nab\dla+\nab \dh\la+\nab h\dla] \|_{L^2L^2}\\
	    \lesssim & \ p_0+s_0\| \SS\|_{\EE^s}.
	\end{align*}
	The high frequency part is obtained by a standard Littlewood-Paley decomposition and Bernstein inequality.	This gives the elliptic estimate for the $\dla$-equation.
    \medskip
    
     \emph{Step 2. Proof of the elliptic estimates for $V$, $A$ and $B$ equations.} By the $V, A,B$-equations and Proposition \ref{equivalent_lemma}, it suffices to consider the following form
	\begin{align*}
	\De V=&\ h\nab^2 V+\nab h\nab V+\nab h\nab h V+\la^2 (A+V+\nab h)+\la\nab \la,\\
	\De A=&\ h\nab^2 A+\nab h\nab A+\nab h\nab h A+\la^2( A+\nab h)+\nab (\la^2),\\
	\Delta B=&\ h\nab^2B+\nab(\lambda \nab \la+(V+A)\la^2)+\la^2\nab A+\nab h (\la\nab\la+(V+A)\la^2)\\
	&+\nab V\nab A+\nab h V\nab A.
	\end{align*}
	The proofs of the four elliptic estimates for the above equations are similar, so we only prove the elliptic estimate for the linearization of $A$-equation in detail, i.e.
	\begin{align*}
	    \De \dA=&\dh\nab^2 A+h\nab^2\dA+\nab \dh\nab A+\nab h\nab \dA+\nab \dh\nab h A+(\nab h)^2\dA\\
	    &+\dla \la (A+\nab h)+\la^2 (\dA+\nab\dh)+\nab\la\dla+\la\nab\dla.
	\end{align*}
	
	We bound all the nonlinearities except $\nab\la\dla $ and $\la\nab\dla$ by
	\begin{align*}
	    &\| |D|^{-2}S_k(\dh\nab^2 A+h\nab^2\dA+\nab \dh\nab A+\nab h\nab \dA+\nab \dh\nab h A+(\nab h)^2\dA)\|_{Z^{1,\si+1}}\\
	    &+\| |D|^{-2}S_k(\dla \la (A+\nab h)+\la^2 (\dA+\nab\dh))\|_{Z^{1,\si+1}}\\
	    \lesssim &\  s_k\| \dSS\|_{\bEE^{\si}}  \| \SS\|_{\bEE^{s}}(1+\| \SS\|_{\bEE^{s}})^N,
	\end{align*}
	for $\si\in(d/2-1,s]$. All terms are estimated in a similar fashion, so we only bound the first term $\dh \nab^2 A$.
	
	For the low-frequency part we use the relation \eqref{f(t)f(0)relation} to
	bound the second term $\dh\nab^2A$ by
	\begin{align*}
	&\lV \nab^{-1}(\dh\nab^2A)_{\leq 0}\rV_{l_0^2L^{\infty}L^2}\\
	\lesssim & \  \lV \nab^{-1}(\dh\nab^2 A)_{\leq 0}(0)\rV_{L^2}+\lV \nab^{-1}\d_t(\dh\nab^2 A)_{\leq 0}\rV_{L^2L^2}\\
	\lesssim & \  \lV (\dh\nab^2 A)_{\leq 0}(0)\rV_{L^{2d/(d+2)}}
	+\lV \d_t(\dh\nab^2 A)_{\leq 0}\rV_{L^2L^{2d/(d+2)}}\\
	\lesssim & \  \lV  \dh\rV_{Z^{1,1}}\lV A\rV_{Z^{1,1}}+\lV \nab\d_t \dh\rV_{L^2H^{\si-1}}\lV A\rV_{Z^{1,s+1}}+\lV  \dh\rV_{Z^{1,\si+2}}\lV \nab\d_t A\rV_{L^2H^{s-3}}\\
	\lesssim & \  \lV \dh\rV_{\bZ^{1,\si+2}}\lV A\rV_{\bZ^{1,s+1}}.
	\end{align*}	
	A minor modification of this argument also yields
	\begin{equation*}
	\lV \nab^{-1}(\dh\nab^2 A)_{\leq 0}\rV_{l_0^2L^{\infty}L^2}
	\lesssim s_0\lV \SS\rV_{\bEE^{s}}.
	\end{equation*}
	For the high-frequency part, by Littlewood-Paley dichotomy and Bernstein's inequality \eqref{Bern-ineq}, we have
	\begin{align*}
	&2^{\si j}\lV |D|^{-1}(\dh\nab^2A)_j\rV_{l^2_jL^{\infty}L^2}\\
	\lesssim &\ 2^{(\si -1)j}(\lV \dh_{<j}\nab^2 A_j\rV_{l^2_jL^{\infty}L^2}+\lV \dh_{j}\nab^2A_{<j}\rV_{l^2_jL^{\infty}L^2}+\sum_{l>j}\lV \dh_{l}\nab^2A_l\rV_{l^2_jL^{\infty}L^2})\\
	\lesssim &\  \lV \dh\rV_{L^{\infty}L^{\infty}}2^{(\si+1) j}\lV A_j\rV_{l^2_jL^{\infty}L^2}+\sum_{l<j} 2^{(l-j)+\si j}\lV  \dh_j\rV_{l^2_jL^{\infty}L^2}\lV \nab^{d/2+1}A_l\rV_{L^{\infty}L^2}\\
	&\ +\sum_{l>j}2^{(\si-1)j}2^{(d/2+2)l}\lV  \dh_l\rV_{L^{\infty}L^2}\lV A_l\rV_{l^2_lL^{\infty}L^2}\\
	\lesssim & s_j\lV \SS\rV_{\bEE^{s}}.
	\end{align*}

	Finally, we bound the last two terms $\nab\la\dla$ and $\la\nab\dla$. For low-frequency part, using $d\geq 4$ we have
	\begin{align*}
	&\lV |D|^{-1}(\nab\la\dla)_{\leq 0}\rV_{l^2_0L^{\infty}L^2}\\
	\lesssim & \  \lV |D|^{-1}(\nab\la\dla)_{\leq 0}(0)\rV_{L^2}+\lV |D|^{-1}\d_t(\nab\la\dla)_{\leq 0}\rV_{L^2L^2}\\
	\lesssim & \ \lV (\nab\la\dla)_{\leq 0}(0)\rV_{L^1}+\lV \d_t(\nab\la\dla)_{\leq 0}\rV_{L^2L^1}\\
	\lesssim & \  \lV\dla\rV_{\bZ^{0,\si}}\lV\la\rV_{\bZ^{0,s}}.
	\end{align*}
	We also obtain
	\begin{equation*}
	\lV |D|^{-1}(\nab\la\dla)_{\leq 0}\rV_{l^2_0L^{\infty}L^2}
	\lesssim s_0 \lV\SS\rV_{\bEE^{s}}.
	\end{equation*}
	For the high-frequency part, we have
	\begin{align*}
	\lV\De^{-1}(\nab\la\dla)_j\rV_{Z^{1,\si+1}}\lesssim s_j \lV\SS\rV_{\bEE^{s}}.
	\end{align*}
	We can also bound the term $\la\nab\dla$ similarly. 
	This gives the elliptic estimate for $\dA$-equation.

    \medskip

\emph{Step 3. Proof of the elliptic estimate for $h$-equation.}
	By $h$-equation in (\ref{ell-syst_2}) and Proposition \ref{equivalent_lemma}, it suffices to consider a more general equation of the form		
	\begin{equation*}
	\Delta \dh=\dh\nab^2 h+h\nab^2 \dh+\nab \dh\nab h+\dh\nab h\nab h+h\nab h\nab \dh+\dla\lambda.
	\end{equation*}
	The proof of the $Z^{1,\si+2}$ bound is similar to the estimates for $\la,V,A,B$ equations in \emph{Step 1}, hence we only bound of the $Y^{d/2-1-\de,\si+2}$-norm. We prove that
    the following frequency envelope version holds:
	\begin{equation*}     
	\lV S_j \dh\rV_{Y^{d/2-1-\de,\si+2}}\lesssim s_j \lV\SS\rV_{\bEE^{s}}(1+\lV \SS\rV_{\bEE^{s}})^N.
	\end{equation*}
	 
	\medskip
	
	\emph{Case 1. The contribution of $\dla\lambda$.} 
	By the Littlewood-Paley dichotomy, it suffices to consider the high-low, low-high and high-high case
	\begin{equation*}
	\sum_{j_2<j+O(1)}(\dla_{j} \lambda_{j_2})_j,\ \ \ \sum_{j_1<j+O(1)}(\dla_{j_1} \lambda_{j})_j,\quad \sum_{j_1>j+O(1)}(\dla_{j_1} \lambda_{j_1})_j,
	\end{equation*}
	for any $j\in\Z$.
	
	\smallskip
	\emph{Case 1(a). The contribution of high-low and low-high interaction.} The two cases are proved similarly, so we only consider the worst case, namely the low-high interaction. When $j\leq 0$, by the definition of the $Y_j$-norm we have
	\begin{align*}
	2^{(d/2-1-\delta)j}\lV \Delta^{-1}\sum_{j_1<j}(\dla_{j_1} \lambda_{j})_j\rV_{Y_j}\lesssim &\ 2^{(d/2-\delta)j}\sum_{j_1<j} 2^{-j_1}\lV 2^{-2j}(\dla_{j_1} \lambda_{j})\rV_{l^1_{|j_1|}L^{\infty}L^2} \\
	\lesssim & \  2^{(d/2-2-\delta)j}\sum_{j_1<j} 2^{dj_1/2-j_1}\lV \dla_{j_1} \rV_{l^2_{|j_1|}L^{\infty}L^2} \lV\lambda_{j}\rV_{l^2_{|j_1|}L^{\infty}L^2}\\
	\lesssim & \  2^{(d-3-3\delta)j} \lV |D|^{\delta}\dla_{\leq 0}\rV_{l^2_0L^{\infty}L^2}^2\lV |D|^{\delta}\la_{\leq 0}\rV_{l^2_0L^{\infty}L^2}^2\\
	\lesssim & \  2^{(d-3-3\delta)j} s_0\lV \SS\rV_{\bEE^{s}}.
	\end{align*}

	When $j>0$, we further divide the high-low interaction into
	\begin{equation*}
	\sum_{j_1<j}(\dla_{j_1} \lambda_{j})_j=\sum_{-j\leq j_1<j}(\dla_{j_1} \lambda_{j})_j+\sum_{j_1<-j}(\dla_{j_1} \lambda_{j})_j.
	\end{equation*}
	For the first term, by Bernstein's inequality we have 
	\begin{align*}
	2^{(\si+2)j}\lV \Delta^{-1}\sum_{-j\leq j_1<j}(\dla_{j_1} \lambda_{j})_j\rV_{Y_j}\lesssim & \  2^{\si j}\sum_{-j\leq j_1\leq j} \lV (\dla_{j_1} \lambda_{j})_j\rV_{l^1_j L^{\infty}L^2}\\
	\lesssim & \  2^{\si j}\sum_{-j\leq j_1\leq j} 2^{dj_1/2}\lV \dla_{j_1} \rV_{l^2_j L^{\infty}L^2}\lV \lambda_{j}\rV_{l^2_j L^{\infty}L^2} \\
	\lesssim &\ s_j \lV \SS\rV_{\bEE^{s}}.
	\end{align*}
	For the second term we have
	\begin{align*}
	2^{(\si+2)j}\lV \Delta^{-1}\sum_{j_1<-j}(\dla_{j_1} \lambda_{j})_j\rV_{Y_j}\lesssim &\ 2^{(\si-1)j}\sum_{j_1<-j} 2^{-j_1}\lV (\dla_{j_1} \lambda_{j})\rV_{l^1_{|j_1|}L^{\infty}L^2} \\
	\lesssim &\ 2^{(\si-1)j}\sum_{j_1<-j} 2^{(d/2-1)j_1}\lV \dla_{j_1} \rV_{l^2_{|j_1|}L^{\infty}L^2} \lV\lambda_{j}\rV_{l^2_{|j_1|}L^{\infty}L^2}\\
	\lesssim &\  2^{(\si-1)j} \lV |D|^{\de}\dla_{\leq 0} \rV_{l^2_0L^{\infty}L^2}\lV \lambda_{j}\rV_{l^2_jL^{\infty}L^2}\\
	\lesssim &\  s_j \lV \SS\rV_{\bEE^{s}}.
	\end{align*}
	
	\smallskip
	
	\emph{Case 1(b). The contribution of high-high interactions.} When $j<0$, we divide this into 
	\begin{equation*}
	\sum_{j_1>j} (\dla_{j_1}\lambda_{j_1})_j=\sum_{-j\geq j_1>j} (\dla_{j_1}\lambda_{j_1})_j+\sum_{j_1>-j} (\dla_{j_1}\lambda_{j_1})_j.
	\end{equation*}
	Then we bound the first term by
	\begin{align*}
	&2^{(d/2-1-\delta)j}\lV\Delta^{-1}\sum_{-j\geq j_1>j}(\dla_{j_1}\lambda_{j_1})_j\rV_{Y_j}\\
	\lesssim & \  2^{(d-3-\delta)j}\sum_{-j\geq j_1>j}\lV(\dla_{j_1}\lambda_{j_1})_j\rV_{l^1_{|j|}L^{\infty}L^1}\\
	\lesssim & \  2^{(d-3-4\delta)j}(\sum_{0\geq j_1>j}2^{\delta j_1}\lV\nab^{\delta}\dla_{\leq 0}\rV_{l^2_0L^{\infty}L^2}\lV\nab^{\delta}\lambda_{\leq 0}\rV_{l^2_0L^{\infty}L^2}+\sum_{j_1>0}\lV \dla_{j_1}\rV_{l^2_{j_1}L^{\infty}L^2}\lV \lambda_{j_1}\rV_{l^2_{j_1}L^{\infty}L^2})\\
	\lesssim &\ 2^{(d-3-4\delta)j} s_0\lV \SS\rV_{\bEE^{s}}.
	\end{align*}
	Using the $Y_j$ norm we can also bound the second term by
	\begin{align*}
	2^{(d/2-1-\delta)j}\lV\Delta^{-1}\sum_{j_1>-j}(\dla_{j_1}\lambda_{j_1})_j\rV_{Y_j}
	\lesssim &\  2^{(d-2-\delta)j}\sum_{j_1>-j}2^{j_1}\lV(\dla_{j_1}\lambda_{j_1})_j\rV_{l^1_{j_1}L^{\infty}L^1}\\
	\lesssim & \  2^{(d-2-\delta)j}\sum_{j_1>-j}2^{j_1}\lV\dla_{j_1}\rV_{l^2_{j_1}L^{\infty}L^2}\lV\lambda_{j_1}\rV_{l^2_{j_1}L^{\infty}L^2}\\
	\lesssim &\ 2^{(d-2-4\delta)j} s_0\lV \SS\rV_{\bEE^{s}}.
	\end{align*}
	Finally, when $j>0$, using again the $Y_j$ norm we  have
	\begin{align*}
	2^{(\si+2)j}\lV\Delta^{-1}\sum_{j_1>j}(\dla_{j_1}\lambda_{j_1})_j\rV_{Y_j}
	\lesssim & \ 2^{(\si+d/2-1)j}\sum_{j_1>j}2^{j_1}\lV(\dla_{j_1}\lambda_{j_1})_j\rV_{l^1_{j_1}L^{\infty}L^1}\\
	\lesssim & \  \sum_{j_1>j}2^{(\si+d/2-1)(j-j_1)}2^{(\si+d/2)j_1}\lV\dla_{j_1}\rV_{l^2_{j_1}L^{\infty}L^2}\lV\lambda_{j_1}\rV_{l^2_{j_1}L^{\infty}L^2}\\
	\lesssim & \ s_j \lV \SS\rV_{\bEE^{s}}.
	\end{align*}

	\medskip
	
	\emph{Case 2. The contribution of $\dh\nab^2 h$, $h\nab^2 \dh$ and $\nab \dh \nab h$.} It suffices to prove that
	\begin{equation*}
	\lV \De^{-1}S_j(\dh\nab^2 h+\nab^2\dh\cdot h+\nab \dh\nab h)\rV_{Y^{d/2-1-\de,\si+2}}\lesssim s_j \lV \SS\rV_{\bEE^{s}}.
	\end{equation*}
	
	For the high-low interactions, it suffices to consider the worst case	$\nab^2 P_j\dh\cdot P_{\leq j} h$. For any decomposition $P_j\dh=\sum_{l\geq |j|} \dh_{j,l}$, we have
	\begin{align*}
	\lV \Delta^{-1}\sum_{l\geq |j|}(\nab^2  \dh_{j,l} P_{\leq j}h )\rV_{Y_j}\lesssim &\sum_{l\geq |j|} 2^{l-|j|-2j}\lV(\nab^2 \dh_{j,l} P_{\leq j}h )\rV_{l^1_lL^{\infty}L^2}\\
	\lesssim & \sum_{l\geq |j|} 2^{l-|j|}\lV \dh_{j,l}\rV_{l^1_lL^{\infty}L^2}\lV P_{\leq j}h\rV_{L^{\infty}L^{\infty}}
	\end{align*}
	Taking the infimum over the decomposition of $P_jh$ yields
	\begin{equation*}
	\lV \Delta^{-1}(\nab^2P_j \dh P_{\leq j}h)\rV_{Y_j}\lesssim \lV P_j\dh\rV_{Y_j}  \lV P_{\leq j}h\rV_{L^{\infty}L^{\infty}},
	\end{equation*}
	which is acceptable. The low-high interactions is similar and omitted.
	
	For the high-high interaction, it suffices to estimate $\sum_{j_1>j}(\nab \dh_{j_1}\nab h_{j_1})_j$. By Bernstein's inequality we have
	\begin{align*}
	&2^{(d/2-1-\delta)j^-+(\si+2)j^+}\lV\Delta^{-1}\sum_{j_1>j}(\nab \dh_{j_1}\nab h_{j_1})_j\rV_{Y_j}\\
	\lesssim & 2^{(d-3-\delta)j^-+(\si+d/2)j^+}\sum_{j_1>j}\lV(\nab \dh_{j_1}\nab h_{j_1})_j\rV_{l^1_{|j|}L^{\infty}L^1}\\
	\lesssim & 2^{(d-3-\delta)j^-+(\si+d/2)j^+}\sum_{j_1>j}\lV\nab \dh_{j_1}\rV_{l^2_{|j|}L^{\infty}L^2}\lV\nab h_{j_1}\rV_{l^2_{|j|}L^{\infty}L^2}\\
	\lesssim & 2^{(d-3-2\delta)j^-}(\lV \nab \dh_{\leq 0}\rV_{l^2_0L^{\infty}L^2}\lV \nab h_{\leq 0}\rV_{l^2_0L^{\infty}L^2}+\sum_{j_1>0}2^{dj_1}\lV\nab \dh_{j_1}\rV_{l^2_{j_1}L^{\infty}L^2}\lV\nab h_{j_1}\rV_{l^2_{j_1}L^{\infty}L^2})\\
	&+\sum_{j_1>j,j>0}2^{(\si-d/2)(j-j_1)}2^{(\si+d/2)j_1}\lV\nab \dh_{j_1}\rV_{l^2_{j_1}L^{\infty}L^2}\lV\nab h_{j_1}\rV_{l^2_{j_1}L^{\infty}L^2}\\
	\lesssim & 2^{(d-3-2\delta)j^-} s_0 \lV \SS\rV_{\bEE^{s}}+s_j \lV \SS\rV_{\bEE^{s}},
	\end{align*}
	which is acceptable.
	
	\medskip
	
	\emph{Case 3. The contribution of $\dh\nab h\nab h$ and $h\nab h\nab \dh$.} It suffices to prove that
	\begin{equation*}
	\lV \De^{-1}S_j(\dh\nab h\nab h+h\nab h\nab \dh)\rV_{Y^{d/2-1-\de,\si+2}}\lesssim s_j\lV \SS\rV_{\bEE^{s}}^2.
	\end{equation*}
	For the low-frequency part, By Bernstein's inequality and $d\geq 4$ we have
	\begin{align*}
	&\lV\De^{-1}(\dh\nab h\nab h)_{\leq 0}\rV_{Y^{d/2-1-\de,\si+2}}\\
	\lesssim &\lV (\dh\nab h\nab h)_{\leq 0}\rV_{l^1_0L^{\infty}L^2}\\
	\lesssim & \lV \dh_{\leq 0}\rV_{L^{\infty}L^{\infty}}\lV (\nab h\nab h)_{\leq 0}\rV_{l^1_0L^{\infty}L^2}+\sum_{j>0}\lV \dh_j\rV_{l^2_0L^{\infty}L^2}\lV (\nab h\nab h)_j\rV_{l^2_0L^{\infty}L^2}\\
	\lesssim & s_0\lV \SS\rV_{\bEE^{s}}^2.
	\end{align*}
	For the high-frequency part, by Bernstein's inequality we also have
	\begin{align*}
	2^{(\si+2)j}\lV\De^{-1}(\dh\nab h\nab h)_j\rV_{Y_j}
	\lesssim 2^{\si j}\lV (\dh\nab h\nab h)_j\rV_{l^1_jL^{\infty}L^2}
	\lesssim  s_j\lV \SS\rV_{\bEE^{s}}^2.
	\end{align*}
	Thus this completes the proof of $Y^{d/2-1-\de,\si+2}$ bound.
\end{proof}

\section{Multilinear and nonlinear estimates} \label{Sec-mutilinear}
This section contains our main multilinear estimates which are needed for the analysis  of the Schr\"{o}dinger equation in (\ref{mdf-Shr-sys-2}). We begin with the following low-high bilinear estimates of $\nab h\nab \psi$. 

\begin{lemma}   \label{bilinear-est}
	Let $s>\frac{d}{2}$, $d\geq 2$ and $k\in \N$. Suppose that $\nab a(x)\lesssim \<x\>^{-1}$, $ h\in  \bY^{\si+2}$ and $\psi_k\in l^2X^s$. Then for $-s\leq  \si\leq s$ we have
	\begin{align}          \label{qdt_d h<k dpsi}
	&\lV \nab h_{\leq k}\cdot\nab\psi_k\rV_{l^2N^{\si}}\lesssim \min\{ \lV  h\rV_{\bY^{\si+2}}\lV \psi_k\rV_{l^2X^s},\lV h\rV_{\bY^{s+2}}\lV\psi_k\rV_{l^2X^{\si}}\},\\\label{h<0-naba}
	&\lV h_{\leq k}\nab a\nab\psi_k\rV_{l^2N^{\si}}\lesssim \min\{ \lV  h\rV_{\bY^{\si+2}}\lV \psi_k\rV_{l^2X^s},\lV h\rV_{\bY^{s+2}}\lV\psi_k\rV_{l^2X^{\si}}\}. 
	\end{align}
	In addition, if $-s\leq \si\leq s-1$ then we have
	\begin{align}\label{qdt_h-d2psi}
	\lV h_{\leq k}\nab^2\psi_k\rV_{l^2N^{\si}}\lesssim \lV  h\rV_{\bY^{\si+2}}\lV \psi_k\rV_{l^2X^{s}}.
	\end{align}
\end{lemma}
\begin{proof}
	\emph{a) The estimates (\ref{qdt_d h<k dpsi}) and (\ref{qdt_h-d2psi}).} The proof of second bound (\ref{qdt_h-d2psi}) is similar to the first, so we only prove the first bound in detail. By duality, it suffices to estimate 
	\begin{align*}       
	I_j:=\<\nab P_j h\nab \psi_k,z_k\>,\quad j\leq k,\ j\in\Z,\ k\in \mathbb{N},
	\end{align*}
	for any $z_k:=S_kz\in l_k^2X_k$ with
	$\lV z_k\rV_{l^2_kX_k}\leq 1$.	For $I_j$ and any decomposition $P_jh=\sum_{l\geq |j|}h_{j,l}$, by duality and Bernstein inequality, we have
	\begin{align*}
	I_j\lesssim & \sum_{l\geq |j|} \sup_{\lV z_k\rV_{l^2_k X_k}\leq 1}\< \nab h_{j,l}\nab\psi_k,z_k\>\\
	\lesssim & \sum_{l\geq |j|} \sup_{\lV z_k\rV_{l^2_k X_k}\leq 1}\lV \nab h_{j,l}\rV_{l^1_l L^{\infty}L^{\infty}} \lV \nab \psi_k\rV_{l^{\infty}_lL^2L^2}\lV z_k\rV_{l^{\infty}_lL^2L^2}\\
	\lesssim &\sum_{l\geq |j|} 2^l\lV \nab h_{j,l}\rV_{l^1_l L^{\infty}L^{\infty}} \lV  \psi_k\rV_{X_k}\\
	\lesssim & 2^{(\frac{d}{2}+1)j+|j|}\sum_{l\geq |j|}2^{l-|j|}\lV  h_{j,l}\rV_{l^1_lL^{\infty}L^2}\lV \psi_k\rV_{X_k}.
	\end{align*}
	Then taking the infimum over the decomposition of $P_jh$ and incorporating the summation over $j$ yield
	\begin{align*}
	\sum_{j\leq k} 2^{\si k}I_j\lesssim & \lV h\rV_{\bY^{d/2+2+\ep}}\lV\psi_k\rV_{X^{\si}},
	\end{align*}
	for any $\ep>0$. If $-s\leq \si\leq d/2$, we also have
	\begin{align*}
	\sum_{j\leq k} 2^{\si k}I_j\lesssim & \sum_{j\leq 0}2^{dj/2}\lV P_j h\rV_{Y_j}\lV\psi_k\rV_{X^{\si}}+\sum_{j> 0}2^{(d/2+\ep-\si)(j-k)}2^{(\si+2)j}\lV P_j h\rV_{Y_j}\lV\psi_k\rV_{X^{s}}\\
	\lesssim &\lV h\rV_{\bY^{\si+2}}\lV\psi_k\rV_{X^{s}}.
	\end{align*}
	Thus the bound (\ref{qdt_d h<k dpsi}) follows.

	\emph{Estimate (\ref{h<0-naba}).} By duality, it suffices to bound
	\begin{equation*}
	II_j= \< P_jh\nab a\nab\psi_k,z_k\>,\ \ j\leq k,\ j\in\Z,
	\end{equation*}
	for any $z_k\in l^2_k X_k$ with $\lV z_k\rV_{l^2_k X_k}\leq 1$. For any decomposition $P_jh=\sum_{l\geq |j|}h_{j,l}$, by $|\nab a|(x)\lesssim \<x\>^{-1}$, we consider the two cases $|x|\geq 2^{j/2}$ and $|x|<2^{j/2}$ respectively and then obtain
	\begin{align*}
	II_j\lesssim & \sum_{l\geq |j|} \sup_{\lV z_k\rV_{l^2_k X_k}\leq 1}\<  h_{j,l}\<x\>^{-1}\mathbf{1}_{\leq 2^{l/2}}(x)\nab\psi_k,z_k\>+\sum_{l\geq |j|} \sup_{\lV z_k\rV_{l^2_k X_k}\leq 1}\<  h_{j,l}\<x\>^{-1}\mathbf{1}_{>2^{l/2}}(x)\nab\psi_k,z_k\>\\
	=& II_{j1}+II_{j2}.
	\end{align*}
	The first term is bounded by
	\begin{align*}
	II_{j1}\lesssim & \sum_{l\geq |j|} \sup_{\lV z_k\rV_{l^2_k X_k}\leq 1}\lV  h_{j,l}\rV_{l^1_l L^{\infty}L^{\infty}} \lV \nab \psi_k\rV_{l^{\infty}_{l/2}L^2L^2}\lV z_k\rV_{l^{\infty}_{l/2}L^2L^2}\\
	\lesssim &\sum_{l\geq |j|} 2^{l/2}\lV  h_{j,l}\rV_{l^1_l L^{\infty}L^{\infty}} \lV  \psi_k\rV_{X_k}\\
	\lesssim & 2^{dj/2+|j|/2}\sum_{l\geq |j|} 2^{l-|j|}\lV  h_{j,l}\rV_{l^1_l L^{\infty}L^2} \lV  \psi_k\rV_{X_k}
	\end{align*}
	The second term is bounded by
	\begin{align*}
	II_{j2}\lesssim & \sum_{l\geq |j|} 2^{-l/2}\sup_{\lV z_k\rV_{l^2_k X_k}\leq 1}\lV  h_{j,l}\rV_{l^1_l L^{\infty}L^{\infty}} \lV \nab \psi_k\rV_{l^{\infty}_{l}L^2L^2}\lV z_k\rV_{l^{\infty}_{l}L^2L^2}\\
	\lesssim &\sum_{l\geq |j|} 2^{l/2}\lV  h_{j,l}\rV_{l^1_l L^{\infty}L^{\infty}} \lV  \psi_k\rV_{X_k}\\
	\lesssim & 2^{dj/2+|j|/2}\sum_{l\geq |j|} 2^{l-|j|}\lV  h_{j,l}\rV_{l^1_l L^{\infty}L^2} \lV  \psi_k\rV_{X_k}.
	\end{align*}
	Then we obtain
	\begin{align*}
	\sum_{j\leq k} 2^{\si k}II_j\lesssim & (\sum_{j\leq 0}2^{(d-1)j/2}\lV P_j h\rV_{Y_j}+\sum_{j>0}2^{(d+1)j/2}\lV h_j\rV_{Y_j})\lV\psi_k\rV_{X^{\si}}\\
	\lesssim &\min\{ \lV  h\rV_{\bY^{\si+2}}\lV \psi_k\rV_{l^2X^s},\lV h\rV_{\bY^{s+2}}\lV\psi_k\rV_{l^2X^{\si}}\}.
	\end{align*}
	Thus the bound (\ref{h<0-naba}) follows. 
\end{proof}

We  next prove the remaining  bilinear estimates and trilinear estimates.
\begin{prop}[Nonlinear estimates]      \label{Non-Est}
	Let $ s>\frac{d}{2}$, $-s\leq \si\leq s$ and $d\geq 2$, assume that $p_k$ and $s_k$ are admissible frequency envelopes for $\psi\in l^2X^{\si}$, $\SS\in \EE^{\si}$ respectively. Then we have
	\begin{gather}    		\label{qdt-gdpsi_hl&hh}
		\lV S_k\nab(h_{\geq k-4}\nab\psi)\rV_{l^2N^{\si}}\lesssim  \min\{s_k\lV \psi\rV_{l^2X^{s}},p_k\lV  h\rV_{Z^{1,s+2}}\},\\ \label{qdt-Ag-dpsi}
		\lV S_k(A_{\geq k-4}\nab\psi)\rV_{l^2N^{\si}}\lesssim  \min\{s_k\lV \psi\rV_{l^2X^{s}},p_k\lV  A\rV_{Z^{1,s+1}}\},\\\label{cb-B}
		\lV S_k(B\psi)\rV_{l^2N^{\si}}\lesssim \min\{s_k\lV \psi\rV_{l^2 X^{s}},p_k\lV B\rV_{Z^{1,s}} \},\\\label{cb-AA}
		\lV S_k(A^2\psi)\rV_{l^2N^{\si}}\lesssim \min\{s_k\lV  A\rV_{Z^{1,s}}\lV\psi\rV_{l^2X^{s}}, p_k\lV  A\rV_{Z^{1,s}}^2\},\\
		\label{cb_lam2-psi}
		\lV S_k(\lambda^3)\rV_{l^2N^{\si}}\lesssim s_k \lV  \lambda\rV_{Z^{0,s}}^2.
	\end{gather}
	If $-s\leq \si\leq s-1$, then
	\begin{gather}\label{qdt-Ag-dpsi_lh}
	\lV S_k(A_{< k-4}\nab\psi)\rV_{l^2N^{\si}}\lesssim s_k\lV \psi\rV_{l^2X^{s}}.
	\end{gather}
\end{prop}
\begin{proof}
	
	We first prove (\ref{qdt-gdpsi_hl&hh} and (\ref{qdt-Ag-dpsi}). These two bounds are proved similarly, here we only prove the first bound in detail. For the high-low case, by (\ref{Bern-ineq}) we have
	\begin{align*}
		\sum_{j_2\leq k+C}\lV S_k \nab( h_{j_1}\nab \psi_{j_2})\rV_{l^2N^{\si}}\lesssim &\sum_{j_1=k+O(1),j_2\leq k+C} 2^{(\si+1) k}\lV  h_{j_1}\rV_{l_k^2L^2L^2}\lV \nab\psi_{j_2}\rV_{l^2_kL^{\infty}L^{\infty}}\\
		\lesssim &\sum_{j_2\leq k+C} 2^{\si k+(d/2+1)j_2}\lV \nab h_k\rV_{L^2L^2}\lV \psi_{j_2}\rV_{l_{j_2}^2L^{\infty}L^2}\\
		\lesssim &\min\{s_k\lV \psi\rV_{l^2X^{s}},p_k\lV  h\rV_{Z^{1,s+2}}\}.
	\end{align*}
	For the high-high case, when $\si+d/2+1>0$ we have
	\begin{align*}
		&\sum_{j_1= j_2+O(1),j_1>k}\lV S_k \nab( h_{j_1}\nab \psi_{j_2})\rV_{l^2N^{\si}}\\
		\lesssim &\sum_{j_1= j_2+O(1),j_1>k}2^{(\si+1) k+dk/2}\lV S_k ( h_{j_1}\nab \psi_{j_2})\rV_{L^2L^1}\\
		\lesssim & \sum_{j_1= j_2+O(1),j_1>k}2^{(\si+1+d/2)(k-j_1)+(\si+2+d/2)j_1}\lV  h_{j_1}\rV_{L^2L^2}\lV \psi_{j_2}\rV_{L^{\infty}L^2}\\
		\lesssim & \min\{s_k\lV \psi\rV_{l^2X^{s}},p_k\lV  h\rV_{Z^{1,s+2}}\},
	\end{align*}
	and when $\si+d/2+1\leq 0$ we have
	\begin{align*}
		&\sum_{j_1= j_2+O(1),j_1>k}\lV S_k \nab( h_{j_1}\nab \psi_{j_2})\rV_{l^2N^{\si}}\\
		\lesssim & \sum_{j_1= j_2+O(1),j_1>k}2^{(\si+1+d/2-\ep)k+(\ep+1) j_1}\lV  h_{j_1}\rV_{L^2L^2}\lV \psi_{j_2}\rV_{L^{\infty}L^2}\\
		\lesssim & \min\{s_k\lV \psi\rV_{l^2X^{s}},p_k\lV  h\rV_{Z^{1,s+2}}\},
	\end{align*}
	
	Next, we prove the bounds (\ref{cb-B})-(\ref{cb_lam2-psi}). These bounds can be estimated similarly, we only prove (\ref{cb-B}) in detail. Indeed, for (\ref{cb-B}), by duality we have
	\begin{align*}
	\lV S_k(B\psi)\rV_{l^2N^{\si}}\lesssim 2^{\si k}\lV S_k(B\psi)\rV_{L^2L^2}.
	\end{align*}
	Then using Littlewood-Paley dichotomy to divide this into low-high, high-low and high-high cases. For the low-high case, by Sobolev embedding we have
	\begin{align*}
	2^{\si k}\lV S_k(B_{<k}\psi_k)\rV_{L^2L^2}\lesssim & \lV B_{<k}\rV_{L^{\infty}L^{\infty}}2^{\si k}\lV \psi_k\rV_{L^2L^2}\lesssim p_k\lV B\rV_{Z^{1,d/2+\ep}} ,
	\end{align*}
	for any $\ep>0$. We also have for $\si\leq d/2$
	\begin{align*}
	2^{\si k}\lV S_k(B_{<k}\psi_k)\rV_{L^2L^2}\lesssim & \sum_{0\leq l<k}2^{(d/2+\delta-\si)(l-k)}\lV \nab^{\si}B_l\rV_{L^{\infty}L^2}2^{(d/2+\delta) k}\lV \psi_k\rV_{L^2L^2}\\
	\lesssim &s_k \lV\psi\rV_{l^2X^{s}}.
	\end{align*}
	The high-low case can be estimated similarly. For the high-high case, by Sobolev embedding when $\si+d/2\geq 0$ we have
	\begin{align*}
	2^{\si k}\lV S_k(B_l\psi_l)\rV_{L^2L^2}\lesssim &\sum_{l>k} 2^{(\si+d/2+\ep)(k-l)} 2^{(\si+d/2+\ep)l} \lV B_l\rV_{L^{\infty}L^2}\lV \psi_l\rV_{L^2L^2}\\
	\lesssim & \min\{s_k\lV \psi\rV_{l^2 X^{s}},p_k\lV B\rV_{Z^{1,s}} \},
	\end{align*}
	and when $\si+d/2<0$ we have
	\begin{align*}
	2^{\si k}\lV S_k(B_l\psi_l)\rV_{L^2L^2}\lesssim &\sum_{l>k} 2^{(\si+d/2)k}  \lV B_l\rV_{L^{\infty}L^2}\lV \psi_l\rV_{L^2L^2}\\
	\lesssim & \min\{s_k\lV \psi\rV_{l^2 X^{s}},p_k\lV B\rV_{Z^{1,s}} \},
	\end{align*}
	These imply the bound (\ref{cb-B}). 
	
	Finally, we prove the bound (\ref{qdt-Ag-dpsi_lh}). If $\si>d/2-1$, by duality and Sobolev embedding, we have
	\begin{align*}
	2^{\si k}\lV A_{<k}\nab\psi_k\rV_{L^2L^2}\lesssim & s_k\lV \psi\rV_{l^2X^{s}}.
	\end{align*}
	If $\si\leq d/2-1$, we have
	\begin{align*}
	2^{\si k}\lV A_{<k}\nab\psi_k\rV_{L^2L^2}\lesssim &\sum_{0\leq l<k} 2^{(d/2-1-\si+\delta)(l-k)}\lV \nab^{\si+1}A_l\rV_{L^{\infty}L^2} 2^{(d/2+\delta)k} \lV \psi_k\rV_{L^2L^2}\\
	\lesssim & s_k\lV \psi\rV_{l^2X^{s}}.
	\end{align*}
	Then the bound (\ref{qdt-Ag-dpsi_lh}) follows. Hence this completes the proof of the lemma.
\end{proof}

We shall also require the following bounds on commutators.

\begin{prop}[Commutator bounds]\label{Comm-est-Lemma} Let $s>\frac{d}{2},d\geq 2$. Let $m(D)$
be a multiplier with symbol $m\in S^0$. Assume $ h\in \bY^{s+2}$, $ A\in Z^{1,s+1}$ and $\psi_k \in l^2X^s$, frequency localized at frequency $2^k$. If $-s\leq \si \leq s$ we have
	\begin{gather}          \label{Comm-bd}
	\lV \nab[S_{<k-4}h,m(D)]\nab\psi_k\rV_{l^2N^{\si}}\lesssim \min\{ \lV  h\rV_{\bY^{\si+2}}\lV \psi_k\rV_{l^2X^s},\lV h\rV_{\bY^{s+2}}\lV\psi_k\rV_{l^2X^{\si}}\},\\
	\label{Com_Ag-dpsi}
	\lV [S_k,A_{<k-4}]\nab\psi_k\rV_{l^2N^{\si}}\lesssim \min\{\lV  A\rV_{Z^{1,s+1}}\lV \psi_k\rV_{l^2X^{\si}},\lV  A\rV_{Z^{1,\si+1}}\lV \psi_k\rV_{l^2X^{s}}\}.
	\end{gather}
\end{prop}
\begin{proof}
	First we estimate (\ref{Comm-bd}). In \cite[Proposition 3.2]{MMT3}, it was shown that 
	\begin{equation*}
	\nab[S_{<k-4}g,m(D)]\nab S_k \psi=L(\nab S_{<k-4}g,\nab S_k \psi),
	\end{equation*}
	where $L$ is a translation invariant operator satisfying
	\begin{equation*}
	L(f,g)(x)=\int f(x+y)g(x+z)\tilde{m}(y+z)dydz,\ \ \ \tilde{m}\in L^1.
	\end{equation*}
	Given this representation, as we are working in translation-invariant spaces, by (\ref{qdt_d h<k dpsi}) the bound (\ref{Comm-bd}) follows.

	Next, for the bound (\ref{Com_Ag-dpsi}). Since 
	\begin{equation*}
	[S_k,A_{<k}]\nab\psi=\int_0^1 \int 2^{kd}\check{\varphi}(2^k y)2^k y \nab A_{<k}(x-sy) 2^{-k}\nab\psi_{[k-3,k+3]}(x-y)dyds,
	\end{equation*}
	By translation-invariance and the similar argument to (\ref{cb-B}), the bound (\ref{Com_Ag-dpsi}) follows. This completes the proof of the lemma.
\end{proof}

\section{Local energy decay and the linearized problem}\label{Sec-LED}

In this section, we consider a linear  Schr\"odinger equation 
\begin{equation}\label{Lin-eq0}
\left\{
\begin{aligned}      
&i\d_t \psi+\d_{\al}g^{\al\be}\d_{\be}\psi+2i A^{\al}\d_{\al}\psi
 + B\psi=F,\\
&\psi(0)=\psi_0,
\end{aligned}\right.
\end{equation}
and, under suitable assumptions on the coefficients,
we prove that the solution satisfies suitable energy and local energy bounds.

\subsection{The linear paradifferential Schr\"odinger flow}
As an intermediate step, here we  prove energy and local energy bounds for a frequency localized linear paradifferential Schr\"odinger equation
\begin{equation}        \label{LinSch}
	i\d_t\psi_k+\d_{\al}(g^{\al\be}_{<k-4}\d_{\be}\psi_k)+2iA^{\al}_{<k-4}\d_{\al}\psi_k=f_k.
\end{equation}
We begin with the energy estimates, which are fairly standard:

\begin{lemma}[Energy-type estimate]
	Let $d\geq 2$, $\psi_k$ solves the equation (\ref{LinSch}) with initial data $\psi_k(0)$ in the time interval $[0,1]$. For a fixed $s>\frac{d}{2}$, assume that $A\in Z^{1,s+1}$, $\psi_k\in l^2_kX_k$, $f_{1k}\in N$ and $f_{2k}\in L^1L^2$, where $f_k=f_{1k}+f_{2k}$. Then we have
	\begin{equation}       \label{Eng-Estm}
		\begin{aligned}
			\lV \psi_k\rV_{L^{\infty}_tL^2_x}^2\lesssim & \lV\psi_k(0)\rV_{L^2}^2+\lV  A\rV_{Z^{1,s+1}}\lV \psi_k\rV_{X_k}^2	+\lV\psi_k\rV_{X_k}\lV f_{1k}\rV_{N_k}\\
			&+\lV \psi_k\rV_{L^{\infty}L^2}\lV f_{2k}\rV_{L^1L^2}.
		\end{aligned}
	\end{equation}
\end{lemma}
\begin{proof}
	By (\ref{LinSch}), we have
	\begin{align*}
		\frac{1}{2}\frac{d}{dt}\lV\psi_k\rV_{L^2}^2=&\Re\<\psi_k,\d_t \psi_k\>\\
		=&\Re \<\psi_k,i\d_{\al}g^{\al\be}_{<k-4}\d_{\be}\psi_k-2A^{\al}_{<k-4}\d_{\al}\psi_k-if_k\>\\
		=& -\Re \<\d_{\al}\psi_k,ig^{\al\be}_{<k-4}\d_{\be}\psi_k\>
		-\Re\int_{\R^d} A^{\al}_{<k-4}\d_{\al}|\psi_k|^2 dx-\Re\<\psi_k,if_k\>\\
		=&\Re \int_{\R^d} \d_{\al}A^{\al}_{<k-4}|\psi_k|^2 dx-\Re \<\psi_k,if_k\>,
	\end{align*}
	and notice that for each $t\in[0,1]$ we have by duality and Sobolev embedding
	\begin{align*}
		\lV\psi_k(t)\rV_{L^2}^2\lesssim & \lV\psi_k(0)\rV_{L^2}^2+\int_0^1\int_{\R^d} |\d_{\al}A^{\al}_{<k-4}||\psi_k|^2 dxdt
		+\lV\psi_k\rV_{X_k}\lV f_{1k}\rV_{N_k}\\
		&+\lV \psi_k\rV_{L^{\infty}L^2}\lV f_{2k}\rV_{L^1L^2}\\
		\lesssim & \lV\psi_k(0)\rV_{L^2}^2+\lV A\rV_{Z^{1,s+1}}\lV \psi_k\rV_{X_k}^2\\
		&+\lV\psi_k\rV_{X_k}\lV f_{1k}\rV_{N_k}+\lV \psi_k\rV_{L^{\infty}L^2}\lV f_{2k}\rV_{L^1L^2}
	\end{align*}
	We take the supremum over $t$ on the left hand side and the conclusion follows.
\end{proof}

Next, we prove the  main result of this section, namely the local energy estimates
for solutions to \eqref{LinSch}:

\begin{prop} [Local energy decay]     \label{Local-Energy-Decay}
	Let $d\geq 3$, assume that the coefficients $g^{\al\be}=\de^{\al\be}+h^{\al\be}$ and $A^{\al}$ in (\ref{LinSch}) satisfy
	\begin{equation}\label{small-ass}
	\lV h\rV_{\bY^{s+2}},\ \lV A\rV_{Z^{1,s+1}}\ll 1
	\end{equation}
	for some $s>\frac{d}{2}$. Let $\psi_k$ be a solution to (\ref{LinSch}) which is localized at frequency $2^k$. Then the following estimate holds:
	\begin{equation}      \label{energy-decay}
		\lV\psi_k\rV_{l^2_k X_k}\lesssim \lV\psi_{0k}\rV_{L^2}+\lV f_k\rV_{l^2_kN_k}
	\end{equation}
\end{prop}
\begin{proof}
	The proof is closely related to that given in \cite{MMT3,MMT4}. However, here we are able to 
	relax the assumptions both on the metric $g$ and on the magnetic potential $A$. In the latter 
	case,  unlike in \cite{MMT3,MMT4},  we treat the magnetic term $2i A^{\al}_{<k-4}\d_{\al}\psi_k$ as a part of the linear equation, which allows us to avoid bilinear estimates for this term and use only the bound for $A$ in $\bZ^{1,s+1}$. 	
	
	As an intermediate step in the proof, we will establish a local energy decay bound 
in a cube  $Q\in \QQ_l$ with $0\leq l\leq k$:	
\begin{equation}   \label{LED-pre0}
	\begin{aligned}      
	2^{k-l}\lV \psi_k\rV_{L^2L^2([0,1]\times Q)}^2\lesssim &\ \lV \psi_k\rV_{L^{\infty}L^2}^2+
	\|f_k\|_{N_k} \|\psi_k \|_{X_k}  \\ 
	&+(2^{-k}+\lV A\rV_{Z^{1,s+1}}+\lV  h\rV_{\bY^{s+2}})\lV \psi_k\rV_{l^2_kX_k}^2.
	\end{aligned}
	\end{equation}
	
The proof of this bound is based on a positive commutator argument using a well chosen 
multiplier $\MM$. This will be first-order differential operator with smooth coefficients which are localized at frequency $\lesssim 1$. Precisely, we will use a multiplier $\MM$ which is 
a  sef-adjoint differential operator having the form 
\begin{equation}           \label{M-def}
		i2^k\MM=a^{\al}(x)\d_{\al}+\d_{\al}a^{\al}(x)
\end{equation}
with uniform bounds on $a$ and its derivatives.

	Before proving \ref{energy-decay}, we need the following lemma which is used to dismiss the $(g-I)$ and the $A$ contributions to the commutator $[\d_{\al}g^{\al\be}\d_{\be},\MM]$. 
	
	\
	
	\begin{lemma}\label{Comm_g-I}
		Let $s>\frac{d}{2}$ and $d\geq 3$, assume that $ h\in \bY^{s+2}$, $ A\in Z^{1,s+1}$ and $\psi\in l^2_kX_k$, let $\MM$ be as (\ref{M-def}). Then we have
		\begin{gather}           \label{Comm-M}
		\int_0^1\< [\d_{\al}h^{\al\be}_{\leq k}\d_{\be},\MM]\psi_k,\psi_k\>ds\lesssim \lV  h\rV_{\bY^{s+2}}\lV \psi_k\rV_{l^2_kX_k}^2,\\\label{non-(Ag)dpsi,Mpsi}
		\int_0^1\Re\<A^{\al}_{<k-4}\d_{\al}\psi_k,\MM\psi_k \>ds\lesssim \lV A\rV_{Z^{1,s+1}}\lV\psi_k\rV_{X_k}^2.
		\end{gather}
	\end{lemma}
	\begin{proof}[Proof of Lemma \ref{Comm_g-I}]
		By (\ref{M-def}) and directly computations, we get
		\begin{align*}
		[\d_{\al}h^{\al\be}\d_{\be},\MM]\approx 2^{-k}[\nab(h\nab a+a\nab h)\nab+\nab h\nab^2 a+h\nab^3 a].
		\end{align*} 
		Then it suffices to estimate
		\begin{align*}
		2^{-k}\int_0^1 \<(h_{\leq k}\nab a+a\nab h_{\leq k})\nab\psi_k,\nab\psi_k\>ds+2^{-k}\int_0^1\< (\nab h_{\leq k}\nab^2 a+h_{\leq k}\nab^3 a)\psi_k,\psi_k\>ds
		\end{align*}
		The first integral is estimated by (\ref{qdt_d h<k dpsi}) and (\ref{h<0-naba}). Using Sobolev embedding, the second integral is bounded by
		\begin{align*}
		2^{-k}\int_0^1 \<(\nab h_{\leq k}+h_{\leq k})\psi_k,\psi_k\>ds
		\lesssim \lV\<\nab\> h_{\leq k}\rV_{L^{\infty}}2^{-k}\lV \psi_k\rV_{L^2L^2}^2
		\lesssim \lV \nab h\rV_{L^{\infty}H^s}\lV \psi_k\rV_{l^2_kX_k}^2.
		\end{align*}
		Hence, the bound (\ref{Comm-M}) follows.
		
		For the second bound (\ref{non-(Ag)dpsi,Mpsi}), by (\ref{M-def}) and integration by parts we rewrite the following term as
		\begin{align*}
		&\Re \< A^{\al}\d_{\al}\psi,i\sum_{\be=1}^d(a_{\be}\d_{\be}+\d_{\be}a_{\be})\psi\>\\
		=& \Re \sum_{\be=1}^d\int_{\R^d} \Big[ i\d_{\al} (\bar{\psi}A^{\al}a_{\be}\d_{\be}\psi)-i\bar{\psi}\d_{\al}A^{\al}a_{\be}\d_{\be}\psi-i\bar{\psi}A^{\al}\d_{\al}a_{\be} \d_{\be}\psi-i\bar{\psi}A^{\al}a_{\be}\d_{\al\be}^2\psi\\
		&+i\d_{\be}(A^{\al}\d_{\al}\bar{\psi}a_{\be}\psi)-i\d_{\be}A^{\al}\d_{\al}\bar{\psi}a_{\be}\psi-iA^{\al}\d_{\al\be}^2\bar{\psi}a_{\be}\psi\Big] dx\\
		\approx& \int_{\R^d} \<\nab\>A\psi\nab\psi dx.
		\end{align*}
		Then we bound the left-hand side of (\ref{non-(Ag)dpsi,Mpsi}) by
		\begin{align*}
		\int_0^1\Re\<A^{\al}_{<k-4}\d_{\al}\psi_k,\MM\psi_k \>ds\lesssim &2^{-k} \int_0^1\int_{\R^d} |\<\nab\> A_{<k} \psi_k \nab \psi_k |dxds\\
		\lesssim &\lV \nab A\rV_{L^{\infty}H^s}\lV\psi_k\rV_{L^2L^2}^2.
		\end{align*}
		This implies the bound (\ref{non-(Ag)dpsi,Mpsi}), and hence completes the proof of the lemma.
	\end{proof}
	
	\
	
	Returning to the proof of \eqref{LED-pre0}, for the  self-adjoint multiplier $\MM$ we compute
	\begin{align*}
		\frac{d}{dt}\<\psi_k,\MM\psi_k\>=&2\Re\<\d_t\psi_k,\MM\psi_k\>\\
		=&2\Re\< i\d_{\al}(g^{\al\be}_{<k-4}\d_{\be}\psi_k)-2A^{\al}_{<k-4}\d_{\al}\psi_k-if_k,\MM\psi_k \>\\
		=&i\< [-\d_{\al}g^{\al\be}_{<k-4}\d_{\be},\MM]\psi_k,\psi_k\>
		+2\Re\<-2A^{\al}_{<k-4}\d_{\al}\psi_k-if_k,\MM\psi_k \>
	\end{align*}
	We then use the multiplier $\MM$ as in \cite{MMT3,MMT4} so that the following three properties hold:
	\begin{enumerate}
	\item[(1)] Boundedness on frequency $2^k$ localized functions,
	\[
	\lV\MM u\rV_{L^2_x}\lesssim \lV u\rV_{L^2_x}.
	\]
	
	\item[(2)] Boundedness in $X$, 
	\[
	\lV\MM u\rV_{X}\lesssim \lV u\rV_{X}.
	\]
	
	\item[(3)] Positive commutator,
	\[
	i\< [-\d_{\al}g^{\al\be}_{<k-4}\d_{\be},\MM]u,u\>\gtrsim 2^{k-l}\lV u\rV^2_{L^2_{t,x}([0,1]\times Q)}-O(2^{-k}+\lV h\rV_{\bY^{s+2}})\lV u\rV_{l^2_kX_k}^2.
	\]
\end{enumerate}
	
	\noindent If these three properties hold for $u=\psi_k$, then by (\ref{non-(Ag)dpsi,Mpsi}) and (\ref{small-ass}) the bound (\ref{LED-pre0}) follows.

	We first do this when the Fourier transform of the solution $\psi_k$ is restricted to a small angle
	\begin{equation}       \label{xi1-near}
		\text{supp} \widehat{\psi}_k\subset \{|\xi|\lesssim \xi_1\}.
	\end{equation}
	Without loss of generality due to translation invariance, $Q=\{|x_j|\leq 2^l:j=1,\ldots,d\}$, and we set $m$ to be a smooth, bounded, increasing function such that $m'(s)=\varphi^2(s)$ where $\varphi$ is a Schwartz function localized at frequencies $\lesssim 1$, and $\varphi\approx 1$ for $|s|\leq 1$. We rescale $m$ and set $m_l(s)=m(2^{-l}s)$. Then, we fix
	\[
	\MM=\frac{1}{i2^k} (m_l(x_1)\d_1+\d_1 m_l(x_1)).
	\]
	
	The properties $(1)$ and $(2)$ are immediate due to the frequency localization of $u=\psi_k$ and $m_l$ as well as the boundedness of $m_l$. By (\ref{Comm-M}) it suffices to consider the property $(3)$ for the operator 
	\[
	-\De=-\d_{\al}g^{\al\be}_{<k-4}\d_{\be}+\d_{\al}h^{\al\be}_{<k-4}\d_{\be}.
	\]
	This yields
	\begin{align*}
		i2^k[-\Delta,\MM]=-2^{-l+2}\d_1 \varphi^2(2^{-l}x_1)\d_1+O(1),
	\end{align*}
	and hence
	\begin{equation*}
		i2^k\<[-\Delta,\MM]\psi_k,\psi_k\>=2^{-l+2}\lV \varphi(2^{-l}x_1)\d_1\psi_k\rV_{L^2L^2}^2+O(\lV\psi_k\rV_{L^2L^2}^2)
	\end{equation*}
	Utilizing our assumption (\ref{xi1-near}), it follows that
	\begin{equation*}
		2^{k-l}\lV \varphi(2^{-l}x_1)\psi_k\rV_{L^2L^2}^2\lesssim i\< [-\Delta,\MM]\psi_k,\psi_k\>+2^{-k}O(\lV\psi_k\rV_{L^2L^2}^2)
	\end{equation*}
	which yields $(3)$ when combined with (\ref{Comm-M}).
	
	We proceed to reduce the problem to the case when (\ref{xi1-near}) holds. We let $\{ \theta_j (\om) \}_{j=1}^d$ be a partition of unity,
	\begin{equation*}
		\sum_{j}\theta_j(\om)=1,\ \ \ \ \om\in\S^{d-1},
	\end{equation*}
	where $\theta_j(\om)$ is supported in a small angle about the $j$-th coordinate axis. Then, we can set $\psi_{k,j}=\Theta_{k,j}\psi_k$ where
	\begin{equation*}
		\FF\Theta_{k,j}\psi=\theta_j(\frac{\xi}{|\xi|})\sum_{k-1\leq l\leq k+1}\varphi_l(\xi)\widehat{\psi}(t,\xi).
	\end{equation*}
	We see that 
	\begin{align*}
		&(i\d_t+\d_{\al}g^{\al\be}_{<k-4}\d_{\be})\psi_{k,j}+2iA^{\al}_{<k-4}\d_{\al}\psi_{k,j}\\
		=&\Theta_{k,j}f_k-\d_{\al}[\Theta_{k,j},g^{\al\be}_{<k-4}]\d_{\be}\psi_k-2i[\Theta_{k,j},A^{\al}_{\leq k-4}]\d_{\al}\psi_k.
	\end{align*}
	
	By applying $\MM$, suitably adapted to the correct coordinate axis, to $\psi_{k,j}$ and summing over $j$, we obtain
	\begin{align*}
		&2^{k-l}\lV \psi_k\rV_{L^2L^2([0,1]\times Q)}^2\\
		\lesssim &\ \lV\psi_k\rV_{L^{\infty}L^2}^2+\sum_{j=1}^d\int_0^1\<-\Theta_{k,j}f_k,\MM\psi_{k,j}\>ds\\
		&\ +\sum_{j=1}^d\int\< [\Theta_{k,j},\d_{\al}g^{\al\be}_{<k-4}\d_{\be}]\psi_k+[\Theta_{k,j},2iA^{\al}_{<k-4}]\d_{\al}\psi_k,\MM\psi_{k,j}\>ds\\
		&\ +(2^{-k}+\lV A\rV_{Z^{1,s+1}}+\lV  h\rV_{\bY^{s+2}})\lV \psi_k\rV_{l^2_kX_k}^2\\
		\lesssim &\  \lV\psi_k\rV_{L^{\infty}L^2}^2+\lV f_k\rV_{N_k}\lV \psi_k\rV_{X_k} 
		+(2^{-k}+\lV A\rV_{Z^{1,s+1}}+\lV  h\rV_{\bY^{s+2}})\lV \psi_k\rV_{l^2_kX_k}^2.
	\end{align*}
	The commutator is done via (\ref{Comm-bd}) and (\ref{Com_Ag-dpsi}). Then (\ref{LED-pre0}) follows.

	Next we use the bound \eqref{LED-pre0} to complete the proof of Proposition \ref{Local-Energy-Decay}. 	Taking the supremum in \eqref{LED-pre0} 
	over $Q\in\QQ_l$ and over $l$, we obtain
	\begin{equation*}
	\begin{aligned}
	2^k\lV \psi_k\rV_{X}^2\lesssim &\lV \psi_k\rV_{L^{\infty}L^2}^2+\lV f_{1k}\rV_{N_k}\lV \psi_k\rV_{X_k}+\lV f_{2k}\rV_{L^1L^2}\lV \psi_k\rV_{L^{\infty}L^2}\\
	&+(2^{-k}+\lV A\rV_{Z^{1,s+1}}+\lV  h\rV_{\bY^{s+2}})\lV \psi_k\rV_{l^2_kX_k}^2\\
	\lesssim &\lV \psi_k\rV_{L^{\infty}L^2}^2+\lV f_{1k}\rV_{N_k}\lV \psi_k\rV_{X_k}+\lV f_{2k}\rV_{L^1L^2}^2\\
	&+(2^{-k}+\lV A\rV_{Z^{1,s+1}}+\lV  h\rV_{\bY^{s+2}})\lV \psi_k\rV_{l^2_kX_k}^2.
	\end{aligned}
	\end{equation*}
	Combined with (\ref{Eng-Estm}), we get
	\begin{equation}          \label{psiXk}
	\begin{aligned}
	\lV \psi_k\rV_{X_k}^2\lesssim & \lV\psi_k(0)\rV_{L^2}^2
	+\lV f_{1k}\rV_{N_k}^2+\lV f_{2k}\rV_{L^1L^2}^2\\
	&+(2^{-k}+\lV A\rV_{Z^{1,s+1}}+\lV  h\rV_{\bY^{s+2}})\lV \psi_k\rV_{l^2_kX_k}^2.
	\end{aligned}
	\end{equation}

We now finish the proof by incorporating the summation over cubes. We let $\{\chi_Q\}$ denote a partition via functions which are localized to frequencies $\lesssim 1$ which are associated to cubes $Q$ of scale $M2^k$. We also assume that $|\nab^l\chi_Q|\lesssim (2^k M)^{-l}$, $l=1,2$. Thus,
	\begin{align*}
	&(i\d_t+\d_{\al}g^{\al\be}_{<k-4}\d_{\be})\chi_Q \psi_k+2iA^{\al}_{<k-4}\d_{\al}\chi_Q \psi_k\\
	=&\chi_Q f_k+[\d_{\al}g^{\al\be}_{<k-4}\d_{\be},\chi_Q]\psi_k+2iA^{\al}_{<k-4}\d_{\al}\chi_Q\cdot \psi_k
	\end{align*}
	Applying (\ref{Eng-Estm}) to $\chi_Q\psi_k$, we obtain
	\begin{align*}
	&\sum_Q \lV \chi_Q\psi_k\rV_{L^{\infty}L^2}^2\\
	\lesssim & \sum_Q \lV \chi_Q\psi_k(0)\rV_{L^2}^2
	+\lV A\rV_{Z^{1,s+1}}\sum_Q\lV \chi_Q\psi_k\rV_{X_k}^2\\
	&+(\sum_Q\lV \chi_Qf_k\rV_{N_k}^2)^{1/2}(\sum_Q\lV \chi_Q\psi_k\rV_{X_k}^2)^{1/2}\\
	&+\sum_Q\lV[\d_{\al}g^{\al\be}_{<k-4}\d_{\be},\chi_Q]\psi_k+2iA^{\al}_{<k-4}\d_{\al}\chi_Q\cdot \psi_k\rV_{L^1L^2}^2.
	\end{align*}
	But by (\ref{small-ass}) we have
	\begin{equation}\label{chiQ-comm-1}
	\begin{aligned}
	\sum_Q\lV[\nab g\nab,\chi_Q]\psi_k \rV_{L^1L^2}^2 \lesssim &\sum_Q\lV \nab g\cdot\nab \chi_Q\cdot\psi_k+g\nab(\nab\chi_Q\cdot\psi_k)\rV_{L^1L^2}^2\\
	\lesssim & (1+\lV h\rV_{Z^{1,s+2}}) M^{-2}\sum_Q \lV\chi_Q\psi_k\rV_{L^{\infty}L^2}^2,
	\end{aligned}
	\end{equation}
	and also
	\begin{align}\label{chiQ-comm-2}
	\sum_Q\lV2iA^{\al}_{<k-4}\d_{\al}\chi_Q\cdot \psi_k\rV_{L^1L^2}^2\lesssim (1+\lV A\rV_{Z^{1,s}}) M^{-2}\sum_Q \lV\chi_Q\psi_k\rV_{L^{\infty}L^2}^2.
	\end{align}
	For $M$ sufficiently large, we can bootstrap the commutator terms, and, after a straightforward transition to cubes of scale $2^k$ rather than $M2^k$, we observe that
	\begin{equation}      \label{energy-l^2L^infL^2}
	\begin{aligned}
	\lV \psi_k\rV_{l^2_kL^{\infty}L^2}^2
	\lesssim &  \lV \psi_k(0)\rV_{L^2}^2
	+\lV A\rV_{Z^{1,s+1}}\lV \psi_k\rV_{l_k^2X_k}^2
	+\lV f_k\rV_{l^2_kN_k}\lV \psi_k\rV_{l^2_kX_k}.
	\end{aligned}  
	\end{equation}
	
	We now apply (\ref{psiXk}) to $\chi_Q\psi_k$, and then by (\ref{chiQ-comm-1}) and (\ref{chiQ-comm-2}) we see that
	\begin{align*}
	\sum_Q \lV\chi_Q \psi_k\rV_{X_k}^2\lesssim & \lV\psi_k(0)\rV_{L^2}^2+\sum_Q\lV\chi_Q f_k\rV_{N_k}^2+M^{-2}\sum_Q\lV \chi_Q\psi_k\rV_{X_k}^2\\
	&+(2^{-k}+\lV h\rV_{\bY^{s+2}}+\lV  A\rV_{Z^{1,s+1}})\sum_Q\lV \chi_Q\psi_k\rV_{l^2_kX_k}^2.
	\end{align*}
	For $M\gg 1$, we have
	\begin{align*}
	M^{-d}\lV \psi_k\rV_{l^2_kX_k}^2\lesssim & \lV\psi_k(0)\rV_{L^2}^2+\lV f_k\rV_{l^2_kN_k}^2
	+(2^{-k}+\lV h\rV_{\bY^{s+2}}+\lV A\rV_{Z^{1,s+1}})\lV \psi_k\rV_{l^2_kX_k}^2.
	\end{align*}
	By (\ref{small-ass}), for $k$ sufficiently large (depending on $M$), we may absorb the second and the last terms in the right-hand side into the left, i.e
    \begin{equation*}
    \lV \psi_k\rV_{l^2_kX_k}^2\lesssim \lV\psi_k(0)\rV_{L^2}^2+\lV f_k\rV_{l^2_kN_k}^2.
    \end{equation*} 
	On the other hand, for the remaining bounded range of $k$, we have
	\begin{equation*}
	\lV \psi\rV_{X_k}\lesssim \lV \psi\rV_{L^{\infty}L^2},
	\end{equation*}
	and then (\ref{energy-l^2L^infL^2}) and (\ref{small-ass}) gives 
	\begin{align*}
	\lV\psi_k\rV_{l^2_kX_k}^2\lesssim &\lV \psi_k(0)\rV_{L^2}^2
	+\lV A\rV_{Z^{1,s+1}}\lV \psi_k\rV_{l_k^2X_k}^2
	+\lV f_k\rV_{l^2_kN_k}\lV \psi_k\rV_{l^2_kX_k}\\
	\lesssim & \lV \psi_k(0)\rV_{L^2}^2
	+\lV f_k\rV_{l^2_kN_k}^2,
	\end{align*}
	which finishes the proof of (\ref{energy-decay}). 
\end{proof}

\subsection{ The full linear problem}
Here we use the bounds for the paradifferential equation in the previous subsection 
in order to prove similar bounds for the full equation \eqref{Lin-eq0}:

\begin{prop}[Well-posedness]   \label{p:well-posed}
	Let $s>\frac{d}{2}$, $d\geq 3$ and $h=g-I\in \bY^{s+2}$, assume that the metric $g$, and the magnetic potential $A$  satisfy
	\begin{equation*}
	\lV h\rV_{\bY^{s+2}},\ \lV  A\rV_{Z^{1,s+1}},\ \lV  B\rV_{Z^{1,s}}\ll 1.
	\end{equation*}
Then the equation (\ref{Lin-eq0}) is well-posed for initial data $\psi_0\in H^{\si}$ with $-s\leq \si\leq s$, and we have the estimate 
	\begin{equation}\label{well-posed-s0}
	\lV \psi\rV_{l^2X^{\si}}\lesssim \lV \psi_0\rV_{H^{\si}}+\lV F\rV_{l^2N^{\si}}.
	\end{equation}
	Moreover, for $0\leq \si \leq s$ we have the estimate
	\begin{equation}\label{well-posed-s00}
	\lV \psi\rV_{l^2\bX^{\si}}\lesssim \lV \psi_0\rV_{H^{\si}}+\lV F\rV_{l^2N^{\si}\cap L^2H^{\si-2}}.
	\end{equation}
\end{prop}
\begin{proof}
    The well-posedness follows in a standard fashion from a similar energy estimate for the adjoint equation. Since the adjoint equation has a similar form, with similar bounds on the coefficients, such an estimate follows directly from \eqref{well-posed-s0}. Thus, we now focus on the proof of the bound \eqref{well-posed-s0}. For $\Psi$ solving (\ref{Lin-eq0}), we see that $\Psi_k$ solves
	\begin{equation*}
	\left\{
	\begin{aligned}      
	&i\d_t \Psi_k+\d_{\al}g^{\al\be}_{<k-4}\d_{\be}\Psi_k+2iA^{\al}_{<k-4}\d_{\al}\Psi_k=F_k+H_k,\\
	&\Psi_k(0)=\Psi_{0k},
	\end{aligned}\right.
	\end{equation*}
	where
	\begin{align*}
	H_k=&-S_k\d_{\al}g^{\al\be}_{\geq k-4}\d_{\be}\Psi-[S_k,\d_{\al}g^{\al\be}_{< k-4}\d_{\be}]\Psi-2i[S_k,A^{\al}_{<k-4}]\d_{\al}\Psi\\
	&-2iS_k[ A^{\al}_{\geq k-4}\d_{\al}\Psi_k]-S_k(B\Psi).
	\end{align*}
	If we apply Proposition \ref{Local-Energy-Decay} to each of these equations, we see that
	\begin{equation*}
	\lV\Psi_k\rV_{l^2X^{\si}}^2\lesssim \lV\Psi_{0k}\rV_{H^{\si}}^2+\lV F_k\rV_{l^2N^{\si}}^2+\lV H_k\rV_{l^2N^{\si}}^2.
	\end{equation*}
	
	We claim that 
	\begin{gather*}
	\sum_{k}\lV H_k\rV_{l^2N^{\si}}^2\lesssim (\lV  h\rV_{\bY^{s+2}}+\lV  A\rV_{Z^{1,s+1}}+\lV  B\rV_{Z^{1,s}})^2\lV \Psi\rV_{l^2X^{\si}}^2,\ \text{for }-s\leq \si\leq s.
	\end{gather*}
	Indeed, the bound for the terms in $H_k$ follows from (\ref{qdt-gdpsi_hl&hh}), (\ref{Comm-bd}), (\ref{Com_Ag-dpsi}), (\ref{qdt-Ag-dpsi}) and \eqref{cb-B}, respectively. Then by the above two bounds, we obtain the estimate (\ref{well-posed-s0}).
	
	Finally, by the $\psi$-equation (\ref{Lin-eq0}), for time derivative bound it suffices to consider the form
	\begin{equation*}
	\d_t\psi=\De\psi+\nab(h\nab\psi)+A\nab\psi+B\psi+F.
	\end{equation*}
	Then by the standard Littlewood-Paley dichotomy and Bernstein's inequality, for $0\leq \si\leq s$ we have the following estimates
	\begin{align*}\label{dtpsi-L2}
	\lV \d_t \psi\rV_{L^2H^{\si-2}}\lesssim \lV\psi\rV_{L^{\infty}H^{\si}}+\lV F\rV_{L^2H^{\si-2}},
	\end{align*}
	This, combined with \eqref{well-posed-s0}, yields the bound \eqref{well-posed-s00}, and then completes the proof of the Lemma.
\end{proof}

\subsection{The linearized problem.} Here we consider the linearized equation:
\begin{equation}\label{Lin-eq}
\left\{
\begin{aligned}      
&i\d_t \Psi+\d_{\al}g^{\al\be}\d_{\be}\Psi+2iA^{\al}\d_{\al}\Psi=F+G,\\
&\Psi(0)=\Psi_0,
\end{aligned}\right.
\end{equation}
where
\begin{equation*}
G=-\mathcal{G}\nab^2\psi-2i\mathcal{A}^{\al}\d_{\al} \psi,
\end{equation*}
and we prove the following.
\begin{prop}    \label{well-posedness-lemma}
	Let $s>\frac{d}{2}$, $0\leq \si\leq s-1$, $d\geq 3$ and $h=g-I\in \bY^{s+2}$, assume that $\Psi$ is a solution of \eqref{Lin-eq}, the metric $g$, $A$ and $V$ satisfy
	\begin{equation*}
	\lV h\rV_{\bY^{s+2}},\ \lV  A\rV_{Z^{1,s+1}}\ll 1.
	\end{equation*}
	Then we have the estimate
	\begin{equation}         \label{well-posed-s-1}
	\lV \Psi\rV_{l^2\bX^{\si}}\lesssim \lV \Psi_0\rV_{H^{\si}}+\lV F\rV_{l^2N^{\si}\cap L^2H^{\si-2}}+(\lV\mathcal{G}\rV_{\bY^{\si}}+\lV\mathcal{A}\rV_{Z^{1,\si+1}})\lV\psi\rV_{l^2X^{s}}.
	\end{equation}
\end{prop}
\begin{proof}
	For $\Psi$ solving (\ref{Lin-eq}), we see that $\Psi_k$ solves
	\begin{equation*}
	\left\{
	\begin{aligned}      
	&i\d_t \Psi_k+\d_{\al}g^{\al\be}_{<k-4}\d_{\be}\Psi_k+2iA^{\al}_{<k-4}\d_{\al}\Psi_k=F_k+G_k+H_k,\\
	&\Psi_k(0)=\Psi_{0k},
	\end{aligned}\right.
	\end{equation*}
	where
	\begin{gather*}
	G_k=-S_k(\mathcal{G}\nab^2\psi-2i\mathcal{A}^{\al}\d_{\al} \psi),
	\end{gather*}
	\begin{align*}
	H_k=&-S_k\d_{\al}g^{\al\be}_{\geq k-4}\d_{\be}\Psi-[S_k,\d_{\al}g^{\al\be}_{< k-4}\d_{\be}]\Psi-2i[S_k,A^{\al}_{<k-4}]\d_{\al}\Psi\\
	&-2iS_k[ A^{\al}_{\geq k-4}\d_{\al}\Psi_k].
	\end{align*}
	The proof of \eqref{well-posed-s-1} is similar to that of \eqref{well-posed-s00}. Here it suffices to prove
	\begin{gather*}
	\sum_{k}\lV G_k\rV_{l^2N^{\si}}^2\lesssim \lV \mathcal{G}\rV_{\bY^{\si+2}}^2\lV \psi\rV_{l^2X^{s}}^2+\lV  \mathcal{A}\rV_{Z^{1,\si+1}}^2\lV \psi\rV_{l^2X^{s}}^2,\\
	\lV G\rV_{L^2H^{\si-2}}\lesssim (\lV \GG\rV_{\bY^{\si+2}}+\lV  \AA\rV_{Z^{1,\si+1}})\lV \psi\rV_{l^2X^{s}}.
	\end{gather*}
	Indeed, the bound for the terms in $G_k$ follows from (\ref{qdt-gdpsi_hl&hh}), (\ref{qdt_h-d2psi}), (\ref{qdt-Ag-dpsi}) and (\ref{qdt-Ag-dpsi_lh}). The second bound follows from a standard Littlewood-Paley decomposition and Bernstein's inequality. This completes the proof of the Lemma.
\end{proof}

\section{Well-posedness in the good gauge}
\label{Sec-LWP}

In this section we use the elliptic results in Section~\ref{Sec-Ell}, the multilinear estimates in Section~\ref{Sec-mutilinear} and the linear local energy decay bounds in Section~\ref{Sec-LED}
in order to prove the good gauge formulation of our main result,  namely Theorem \ref{LWP-MSS-thm}.

\subsection{The iteration scheme: uniform bounds}
Here we seek to construct solutions to (\ref{mdf-Shr-sys-2}) iteratively, based on the scheme
\begin{equation}         \label{system-iteration}
	\left\{\begin{aligned}	
		&(i\d_t+\d_{\al}g^{(n)\al\be}\d_{\be})\psi^{(n+1)}+2i(A^{(n)\al}-\frac{1}{2}V^{(n)\al})\d_{\al}\psi^{(n+1)}=F^{(n)},\\
		&\psi(0)=\psi_0,
	\end{aligned}\right.	
\end{equation}
with the trivial initialization
\begin{equation*}
	\psi^{(0)}=0,
\end{equation*}
where the nonlinearities are
\begin{equation}\label{Nonlinearities-iteration}
\begin{aligned}
F^{(n)}=&\d_{\al}g^{(n)\al\be}\cdot\d_{\be}\psi^{(n)}+(B^{(n)}+A_{\al}^{(n)}A^{(n)\al}-V^{(n)\al}A^{(n)}_{\al})\psi^{(n)}-i\la_{\si}^{(n)\ga}\Im(\psi^{(n)}\bar{\la}^{(n)\si}_{\ga}),\\
\end{aligned}
\end{equation}
and $\SS^{(n)}=(\la^{(n)},h^{(n)},V^{(n)},A^{(n)},B^{(n)})$ are the solutions of elliptic equations \eqref{ell-syst} with $\psi=\psi^{(n)}$.

We assume that $\psi_0$ is small in $H^s$. Due to the above trivial initialization, we also 
inductively assume that
\begin{equation*}  
\lV \psi^{(n)}\rV_{l^2\bX^s}\leq C\lV\psi_0\rV_{H^s},
\end{equation*}
where $C$ is a big constant.

Applying the elliptic estimate (\ref{t:ell-time-dep_bd1}) to \eqref{ell-syst} with $\psi=\psi^{(n)}$ at each step, we obtain
\begin{align*}
\lV \SS^{(n)}\rV_{\bEE^{s}}\lesssim  \lV \psi^{(n)}\rV_{H^s},
\end{align*}
Applying at each step the local energy bound (\ref{well-posed-s00}) with $\si=s$ we obtain the estimate
\begin{align*}
\lV\psi^{(n+1)}\rV_{l^2\bX^s}\lesssim & \lV\psi_0\rV_{H^s}+\lV F^{(n)}\rV_{l^2N^s\cap L^2H^{s-2}}\\
\lesssim &\lV\psi_0\rV_{H^s}+\lV \SS^{(n)}\rV_{\bEE^{s}}(1+\lV \SS^{(n)}\rV_{\bEE^{s}})\lV\psi^{(n)}\rV_{l^2X^s}\\
\lesssim & \lV\psi_0\rV_{H^s}+(C\lV\psi_0\rV_{H^s})^2(1+C\lV\psi_0\rV_{H^s})\\
\leq &C\lV\psi_0\rV_{H^s}.
\end{align*}
Here the nonlinear terms in $F^{(n)}$ are estimated using (\ref{qdt_d h<k dpsi}), (\ref{cb-B}), (\ref{cb-AA}) and (\ref{cb_lam2-psi}) with $\si=s$.
Since $\psi_0$ is small in $H^s$, the above two bounds give
\begin{equation}  \label{Unif-bound}
\begin{aligned}
\lV \psi^{(n+1)}\rV_{l^2\bX^s}\leq C\lV\psi_0\rV_{H^s},
\end{aligned}
\end{equation}
which closes our induction.
\subsection{The iteration scheme: weak convergence.}
Here we prove that our iteration scheme converges in the weaker $H^{s-1}$ topology. We denote the differences by
\begin{gather*}
\Psi^{(n+1)}=\psi^{(n+1)}-\psi^{(n)},\\ \dSS^{(n+1)}=(\La^{(n+1)},\GG^{(n+1)},\VV^{(n+1)},\AA^{(n+1)},\BB^{(n+1)})=\SS^{(n+1)}-\SS^{(n)}
\end{gather*}
Then from (\ref{system-iteration}) we obtain the system
\begin{equation*}              
\left\{\begin{aligned}
&i\d_t \Psi^{(n+1)}+g^{(n)\al\be}\d_{\al\be}^2\Psi^{(n+1)}+2i(A^{(n)\al}-\frac{1}{2}V^{(n)\al})\d_{\al}\Psi^{(n+1)}=F^{(n)}-F^{(n-1)}+G^{(n)},\\
&\Psi^{(n+1)}(0,x)=0,
\end{aligned}\right.
\end{equation*}
where the nonlinearities $G^{(n)}$ have the form
\begin{align*}
G^{(n)}=&-\GG^{(n)}\d^2_{\al\be}\psi^{(n)}-2i(\AA^{(n)\al}-\frac{1}{2}\VV^{(n)\al})\d_{\al}\psi^{(n)},
\end{align*}

By \eqref{t:ell-time-dep_delta} we obtain 
\begin{equation*}
    \lV \dSS^{(n)}\rV_{\bEE^{s-1}}\lesssim \lV \Psi^{(n)}\rV_{l^2\bX^{s-1}}.
\end{equation*}
Applying \eqref{well-posed-s-1} with $\si=s-1$ for the $\Psi^{(n+1)}$ equation we have
\begin{align*}
\lV \Psi^{(n+1)}\rV_{l^2\bX^{s-1}}\lesssim & \lV F^{(n)}-F^{(n-1)}\rV_{l^2N^{s-1}\cap L^2H^{s-3}}+[\lV\GG^{(n)}\rV_{\bY^{s+1}}+\lV(\VV^{(n)},\AA^{(n)})\rV_{\bZ^{1,s}}]\lV\psi^{(n)}\rV_{l^2X^{s}}.
\end{align*}
Then by (\ref{qdt_d h<k dpsi}), (\ref{cb-B}), (\ref{cb-AA}) and (\ref{cb_lam2-psi}) with $\si=s-1$ we bound the right hand side above by
\begin{align*}
\lV \Psi^{(n+1)}\rV_{l^2\bX^{s-1}}\lesssim & C\lV\psi_0\rV_{H^s}\lV (\Psi^{(n)},\dSS^{(n)})\rV_{l^2\bX^{s-1}\times \bEE^{s-1}}\ll \lV \Psi^{(n)}\rV_{l^2\bX^{s-1}}.
\end{align*}
This implies that our iterations $\psi^{(n)}$ converge in $l^2\bX^{s-1}$ to some function $\psi$.
Furthermore, by the uniform bound (\ref{Unif-bound}) it follows that
\begin{equation}\label{Energy-bound}
\lV\psi\rV_{l^2\bX^s}\lesssim \lV\psi_0\rV_{H^s}.
\end{equation}
Interpolating, it follows that $\psi^{(n)}$ converges to $\psi$ in $l^2\bX^{s-\epsilon}$ 
for all $\epsilon > 0$. This allows us to conclude that the auxiliary functions 
$\SS^{(n)}$ associated to $\psi^{(n)}$ converge to the functions $\SS$ associated to $\psi$,
and also to pass to the limit and conclude that $\psi$ solves the (SMCF) equation \eqref{mdf-Shr-sys-2}. Thus we have established the existence part of our main theorem.

\subsection{Uniqueness via weak Lipschitz dependence.}
Consider the difference of two solutions 
\[
(\Psi,\dSS)=(\psi^{(1)}-\psi^{(2)},\SS^{(1)}-\SS^{(2)}).
\]
The $\Psi$ solves an equation of this form
\begin{equation*}              
\left\{\begin{aligned}
&i\d_t \Psi+\d_{\al}g^{(1)\al\be}\d_{\be}\Psi+2i(A^{(1)\al}-\frac{1}{2}V^{(1)\al})\d_{\al}\Psi=F^{(1)}-F^{(2)}+G,\\
&\Psi(0,x)=\psi^{(1)}_0(x)-\psi^{(2)}_0(x),
\end{aligned}\right.
\end{equation*}
where the nonlinearity $G$ is 
\begin{align*}
G=&-\d_{\al}(\GG\d_{\be}\psi^{(2)})-2i(\AA^{\al}-\frac{1}{2}\VV^{\al})\d_{\al}\psi^{(2)}.
\end{align*}

By \eqref{t:ell-time-dep_delta}, we have
\begin{equation*}
    \lV \dSS\rV_{\bEE^{s-1}}\lesssim \lV \Psi\rV_{l^2\bX^{s-1}}.
\end{equation*}
Applying (\ref{well-posed-s-1}) with $\si=s-1$ to the $\Psi$ equation, we obtain the estimate
\begin{align*}
\lV \Psi\rV_{l^2\bX^{s-1}}\lesssim &\ \lV\Psi_0\rV_{H^{s-1}}+\lV F^{(1)}-F^{(2)}\rV_{l^2N^{s-1}\cap L^2H^{s-3}}+(\lV\GG\rV_{\bY^{s+1}}+\lV(\VV,\AA)\rV_{\bZ^{1,s}})\lV\psi^{(2)}\rV_{l^2X^{s}}\\
\lesssim &\  \lV\Psi_0\rV_{H^{s-1}}+C\lV(\psi^{(1)}_0,\psi^{(2)}_0)\rV_{H^{s}}\lV (\Psi,\dSS)\rV_{l^2\bX^{s-1}\times\bEE^{s-1}}.
\end{align*}
Then, by \eqref{t:ell-time-dep_delta}, we further have
\begin{align*}
\lV \Psi\rV_{l^2\bX^{s-1}}\lesssim \lV\Psi_0\rV_{H^{s-1}}+C\lV(\psi^{(1)}_0,\psi^{(2)}_0)\rV_{H^{s}}\lV\Psi\rV_{l^2\bX^{s-1}}
\end{align*}
Since the initial data $\psi^{(1)}_0$ and $\psi^{(2)}_0$ are sufficiently small, we obtain 
\begin{equation}      \label{Uniq-bound}
\lV \Psi\rV_{l^2\bX^{s-1}}\lesssim \lV\Psi_0\rV_{H^{s-1}}.
\end{equation}
This gives the weak Lipschitz dependence, as well as the uniqueness of solutions for \eqref{mdf-Shr-sys-2}.

\subsection{Frequency envelope bounds}\label{Freq-envelope-section}
Here we prove a stronger frequency envelope version of estimate (\ref{Energy-bound}).

\begin{prop}\label{Freq-envelope-bounds}
	Let $\psi\in l^2\bX^s$ be a small data solution to (\ref{mdf-Shr-sys-2}), which satisfies (\ref{Energy-bound}). Let $\{p_{0k}\}$ be an admissible frequency envelope for the initial data $\psi_0\in H^s$. Then $\{p_{0k}\}$ is also frequency envelope for $\psi$ in $l^2\bX^s$.
\end{prop}
\begin{proof}
	Let $p_k$ and $s_k$ be the admissible frequency envelopes for solution $(\psi,\SS)\in l^2\bX^s\times \bEE^s$. Applying $S_k$ to the modified Schr\"{o}dinger equation in (\ref{mdf-Shr-sys-2}), we obtain the paradifferential equation
	\begin{equation*}         
	\left\{\begin{aligned}
	&(i\d_t+\d_{\al}g_{<k-4}^{\al\be}\d_{\be})\psi_k+2i(A-\frac{1}{2}V)_{<k-4}^{\al}\d_{\al}\psi_k=F_k+J_k,\\
	&\psi(0,x)=\psi_0(x),
	\end{aligned}\right.
	\end{equation*}
	where 
	\begin{align*}
	J_k=&-S_k\d_{\al}g^{\al\be}_{\geq k-4}\d_{\be}\psi-[S_k,\d_{\al}g^{\al\be}_{< k-4}\d_{\be}]\psi\\
	&-2i[S_k,(A-\frac{1}{2}V)^{\al}_{<k-4}]\d_{\al}\psi
	-2iS_k[ (A-\frac{1}{2}V)^{\al}_{\geq k-4}\d_{\al}\psi_k],
	\end{align*}
	and $\SS=(\la,h,V,A,B)$ is the solution to the elliptic system \eqref{ell-syst}.
	We estimate $\psi_k=S_k\psi$ using Proposition \ref{well-posedness-lemma}. By Proposition \ref{Non-Est}, Lemma \ref{bilinear-est} and Lemma \ref{Comm-est-Lemma} we obtain
	\begin{equation*}    \label{freq-envelope-psi}
	    \begin{aligned}
	\lV \psi_k\rV_{l^2\bX^s}\lesssim& \ p_{0k}+(p_k+s_k)
	\lV\psi \rV_{l^2\bX^s}.
	\end{aligned}
	\end{equation*}
	Then by \eqref{t:ell-time-dep_freq-envelope}, the definition of frequency envelope \eqref{Freq-envelope}  and \eqref{Energy-bound}, this implies
	\begin{align*}
	    p_k\lesssim p_{0k}+p_k\lV\psi \rV_{l^2\bX^s}.
	\end{align*}
	By the smallness of $\psi\in l^2\bX^s$, this further gives $p_k\lesssim p_{0k}$, and concludes the proof.
\end{proof}

\subsection{Continuous dependence on the initial data}
Here we show that the map $\psi_0\rightarrow(\psi,\SS)$ is continuous from $H^s$ into $l^2\bX^s\times\bEE^s$. By \eqref{t:ell-time-dep_delta}, it suffices to prove $\psi_0\rightarrow \psi$ is continuous from $H^s$ to $l^2\bX^s$.

Suppose that $\psi_0^{(n)}\rightarrow \psi_0$ in $H^s$. Denote by $p_{0k}^{(n)}$, respectively $p_{0k}$ the frequency envelopes associated to $\psi_0^{(n)}$, respectively $\psi_0$, given by (\ref{Freq-envelope}). If $\psi_0^{(n)}\rightarrow\psi_0$ in $H^s$ then $p_{0k}^{(n)}\rightarrow p_{0k}$ in $l^2$. Then for each $\ep>0$ we can find some $N_{\ep}$ so that
\[
\lV p_{0,>N_{\ep}}^{(n)}\rV_{l^2}\leq \ep,\ \text{for all }n.
\]
By Proposition \ref{Freq-envelope-bounds} we obtain that
\begin{equation}  \label{high-freq-small}
\lV \psi_{>N_{\ep}}^{(n)}\rV_{l^2\bX^s}\leq \ep,\ \text{for all }n.
\end{equation}
To compare $\psi^{(n)}$ with $\psi$ we use (\ref{Uniq-bound}) for low frequencies and (\ref{high-freq-small}) for the high frequencies,
\begin{align*}
\lV \psi^{(n)}-\psi\rV_{l^2\bX^s}\lesssim& \lV S_{<N_{\ep}}(\psi^{(n)}-\psi)\rV_{l^2\bX^s}+\lV S_{>N_{\ep}}\psi^{(n)}\rV_{l^2X^s}+\lV S_{>N_{\ep}}\psi\rV_{l^2X^s}\\
\lesssim & 2^{N_{\ep}}\lV S_{<N_{\ep}}(\psi^{(n)}-\psi)\rV_{l^2\bX^{s-1}}+2\ep\\
\lesssim &2^{N_{\ep}}\lV S_{<N_{\ep}}(\psi^{(n)}_0-\psi_0)\rV_{H^{s-1}}+2\ep.
\end{align*}	
Letting $n\rightarrow\infty$ we obtain
\[
\limsup_{n\rightarrow\infty}\lV \psi^{(n)}-\psi\rV_{l^2\bX^s}\lesssim \ep.
\]	
Letting $\ep\rightarrow 0$ we obtain
\begin{equation*}      \label{Con-depen-psi}
\lim_{n\rightarrow 0}\lV \psi^{(n)}-\psi\rV_{l^2\bX^s}=0,
\end{equation*}
which completes the desired result.

\subsection{Higher regularity}
Here we prove that the solution $(\psi,\SS)$ satisfies the bound
\begin{equation*}
\lV (\psi,\SS)\rV_{l^2\bX^{\si}\times \bEE^{\si}}\lesssim \lV\psi_0\rV_{H^{\si}},\quad\si\geq s,
\end{equation*}
whenever the right hand side is finite.

Differentiating the original Schr\"odinger equation \eqref{mdf-Shr-sys-2} yields
\begin{equation*}           \label{Diff-Shr-sys}
(i\d_t+\d_{\al}g^{\al\be}\d_{\be})\nab\psi+2i(A-\frac{V}{2})^{\al}\d_{\al}\nab\psi=-\d_{\al}(\nab g^{\al\be}\d_{\be}\psi)-2i\nab A^{\al}\d_{\al}\psi+\nab F,
\end{equation*}
where $F$ is defined as in (\ref{Nonlinearities-iteration}) without superscript $(n)$. Using Proposition \ref{well-posedness-lemma}(b) we obtain
\begin{align*}
\lV \nab\psi\rV_{l^2\bX^s}\lesssim \lV \nab\psi_0\rV_{H^s}+\lV (\nab\psi,\nab\SS)\rV_{l^2\bX^s\times\bEE^s}\lV (\psi,\SS)\rV_{l^2\bX^s\times\bEE^s}(1+\lV (\psi,\SS)\rV_{l^2\bX^s\times\bEE^s})^N.
\end{align*}
For elliptic equations, by \eqref{t:ell-time-dep_delta} we obtain
\begin{align*}
\lV \nab \SS\rV_{\bEE^s}\lesssim  \lV \nab\psi\rV_{l^2\bX^s}.
\end{align*}
Hence, by \eqref{Energy-bound}, these imply
\begin{align*}
\lV (\nab\psi,\nab\SS)\rV_{l^2\bX^{s}\times \bEE^s}\lesssim  \lV \nab\psi_0\rV_{H^{s}}.
\end{align*}
Inductively, we can obtain the system for $(\nab^n\psi,\nab^n\SS)$. This leads to 
\begin{align*}
\lV (\nab^n\psi,\nab^n\SS)\rV_{l^2\bX^s\times\bEE^s}\lesssim  \lV \psi_0\rV_{H^{s+n}}+\|\psi\|_{l^2\bX^{s+n-1}}^N,
\end{align*}
which shows that
\begin{align*}
\lV (\psi,\SS)\rV_{l^2\bX^{s+n}\times\bEE^{s+n}}\lesssim  \lV \psi_0\rV_{H^{s+n}}+\|\psi\|_{l^2\bX^{s+n-1}}^N.
\end{align*}

\section{The reconstruction of the flow}

In this last section we close the circle of ideas in this paper, and prove that 
one can start from the good gauge solution given by Theorem~\ref{LWP-MSS-thm}, and 
reconstruct the flow at the level of $d$-dimensional embedded submanifolds.
We do this in several steps:

\bigskip

\subsection{ The starting point} Our evolution begins at time $t=0$, where we need to represent
the initial submanifold as parametrized with global harmonic coordinates, represented via the map 
$F: \R^d \to \R^{d+2}$. This is the goal of this 
subsection, which is carried out in Proposition \ref{Global-harmonic}.

Once this is done, we have the frame $F_\alpha$ in the tangent space and the frame $m$ in the normal bundle. In turn, as described in Section~\ref{gauge}, these generate the metric $g$, the second fundamental form $\lambda$ with trace $\psi$ and the connection $A$, all at the initial time $t=0$.

Moving forward in time, Theorem~\ref{LWP-MSS-thm} provides us with the time evolution 
of $\psi$ via the Sch\"odinger flow \eqref{mdf-Shr-sys-2}, as well as the 
functions $(\lambda, g, V, A,B)$ satisfying the elliptic system \eqref{ell-syst}
together with the constraints  \eqref{csmc}, \eqref{R-la}, \eqref{lambda-sim}, \eqref{cpt-AiAj-2}, \eqref{Coulomb} and \eqref{hm-coord}.  The objective of the rest of this section is then to use these functions in order  to reconstruct the map $F$ which describes the manifold $F$ at later times.

We now return to the question of constructing the 
harmonic coordinates at the initial time.
In order to state the following proposition, we define some notations. Let $F:\R^d_x \rightarrow(\R^{d+2},g_0)$ be an immersion with induced metric $g(x)$. For any change of coordinate $y=x+\phi(x)$, we denote 
\begin{equation*}
    \tilde{F}(y)=F(x(y)),
\end{equation*}
and its induced metric $\tilde g_{\al\be}(y)=\<\d_{y_\al}\tF,\d_{y_\be}\tF\>$. We also denote its Christoffel symbol as $\tilde \Ga$ and $\tilde h(y)=\tilde g(y)-I_d$.

\begin{prop}[Global harmonic coordinates]\label{Global-harmonic}
Let $d\geq 3$, $s>\frac{d}{2}$, and 
\[
F:(\R^d_x,g)\rightarrow (\R^{d+2},g_0)
\]
be an immersion with induced metric $g=I_d+h$. Assume that $\nab h(x)$ is small in $H^{s}(dx)$. Then there exists a unique change of coordinates 
$y=x+\phi(x)$ with $\lim_{x\rightarrow\infty}\phi(x)=0$ and $\nab\phi$ uniformly small,
such that the new coordinates $\{y_1,\cdots,y_d\}$ are global harmonic coordinates, namely,
\begin{equation*}
    \tilde g^{\al\be}(y)\tilde{\Ga}_{\al\be}^\ga(y)=0,\quad \text{for any }y\in\R^d.
\end{equation*}
Moreover,
\begin{equation}   \label{bound-phi}
   \|\nab^{2}\phi(x)\|_{H^{s}(dx)}\lesssim \|\nab h(x)\|_{H^{s}(dx)},
\end{equation}
and, in the new coordinates $\{y_1,\cdots,y_d\}$,
\begin{equation}   \label{bound-tildeg}
    \| \d_y \tilde h\|_{H^{s}(dy)}\lesssim \|\d_x h\|_{H^s(dx)}.
\end{equation}
\end{prop}
\begin{proof}
    \emph{Step 1: Derivation of the $\phi$-equations.} 
    
    We make the following change of coordinates such that the $\{y_1,\cdots,y_d\}$ is a global harmonic coordinate
    \begin{equation*}
    \begin{matrix}
        \R^d &\longrightarrow & \R^d &\longrightarrow &\R^{d+2}\\
        y&\longmapsto & x&\longmapsto & F(x(y))=\tilde{F}(y)
    \end{matrix}
    \end{equation*}
    where $x+\phi(x)=y$ with $\lim_{x\rightarrow\infty}\phi(x)=0$ and $\nab\phi$ small.
    
    To determine the function $\phi$, we perfor a few computations.    For any vector $f=(f_1,\cdots,f_d)$, we denote 
    \begin{equation*}
        \frac{\d f}{\d x}=
        \left( \begin{matrix}
            \frac{\d f_1}{\d x_1}& \cdots &\frac{\d f_1}{\d x_d}\\
            \vdots & \ddots &\vdots\\
            \frac{\d f_d}{\d x_1}& \cdots &\frac{\d f_d}{\d x_d}
        \end{matrix} \right).
    \end{equation*}
    Then we have
    \begin{equation*}
        \frac{\d x}{\d y}+\frac{\d \phi}{\d x}\frac{\d x}{\d y}=I_d.
    \end{equation*}
    This implies that
    \begin{equation*}
        \frac{\d x}{\d y}=I_d-\frac{\d \phi}{\d x}+\CC(x),
    \end{equation*}
    where the matrix $\mathcal C(x)$ is a higher order term which satisfies
    \begin{equation*}
        \mathcal C(x)=(\frac{\d \phi}{\d x})^2-\mathcal C(x) \frac{\d \phi}{\d x},
    \end{equation*}
    or, equivalently,  it is given by 
    \[
    \mathcal C(x)=(\frac{\d \phi}{\d x})^2(I-\frac{\d \phi}{\d x})^{-1}.
    \]
    We denote
    \begin{equation}   \label{exp-P}
        \PP^\mu_\al:= -\frac{\d \phi_\mu}{\d x_\al}+\mathcal C_{\mu\al}(x).
    \end{equation}
    
    Since $\tilde{F}(y)=F(x(y))$, then we have
    \begin{equation}         \label{exp-tilde g_albe}
        \begin{aligned}
        \tilde{g}_{\al\be}(y)=&\<\frac{\d \tilde F}{\d y_\al},\frac{\d \tilde F}{\d y_\be}\>
        =\<\frac{\d  F}{\d x_\mu}\frac{\d  x_\mu}{\d y_\al},\frac{\d  F}{\d x_\nu}\frac{\d x_\nu}{\d y_\be}\>\\
        =& g_{\mu\nu}(x)(\de^\mu_\al-\d_\al \phi_\mu+\mathcal{C}_{\mu\al})(\de^\nu_\be-\d_\be \phi_\nu+\mathcal{C}_{\nu\be})
    \end{aligned}
    \end{equation}
    and 
    \begin{align}  \label{exp-tilde g}
        \tilde{g}^{\al\be}(y)=&g^{\mu\nu}\frac{\d y_\al}{\d x_\mu}\frac{\d y_\be}{\d x_\nu}=g^{\mu\nu}(\de^\al_\mu+\d_\mu \phi_\al)(\de^\be_\nu+\d_\nu \phi_\be).
    \end{align}
    We also have 
    \begin{align}  \label{exp-til dg}
        \frac{\d\tilde{g}_{\al\be}(y)}{\d y_\ga}=\frac{\d\tilde{g}_{\al\be}(y)}{\d x_m}\frac{\d x_m}{\d y_\ga}
        = -g_{\mu\be}\d^2_{\al\ga}\phi_\mu-g_{\al\nu}\d^2_{\be\ga}\phi_\nu+\d_\ga g_{\al\be}+\mathcal K_{\al\be,\ga},
    \end{align}
    where the higher order terms $\mathcal K_{\al\be,\ga}$ are defined as
    \begin{align*}
        \mathcal K_{\al\be,\ga}:=&-g_{\mu\nu}\d^2_{\al\ga}\phi_\mu \mathcal{P}^\nu_\be+g_{\mu\nu}\d_\ga \mathcal{C}_{\mu\al}\frac{\d x_\nu}{\d y_\be}
        -g_{\mu\nu}\PP^\mu_\al\d^2_{\be\ga}\phi_\nu +g_{\mu\nu}\frac{\d x_\mu}{\d y_\al} \d_\ga \mathcal{C}_{\nu\be}\\
        &+\d_\ga g_{\al\nu}\PP^\nu_\be+\d_\ga g_{\mu\nu}\PP^\mu_\al \frac{\d x_\nu}{\d y_\be} 
        +\d_{x_m}[g_{\mu\nu}(x) \frac{\d x_\mu}{\d y_\al}\frac{\d x_\nu}{\d y_\be}]\PP^m_\ga.
    \end{align*}
    The relation $\tilde{g}^{\al\be}\tilde{\Ga}_{\al\be,\ga}=0$ combined with \eqref{exp-tilde g} and \eqref{exp-til dg} implies that
    \begin{align*}
        0=&g^{mn}(\de^\al_m+\d_m \phi_\al)(\de^\be_n+\d_n \phi_\be)\big[-g_{\mu\be}\d^2_{\al\ga}\phi_\mu-g_{\ga\nu}\d^2_{\be\al}\phi_\nu+\d_\al g_{\ga\be}+\mathcal K_{\ga\be,\al}\\
        &+\frac{1}{2} g_{\mu\be}\d^2_{\al\ga}\phi_\mu+\frac{1}{2}g_{\al\nu}\d^2_{\be\ga}\phi_\nu-\frac{1}{2}\d_\ga g_{\al\be}-\frac{1}{2}\mathcal K_{\al\be,\ga}\big].
    \end{align*}
    This gives the elliptic equations of $\phi$,
    \begin{equation}\label{varphi-eq}
        \begin{aligned}
        \De \phi_\ga=&\mathbf{Non}_\ga(g,\phi),
    \end{aligned}
    \end{equation}
    with the boundary condition $\lim_{x\rightarrow\infty}\phi(x)=0$,
    where the nonlinearities $\mathbf{Non}_\ga(g,\phi)$ are given by
    \begin{align*}
        {\bf Non}_\ga(g,\phi):=&-h_{\ga\nu}\De \phi_\nu-h^{\al\be}g_{\ga\nu}\d^2_{\al\be}\phi_\nu+ g^{\al\be}(\Ga_{\al\be,\ga}+\mathcal K_{\ga\be,\al}-\frac{1}{2}\mathcal K_{\al\be,\ga})\\
        &+g^{mn}(\de^\al_m \d_n \phi_\be+\d_m \phi_\al \de^\be_n+\d_m \phi_\al\d_n \phi_\be)\big[-g_{\mu\be}\d^2_{\al\ga}\phi_\mu-g_{\ga\nu}\d^2_{\be\al}\phi_\nu\\
        &+\frac{1}{2} g_{\mu\be}\d^2_{\al\ga}\phi_\mu+\frac{1}{2}g_{\al\nu}\d^2_{\be\ga}\phi_\nu+\Ga_{\al\be,\ga}+\mathcal K_{\ga\be,\al}-\frac{1}{2}\mathcal K_{\al\be,\ga}\big].
    \end{align*}
    
    \medskip
    
    \emph{Step 2: Solve the $\phi$-equations \eqref{varphi-eq}.}
    By the contraction principle, the existence and uniqueness of solution of \eqref{varphi-eq} and the bound \eqref{bound-phi} are obtained by the following Lemma.
    \begin{lemma}    \label{phi-Lipschitz-Lem}
    Let $g$ be as in Proposition~\ref{Global-harmonic}. Then the  map $\phi \to {\bf Non}_\ga(g,\phi)$
    is Lipschitz from 
    \[
    H^{s+2} + \dot H^2 \to H^s
    \]
    with Lipschitz constant $\epsilon$
    for $\| \nabla^2 \phi \|_{H^s} \lesssim \epsilon$.
    \end{lemma}
    \begin{proof}[Proof of Lemma \ref{phi-Lipschitz-Lem}] In order to prove Lemma \ref{phi-Lipschitz-Lem}, 
    we consider the following simplified linearization for 
    ${\bf Non}_\ga(g,\phi)$ as a function of $\phi$:
     \begin{equation}\label{Linear-varphi-eq}
        \begin{aligned}
        \mathcal T(g,\phi,\Phi)=&h(1+h)\nab^2\Phi+g(\nab h+\dK)\\
        &+g(\nab \Phi+\nab \phi\nab \Phi)\big[g\nab^2\phi+\nab h+\mathcal K\big]\\
        &+g(\nab\phi+\nab\phi\nab\phi)\big[g\nab^2\Phi+\dK]
    \end{aligned}
    \end{equation}
    where $\Phi$ is the linearized variable associated to $\phi$, $\mathcal K$ has the form
    \begin{align*}
        \mathcal K:=&g\nab^2\phi \PP+g\nab \CC(1+\PP)+\nab h \PP(1+\PP)
        +\nab[g(1+\PP)^2]\PP,
    \end{align*}
    and $\dK$ is
    \begin{align*}
        \dK:=&g\nab^2\Phi \PP+g\nab^2\phi \delta \PP+g\nab\dC(1+\PP)+g\nab \CC\dP+\nab h \dP(1+\PP)\\
        &+\nab[g\dP(1+\PP)]\PP+\nab[g(1+\PP)^2]\dP.
    \end{align*}
    Here $\CC$ and $\delta \CC$ satisfy
    \begin{gather*}
        \CC=\nab\phi\nab\phi+\CC\nab\phi,\quad 
        \delta \CC=\nab\phi\nab\Phi+\delta \CC \nab\phi+\CC\nab\Phi,
    \end{gather*}
    and $\PP$ and $\delta \PP$ are 
    \begin{equation*}
        \PP=\nab \phi+\CC,\quad \delta \PP=\nab \Phi+\delta \CC.
    \end{equation*}
    
    Then for the equation \eqref{Linear-varphi-eq} we have estimates as follows:
    \begin{lemma}[Elliptic estimates for \eqref{Linear-varphi-eq}]\label{Ellp_estimate-varphi_Lemma}
    Let $d\geq 3$ and $s>d/2$. Assume that $\|\nab h\|_{H^s}\lesssim \ep$ and $\| \nab^2\phi\|_{H^s}\lesssim \ep$, then for the linearized expression \eqref{Linear-varphi-eq} we have the following estimate
    \begin{equation}   \label{estimate-varphi-eq}
        \|\mathcal T(g,\phi,\Phi)\|_{H^s}\lesssim \|\nab h\|_{H^s}+\ep \|\nab^2 \Phi\|_{H^s}.
    \end{equation}
    \end{lemma}
    \begin{proof}[Proof of Lemma \ref{Ellp_estimate-varphi_Lemma}]
    
    First, we bound $\CC$, $\dC$, $\PP$ and $\dP$. By Sobolev embeddings, using also the smallness
    condition $\|\nab^2\phi\|_{H^s}\lesssim \ep$, we have
    \begin{align*}
        \|\nab \CC\|_{H^s}\lesssim \|\nab^2\phi\|_{H^s}^2+\|\nab \CC\|_{H^s}\|\nab^2\phi\|_{H^s}\lesssim \ep^2+\ep \|\nab \CC\|_{H^s},
    \end{align*}
    and
    \begin{align*}
        \|\nab\dC\|_{H^s}\lesssim & \|\nab^2\phi\|_{H^s}\|\nab^2\Phi\|_{H^s}+\|\nab\dC\|_{H^s}\|\nab^2\phi\|_{H^s}+\|\nab \CC\|_{H^s}\|\nab^2\Phi\|_{H^s}\\
        \lesssim & \ep \|\nab^2\Phi\|_{H^s}+\ep\|\nab\dC\|_{H^s}+\|\nab \CC\|_{H^s}\|\nab^2\Phi\|_{H^s}.
    \end{align*}
    These imply
    \begin{equation}   \label{estimate-C}
        \|\nab \CC\|_{H^s}\lesssim  \ep^2,\quad \|\nab\dC\|_{H^s}\lesssim  \ep \|\nab^2\Phi\|_{H^s}.
    \end{equation}
    Similarly we have
    \begin{equation}   \label{estimate-P}
         \|\nab \PP\|_{H^s}\lesssim  \ep,\quad \|\nab\dP\|_{H^s}\lesssim   \|\nab^2\Phi\|_{H^s}.
    \end{equation}
    
    By Sobolev embedding we bound $\dK$ by
    \begin{align*}
        \|\dK\|_{H^s}\lesssim &(1+\|\nab h\|_{H^s})[\|\nab^2 \Phi\|_{H^s}\|\nab \PP\|_{H^s}+\|\nab^2 \phi\|_{H^s}\|\nab \dP\|_{H^s}\\
        &+\|\nab \dC\|_{H^s}(1+\|\nab \PP\|_{H^s})+\|\nab \CC\|_{H^s}\|\nab \dP\|_{H^s}\\
        &+\|\nab h\|_{H^s}\|\nab \dP\|_{H^s}(1+\|\nab \PP\|_{H^s})^2\\
        &+(1+\|\nab h\|_{H^s})\|\nab\dP\|_{H^s}\|\nab \PP\|_{H^s}(1+\|\nab \PP\|_{H^s})].
    \end{align*}
    This combined with \eqref{estimate-C} and \eqref{estimate-P} implies
    \begin{align} \label{estimate-dK}
        \|\dK\|_{H^s}\lesssim \ep \|\nab^2\Phi\|_{H^s}.
    \end{align}
    Similarly, we also have
    \begin{align}   \label{estimate-K}
        \|\mathcal K\|_{H^s}\lesssim \ep^2.
    \end{align}
    
    Now by Sobolev embedding we bound $\mathcal T(g,\phi,\Phi)$ by
    \begin{align*}
        \|\mathcal T\|_{H^s}\lesssim & \|\nab h\|_{H^s}(1+\|\nab h\|_{H^s})(1+\|\nab^2 \Phi\|_{H^s})+(1+\|\nab h\|_{H^s})\|\dK\|_{H^s}\\
        &+(1+\|\nab h\|_{H^s})\|\nab^2 \Phi\|_{H^s}(1+\|\nab^2 \phi\|_{H^s})\\
        &\quad \cdot[(1+\|\nab h\|_{H^s})\|\nab^2 \phi\|_{H^s}+\|\nab h\|_{H^s}+\|\mathcal K\|_{H^s}]\\
        &+(1+\|\nab h\|_{H^s})\|\nab^2 \phi\|_{H^s}(1+\|\nab^2 \phi\|_{H^s})[(1+\|\nab h\|_{H^s})\|\nab^2 \Phi\|_{H^s}+\|\dK\|_{H^s}].
    \end{align*}
    By the assumptions, \eqref{estimate-K} and \eqref{estimate-dK}, this gives
    \begin{align*}
        \|\mathcal T(g,\phi,\Phi)\|_{H^s}\lesssim \|\nab h\|_{H^s}+\ep \|\nab^2\Phi\|_{H^s}.
    \end{align*}
    We conclude the proof of the lemma.
    \end{proof}
    
    We continue to prove Lemma \ref{phi-Lipschitz-Lem}. With small Lipschitz constant $\ep$ for $\|\nab^2\phi\|_{H^s}\lesssim \ep$, by \eqref{estimate-varphi-eq} we have
    \begin{equation*}
        \|\mathbf{Non}_\ga (g,\phi)\|_{H^s}\lesssim \|\nab h\|_{H^s}+\ep^2,
    \end{equation*}
    and 
    \begin{equation*}
        \|\mathbf{Non}_\ga (g,\phi)-\mathbf{Non}_\ga (g,\tilde\phi)\|_{H^s}\lesssim \ep \|\nab^2(\phi-\tilde\phi)\|_{H^s}.
    \end{equation*}
    These give the Lipschitz continuity, completing the proof of Lemma \ref{phi-Lipschitz-Lem}.
    \end{proof}

    \medskip 
    
    \emph{Step 3: Prove the bound \eqref{bound-tildeg}.}
    First we prove the following bound
    \begin{equation}    \label{bound-tildeg-key1}
    \|(\d_y \tih)(y(x))\|_{H^s(dx)}\lesssim \|\d_x h\|_{H^s(dx)}.
    \end{equation}
    By \eqref{exp-tilde g_albe}, it suffices to bound
    \begin{align*}
    \|(1+\PP)\d_x[g(1+\PP)^2]\|_{H^s}\lesssim & \|\d_x[g(1+\PP)^2]\|_{H^s}(1+\|\nab\PP\|_{H^s})\\
    \lesssim & \|\d_x g\|_{H^s}(1+\|\nab\PP\|_{H^s})^3\\
    &+\|\d_x \PP\|_{H^s}(1+\|\d_x h\|_{H^s})(1+\|\nab\PP\|_{H^s})^2\\
    \lesssim & (\|\d_x g\|_{H^s}+\|\d_x \PP\|_{H^s})(1+\ep)^3\lesssim \|\d_x h\|_{H^s}.
    \end{align*}
    This gives the bound \eqref{bound-tildeg-key1}.
    
    In order to complete the proof, we also need the following lemma:
    \begin{lemma}\label{Equ-f(y)&F(x)-Lemma}
    	Let the change of coordinates $x+\phi(x)=y$ be as in Proposition \ref{Global-harmonic}. 
    Define the linear operator $T$ as $T(F)(y)=F(x(y))$	
for any function $F\in L^2(dx)$. Then we have
    	\begin{equation}\label{Equ-f(y)&F(x)}
    	\| T(F)(y)\|_{H^\sigma(dy)}\lesssim \|F(x)\|_{H^\sigma(dx)}, \qquad
    	\sigma \in [0,[s]+1].
    	\end{equation}
    \end{lemma} 
    Given this lemma, the bound \eqref{bound-tildeg} is obtained by \eqref{bound-tildeg-key1} and \eqref{Equ-f(y)&F(x)} with $\sigma = s$, and the proof of Proposition~\ref{Global-harmonic} is concluded. It remains to prove the Lemma.
    
    \begin{proof}[Proof of Lemma \ref{Equ-f(y)&F(x)-Lemma}]
    	Let $k$ be an integer $k\in[0,[s]+1]$, where $[s]$ is the integer part of $s$. By the change of coordinates $x+\phi(x)=y$, we have
    	\begin{align*}
    	\d_y^k T(F)(y)=[\frac{\d x}{\d y}\frac{\d}{\d x}]^k F(x)\approx [(1+\PP)\d_x]^k F(x).
    	\end{align*}
    	It suffices to consider the following forms
    	\begin{align*}
    	\sum_{\substack{
    			1\leq i\leq k-1,\ l+l_1+\cdots+l_{i}= k,\\
    			l\geq 1,\ l_1\geq \cdots \geq l_{i}\geq 1
    	}}\d_x^{l}F\d_x^{l_1}\PP\cdots \d_x^{l_{i}}\PP(1+\PP)^{k-i}.
    	\end{align*}
    	By Sobolev embedding, we bound each terms by
    	\begin{align*}
    	 \|\d_x^{l}F\d_x^{l_1}\PP\cdots \d_x^{l_{i}}\PP(1+\PP)^{k-i}\|_{L^2(dy)}
    	\lesssim &\|\d_x^{l}F\d_x^{l_1}\PP\cdots \d_x^{l_{i}}\PP(1+\PP)^{k-i} \sqrt{\det (I+\d_x \phi)}\|_{L^2(dx)}\\
    	\lesssim &  \|\d_x^{l}F\d_x^{l_1}\PP\cdots \d_x^{l_{i}}\PP\|_{L^2}(1+\|\nab\PP\|_{H^s})^{k-i} \|1+\nab\phi\|_{L^\infty}^d\\
    	\lesssim &\|F\|_{H^k}\|\nab\PP\|_{H^s}^i (1+\|\nab h\|_{H^s})^{k-i} (1+\|\nab h\|_{H^s})^d\\
    	\lesssim & \ep^i\|F\|_{H^k}.
    	\end{align*}
    	Then we have
    	\begin{equation*}
    	\|\d_y^k T(F)(y)\|_{L^2(dy)}\lesssim \sum_{i=0}^{k-1}\ep^i\|F(x)\|_{H^k(dx)}\lesssim \|F(x)\|_{H^k(dx)}.
    	\end{equation*}
    	This implies 
    	\begin{align}     \label{bound-T-k}
    	    \| T(F)(y)\|_{H^k(dy)}\lesssim \|F(x)\|_{H^k(dx)},\quad \text{for any }k\in[0,[s]+1].
    	\end{align}
    	Thus the bound \eqref{Equ-f(y)&F(x)} is obtained if $\sigma \in [0,[s]+1]$ is an integer. The similar bound for noninteger $\sigma$ follows by interpolation. 
    	\end{proof}
\end{proof}

\subsection{The time evolution of \texorpdfstring{$(\lambda,g,A)$}{}}
As part of our derivation of the (SMCF) equations \eqref{mdf-Shr-sys-2} for the mean curvature $\psi$ in  the good gauge, coupled with the elliptic system \eqref{ell-syst},
we have seen that the time evolution of $(\lambda,g,A)$ is described by the equations 
\eqref{main-eq-abst}, \eqref{g_dt} and \eqref{Cpt-A&B}.  However, our proof of the well-posedness
result for the Schr\"odinger evolution \eqref{mdf-Shr-sys-2} does not apriori guarantee
that \eqref{main-eq-abst}, \eqref{g_dt} and \eqref{Cpt-A&B}  hold. Here we rectify this 
omission:

\begin{lemma}
Assume that $\psi \in C[0,T;H^s]$ solves the SMCF equation \eqref{mdf-Shr-sys-2}
coupled with the elliptic system \eqref{ell-syst}. Then the relations \eqref{g_dt}, \eqref{main-eq-abst} and \eqref{Cpt-A&B} hold.
\end{lemma}

\begin{proof}
We recall that, by Theorem~\ref{t:ell-fixed-time},  the solution $\SS= (\la,h,V,A,B)$ in $\mathcal H^s$ for the system \eqref{ell-syst} satisfies the fixed time constraints \eqref{csmc}, \eqref{R-la},  \eqref{la-commu}, \eqref{cpt-AiAj-2}, \eqref{hm-coord} and \eqref{Coulomb}. On the other hand, in terms of the time evolution, at this point we only have the  equation \eqref{mdf-Shr-sys-2} for the mean curvature $\psi$.  We will show that this implies
\eqref{g_dt}, \eqref{main-eq-abst} and \eqref{Cpt-A&B}.

To shorten the notations, we define the tensors
\begin{align*}
    &T^1_{\al\be}=\d_t g_{\al\be}-2\Im(\psi\bar{\la}_{\al\be})-\nab_{\al}V_{\be}-\nab_{\be}V_{\al},\\
    &T^{2,\si}_{\al}=(\d^{B}_t-V^\ga \nab^A_\ga )\la^{\si}_{\al}-i\nab^A_{\al}\nab^{A,\si} \psi+\la^{\ga}_{\al}\Im(\psi\bar{\la}^{\si}_{\ga}) +\la_{\al\ga}\nab^\ga V^\si-\la_\ga^\si \nab_\al V^\ga,\\
    &T^3_\al=\d_t A_{\al}-\d_{\al} B-\Re(\la_{\al}^{\ga}\bar{\d}^A_{\ga}\bar{\psi})+\Im (\la^\ga_\al \bar{\la}_{\ga\si})V^\si.
\end{align*}
We need to show that $T^1=0$, $T^2 = 0$, $T^3 = 0$. To do this, we will show that 
$(T^1,T^2,T^3)$ solve a linear homogeneous coupled elliptic system of the form
\begin{equation*}
\left\{
\begin{aligned}
& \Delta_g T^1 =  \nabla(\Gamma T^1)+ \lambda^2  T^1 + \lambda T^2 ,
\\
& \nabla^{A,\alpha} T_\al^{2,\si} = \la T^3 + \la \nabla T^1 + T^1 \nabla \la,
\\
&\nabla^{A}_{\alpha} T_\be^{2,\si}- \nabla^{A}_{\be} T_\al^{2,\si} = \la T^3 + \la \nabla T^1 + T^1 \nabla \la,
\\
& \nabla^\al T^3_\al = T^1 \nabla A,
\\
& \nabla_\al T^3_\be -\nabla^\be T^3_\al = \la T^2.
\end{aligned}    
\right.
\end{equation*}
Considering this system for  $(T^1,T^2,T^3) \in \dot H^1 \times L^2 \times L^2$,
the smallness condition on the coefficients $(\la,h,V,A,B ) \in \SS$ insures that this system has the unique solution $(T^1,T^2,T^3)=0$. It remains to derive the system for $(T^1,T^2,T^3)$.

\begin{proof}[The equation for $T^1$] This has the form
\begin{equation}   \label{eqn-T1}
\begin{aligned}
         \De_g T^1_{\al\be}=& \  T^1_{\de\be} {\Ric^\de}_\al+T^1_{\de\al} {\Ric^\de}_\be+2T^{1,\mu\nu}R_{\al\mu\be\nu}-\nab_\be(T^{1,\mu\nu}\Ga_{\mu\nu,\al})-\nab_\al(T^{1,\mu\nu}\Ga_{\mu\nu,\be})\\
         &-2\Re(g_{\si\be}T^{2,\si}_\al\bar{\psi}+T^1_{\si\be}\la^\si_\al \psi+\bar{\la}_{\al\be}T^{2,\si}_\si-g_{\si\mu}T^{2,\si}_\al \bar{\la}^\mu_\be-T^1_{\si\mu}\la^\si_\al\bar{\la}^\mu_\be-\bar{\la}_{\al\si}T^{2,\si}_\be).
\end{aligned}
\end{equation}
We start with the first term in $T^1$, and compute the expression $\De_g \d_t g_{\al\be}$. We have

\begin{align*}
    \De_g \d_t g_{\al\be}=&\  g^{\mu\nu}(\d_\mu\nab_\nu \d_t g_{\al\be}-\Ga^\de_{\nu\al}\nab_\nu \d_t g_{\de\be}-\Ga^\de_{\mu\be}\nab_\nu \d_t g_{\al\de})\\
    =&\ [\d_t (g^{\mu\nu}\d_\mu\d_\nu g_{\al\be})-\d_tg^{\mu\nu}\d_\mu\d_\nu g_{\al\be}] +[-g^{\mu\nu}\Ga^\de_{\nu\al}\d_\mu\d_t g_{\de\be}-g^{\mu\nu}\Ga^\de_{\nu\be}\d_\mu\d_t g_{\de\al}\\
    &\ -g^{\mu\nu}\d_\mu\Ga^\de_{\nu\al}\d_t g_{\de\be}-g^{\mu\nu}\d_\mu\Ga^\de_{\nu\be}\d_t g_{\de\al}-g^{\mu\nu}(\Ga^\de_{\mu\al}\nab_\nu \d_t g_{\de\be}+\Ga^\de_{\mu\be}\nab_\nu \d_t g_{\de\al})]\\
    :=&\ I+II.
\end{align*}
We then use covariant derivatives to write $II$ as 
\begin{align*}
    II=&-g^{\mu\nu}\Ga^\de_{\mu\al}(2\nab_\nu\d_t g_{\de\be}+\Ga^\si_{\nu\de}\d_t g_{\si\be}+\Ga^\si_{\nu\be}\d_t g_{\si\de})\\
    &-g^{\mu\nu}\Ga^\de_{\mu\be}(2\nab_\nu\d_t g_{\de\al}+\Ga^\si_{\nu\de}\d_t g_{\si\al}+\Ga^\si_{\nu\al}\d_t g_{\si\de})\\
    &-g^{\mu\nu}\d_\mu\Ga^\de_{\nu\al}\d_t g_{\de\be}-g^{\mu\nu}\d_\mu\Ga^\de_{\nu\be}\d_t g_{\de\al}\\
    =& -2g^{\mu\nu}\Ga^\de_{\mu\al}\nab_\nu\d_t g_{\de\be}-2g^{\mu\nu}\Ga^\de_{\mu\be}\nab_\nu\d_t g_{\de\al}\\
    &-\d_t g_{\de\be} g^{\mu\nu}(\d_\mu\Ga^\de_{\nu\al}+\Ga^\si_{\mu\al}\Ga^\de_{\nu\si})-\d_t g_{\de\al} g^{\mu\nu}(\d_\mu\Ga^\de_{\nu\be}+\Ga^\si_{\mu\be}\Ga^\de_{\nu\si})\\
    &-2\d_t g_{\si\de}g^{\mu\nu}\Ga^\de_{\mu\al}\Ga^\si_{\nu\be}.
\end{align*}
For $I$, by the $g$ equation \eqref{g-eq-original} we have
\begin{align*}
    I=&\ \d_t[-\d_\al g^{\mu\nu}\d_\mu g_{\nu\be}-\d_\be g^{\mu\nu}\d_\mu g_{\nu\al}+\d_\al g^{\mu\nu}\d_\be g_{\mu\nu}]\\
    &+[2\d_t(g^{\mu\nu}\Ga_{\mu\al,\de}\Ga^\de_{\nu\be})-\d_t g^{\mu\nu}\d_\mu\d_\nu g_{\al\be}]-2\d_t \tRic_{\al\be}\\
    :=&\ I_1+I_2+I_3.
\end{align*}
The expression $I_1$ is written as
\begin{align*}
    I_1=& -\d_\al \d_t g^{\mu\nu}\Ga_{\mu\nu,\be}-\d_\be \d_t g^{\mu\nu}\Ga_{\mu\nu,\al}-\frac{1}{2}\d_\be \d_t g^{\mu\nu}\d_\al g_{\mu\nu}+\frac{1}{2}\d_\al\d_t g^{\mu\nu}\d_\be g_{\mu\nu}\\
    &-\d_\al g^{\mu\nu}\d_\mu\d_t g_{\nu\be}-\d_\be g^{\mu\nu}\d_\mu\d_t g_{\nu\al}+\d_\al g^{\mu\nu}\d_\be\d_t g_{\mu\nu}\\
    =&-(\nab_\al\d_t g^{\mu\nu}-2\Ga^\mu_{\al\de}\d_t g^{\de\nu})\Ga_{\mu\nu,\be}-(\nab_\be\d_t g^{\mu\nu}-2\Ga^\mu_{\be\de}\d_t g^{\de\nu})\Ga_{\mu\nu,\al}\\
    &+\frac{1}{2}[\nab_\al(\d_t g^{\mu\nu}\d_\be g_{\mu\nu}-\nab_\be(\d_t g^{\mu\nu}\d_\al g_{\mu\nu})]\\
    &-\d_\al g^{\mu\nu}(\nab_\mu\d_t g_{\nu\be}+\Ga^\de_{\mu\nu}\d_t g_{\de\be})-\d_\be g^{\mu\nu}(\nab_\mu\d_t g_{\nu\al}+\Ga^\de_{\mu\nu}\d_t g_{\de\al}+\Ga^\de_{\mu\al}\d_t g_{\nu\de})\\
    &+\d_\al g^{\mu\nu}(\nab_\be \d_t g_{\mu\nu}+\Ga^\de_{\be\nu}\d_t g_{\mu\de})\\
    =&\ \nab_\al \d_t g^{\mu\nu} (-\Ga_{\mu\nu,\be}+\Ga_{\mu\be,\nu})+\nab_\be\d_t g^{\mu\nu}(-\Ga_{\mu\nu,\al}+\Ga_{\mu\al,\nu})-\nab_\mu\d_t g_{\nu\be}\d_\al g^{\mu\nu}-\nab_\mu\d_t g_{\nu\al}\d_\be g^{\mu\nu}\\
    &+2\d_t g^{\mu\nu}(\Ga^\de_{\al\mu}\Ga_{\de\nu,\be}+\Ga^\de_{\be\mu}\Ga_{\de\nu,\al})+\d_t g^{\mu\nu}(-\Ga^\de_{\al\mu}\d_\be g_{\de\nu}+\Ga^\de_{\be\mu}\d_\al g_{\de\nu})\\
    &+\d_t g^{\mu\nu}(-\d_\be g_{\mu\si} g^{\si\de}\Ga_{\de\al,\nu}+\d_\al g_{\mu\si} g^{\si\de}\Ga_{\de\be,\nu})\\
    &-\d_t g_{\de\be}\d_\al g^{\mu\nu}\Ga^\de_{\mu\nu}-\d_t g_{\de\al}\d_\be g^{\mu\nu}\Ga^\de_{\mu\nu}
\end{align*}
For $I_2$, we first compute 
\begin{align*}
    2g^{\mu\nu}\d_t(\Ga_{\mu\al,\de}\Ga^\de_{\nu\be})=&g^{\mu\nu}\Ga^\de_{\nu\be}(\nab_\mu \d_t g_{\al\de}+\nab_\al \d_t g_{\mu\de}-\nab_\de \d_t g_{\mu\al})+4g^{\mu\nu}\Ga^\de_{\nu\be}\Ga^\si_{\al\mu}\d_t g_{\si\de}\\
    &+g^{\mu\nu}\Ga^\de_{\nu\al}(\nab_\mu \d_t g_{\be\de}+\nab_\be \d_t g_{\mu\de}-\nab_\de \d_t g_{\mu\be})+2\d_t g^{\si\de}g^{\mu\nu}\Ga_{\mu\al,\de}\Ga_{\nu\be,\si}
\end{align*}

By the above computations, we collect the $\nab\d_t g$ terms from $I_1$, $I_2$ and $II$
\begin{align*}
    &\nab_\al \d_t g^{\mu\nu} (-\Ga_{\mu\nu,\be}+\Ga_{\mu\be,\nu})+\nab_\be\d_t g^{\mu\nu}(-\Ga_{\mu\nu,\al}+\Ga_{\mu\al,\nu})-\nab_\mu\d_t g_{\nu\be}\d_\al g^{\mu\nu}-\nab_\mu\d_t g_{\nu\al}\d_\be g^{\mu\nu}\\
    &+g^{\mu\nu}\Ga^\de_{\nu\be}(\nab_\mu \d_t g_{\al\de}+\nab_\al \d_t g_{\mu\de}-\nab_\de \d_t g_{\mu\al})
    +g^{\mu\nu}\Ga^\de_{\nu\al}(\nab_\mu \d_t g_{\be\de}+\nab_\be \d_t g_{\mu\de}-\nab_\de \d_t g_{\mu\be})\\
    &-2g^{\mu\nu}\Ga^\de_{\mu\al}\nab_\nu\d_t g_{\de\be}-2g^{\mu\nu}\Ga^\de_{\mu\be}\nab_\nu\d_t g_{\de\al},
\end{align*}
where the terms containing $\nab\d_t g_{\nu\al}$ and $\nab\d_t g_{\nu\be}$  vanish, i.e.
\begin{align*}
    \nab_\mu\d_t g_{\nu\be}(-\d_\al g^{\mu\nu}-g^{\nu\de}\Ga^\mu_{\de\al}-g^{\de\mu}\Ga^\nu_{\de\al})+\nab_\mu\d_t g_{\nu\al}(-\d_\be g^{\mu\nu}-g^{\nu\de}\Ga^\mu_{\de\be}-g^{\mu\de}\Ga^\nu_{\de\be})=0,
\end{align*}
and the terms with $\nab\d_t g^{\mu\nu}$ were rewritten as
\begin{equation}        \label{dd_t-g_terms}
\begin{aligned}
    &-\nab_{\al}\d_t g^{\mu\nu}\Ga_{\mu\nu,\be}-\nab_{\be}\d_t g^{\mu\nu}\Ga_{\mu\nu,\al}\\
    =&-\nab_{\al}(\d_t g^{\mu\nu}\Ga_{\mu\nu,\be})-\nab_{\be}(\d_t g^{\mu\nu}\Ga_{\mu\nu,\al})+\d_t g^{\mu\nu}(\nab_{\al}\Ga_{\mu\nu,\be}+\nab_\be\Ga_{\mu\nu,\al})
\end{aligned}
\end{equation}

We collect the $\d_t g$ terms from $I$ and $II$ into
\begin{align*}
    &2\d_t g^{\mu\nu}(\Ga^\de_{\al\mu}\Ga_{\de\nu,\be}+\Ga^\de_{\be\mu}\Ga_{\de\nu,\al})
    -\d_t g_{\de\be}\d_\al g^{\mu\nu}\Ga^\de_{\mu\nu}-\d_t g_{\de\al}\d_\be g^{\mu\nu}\Ga^\de_{\mu\nu}\\
    &+\d_t g^{\mu\nu}(2\Ga_{\mu\al,\de}\Ga^\de_{\nu\be}-\d_\mu\d_\nu g_{\al\be})\\
    &-\d_t g_{\de\be} g^{\mu\nu}(\d_\mu\Ga^\de_{\nu\al}+\Ga^\si_{\mu\al}\Ga^\de_{\nu\si})-\d_t g_{\de\al} g^{\mu\nu}(\d_\mu\Ga^\de_{\nu\be}+\Ga^\si_{\mu\be}\Ga^\de_{\nu\si}).
\end{align*}
Adding the $\d_t g$ terms together with the third term in \eqref{dd_t-g_terms} we obtain
\begin{align*}
    &\d_t g^{\mu\nu}(\nab_{\al}\Ga_{\mu\nu,\be}+\nab_\be\Ga_{\mu\nu,\al}+2\Ga^\de_{\al\mu}\Ga_{\de\nu,\be}+2\Ga^\de_{\be\mu}\Ga_{\de\nu,\al}+2\Ga_{\mu\al,\de}\Ga^\de_{\nu\be}-\d_\mu\d_\nu g_{\al\be})\\
    =&\d_t g^{\mu\nu}(\d_{\al}\Ga_{\mu\nu,\be}+\d_\be\Ga_{\mu\nu,\al}-\d_\nu(\Ga_{\be\mu,\al}+\Ga_{\al\mu,\be})+2\Ga_{\mu\al,\de}\Ga^\de_{\nu\be}-2\Ga_{\al\be}^\de \Ga_{\mu\nu,\de})\\
    =& 2\d_t g^{\mu\nu}R_{\al\mu\be\nu}.
\end{align*}
Finally, using the harmonic coordinate condition  $g^{\mu\nu}\Ga^\de_{\mu\nu}=0$, the terms containing  the $\d_t g_{\de\al}$ expression are written as
\begin{align*}
    &-\d_t g_{\de\be}\d_\al g^{\mu\nu}\Ga^\de_{\mu\nu}-\d_t g_{\de\al}\d_\be g^{\mu\nu}\Ga^\de_{\mu\nu}
    -\d_t g_{\de\be} g^{\mu\nu}(\d_\mu\Ga^\de_{\nu\al}+\Ga^\si_{\mu\al}\Ga^\de_{\nu\si})-\d_t g_{\de\al} g^{\mu\nu}(\d_\mu\Ga^\de_{\nu\be}+\Ga^\si_{\mu\be}\Ga^\de_{\nu\si})\\
    & \qquad \qquad = \d_t g_{\de\be} {\Ric^\de}_\al+\d_t g_{\de\al}{\Ric^\de}_\be.
\end{align*}

Hence, the expression $\De_g\d_t g_{\al\be}$ is written as
\begin{equation}     \label{T^1-d_tg_terms}
    \begin{aligned}
    \De_g \d_t g_{\al\be}=&-\nab_{\al}(\d_t g^{\mu\nu}\Ga_{\mu\nu,\be})-\nab_{\be}(\d_t g^{\mu\nu}\Ga_{\mu\nu,\al})+ \d_t g_{\de\be}{\Ric^\de}_\al +\d_t g_{\de\al}{\Ric^\de}_\be\\
    &+2\d_t g^{\mu\nu}R_{\al\mu\be\nu}-2\d_t \tRic_{\al\be}.
\end{aligned}
\end{equation}
For the last term $-2\d_t \tRic_{\al\be}$, using the expression  $T^2$ we have
\begin{align}     \nonumber
    -2\d_t \tRic_{\al\be}=& -2\Re(g_{\si\be}T^{2,\si}_\al\bar{\psi}+T^1_{\si\be}\la^\si_\al \psi+\bar{\la}_{\al\be}T^{2,\si}_\si-g_{\si\mu}T^{2,\si}_\al \bar{\la}^\mu_\be-T^1_{\si\mu}\la^\si_\al\bar{\la}^\mu_\be-\bar{\la}_{\al\si}T^{2,\si}_\be)\\\tag{$I_{31}$}
    &+2\Im (\nab^A_\al\nab^A_\be \psi\bar{\psi}+\nab^A_\si \nab^{A,\si}\psi\bar{\la}_{\al\be}-\nab^A_\al\nab^A_\si \psi\bar{\la}^\si_\be-\nab^A_\be\nab^A_\si \psi\bar{\la}^\si_\al)\\\tag{$I_{32}$}
    &-2\tRic_{\be\ga}\nab_\al V^\ga-2\tRic_{\al\ga}\nab_\be V^\ga-2\nab_\ga \tRic_{\al\be} V^\ga\\\tag{$I_{33}$}
    &-2\Re(\psi\bar{\la}_{\al\ga})\Im(\psi\bar{\la}^\ga_\be)+2\Re(\la_{\al\be}\bar{\la}^{\si\ga})\Im(\psi\bar{\la}_{\si\ga}).
\end{align}

Next, we compute  
\begin{align*}
    III:=&\ -\De_g(2\Im(\psi\bar{\la}_{\al\be})+\nab_\al V_\be+\nab_\be V_\al)\\
    =&\ -2\nab^\si\nab_\si\Im(\psi\bar{\la}_{\al\be})+[\De_g,\nab_\al] V_\be-[\De_g,\nab_\be] V_\al-\nab_\al\De_g V_\be-\nab_\be \De_g V_\al\\
    =&\  -2\nab^\si\nab_\si\Im(\psi\bar{\la}_{\al\be})-\nab_\be \Ric_{\al\ga}V^\ga-\nab_\al \Ric_{\be\ga}V^\ga-2\nab_\ga \Ric_{\al\be} V^\ga\\
    &\ -\Ric_{\al\ga}\nab^\ga V_\be-\Ric_{\be\ga}\nab^\ga V_\al
    +2R_{\al\si\be\de}(\nab^\si V^\de+\nab^\de V^\si) -\nab_\al\De_g V_\be-\nab_\be \De_g V_\al
\end{align*}
Using the $V$-equation \eqref{Ellp-X} we write the last two terms as
\begin{align*}
    &-\nab_\al\De_g V_\be-\nab_\be \De_g V_\al\\
    =&2\nab_\al\nab_\si \Im(\psi\bar{\la}^\si_\be)+2\nab_\be\nab_\si \Im(\psi\bar{\la}^\si_\al)+\nab_\al \Ric_{\si\be} V^\si+\nab_\be\Ric_{\si\al}V^\si\\
    &+ \Ric_{\si\be}\nab_\al V^\si+\Ric_{\si\al}\nab_\be V^\si+\nab_\al(\widetilde{\d_t g^{\mu\nu}}\Ga_{\mu\nu,\be})+\nab_\be(\widetilde{\d_t g^{\mu\nu}}\Ga_{\mu\nu,\al}),
\end{align*}
where $\widetilde{\d_t g^{\mu\nu}}$ denotes the expression
\begin{equation*}
    \widetilde{\d_t g^{\mu\nu}}:=\d_t g^{\mu\nu}-T^{1,\mu\nu}.
\end{equation*}
We then add $I_{31}$ together with $\nab^2\Im (\psi\la)$ in $III$ to get
\begin{align*}
    &I_{31}-2\nab^\si\nab_\si\Im(\psi\bar{\la}_{\al\be})+2\nab_\al\nab_\si \Im(\psi\bar{\la}^\si_\be)+2\nab_\be\nab_\si \Im(\psi\bar{\la}^\si_\al)\\
    =&-2\Ric_{\be\de}\Im (\psi\bar{\la}^\de_\al)+2R_{\al\si\be\de}\Im(\psi\bar{\la}^{\si\de})+2\Re(\la_\al^\si\bar{\psi})\Im(\la_{\si\mu}\bar{\la}^\mu_\be).
\end{align*}
The last term and $I_{33}$ can be further written as
\begin{align*}
    &2\Re(\la_\al^\si\bar{\psi})\Im(\la_{\si\mu}\bar{\la}^\mu_\be)+I_{33}\\
    =&2R_{\al\si\be\de}\Im(\psi\bar{\la}^{\si\de})-2\Ric_{\al\de}\Im(\psi\bar{\la}^\de_\be).
\end{align*}
Hence, given the expressions of $I_{3}$ and $III$, we obtain
\begin{align*}
    &I_3+III\\
    =& -2\Re(g_{\si\be}T^{2,\si}_\al\bar{\psi}+T^1_{\si\be}\la^\si_\al \psi+\bar{\la}_{\al\be}T^{2,\si}_\si-g_{\si\mu}T^{2,\si}_\al \bar{\la}^\mu_\be-T^1_{\si\mu}\la^\si_\al\bar{\la}^\mu_\be-\bar{\la}_{\al\si}T^{2,\si}_\be)\\
    &-{\Ric^\si}_\be \widetilde{\d_t g_{\al\si}}-{\Ric^\si}_\al \widetilde{\d_t g_{\be\si}}-2R_{\al\si\be\de}\widetilde{\d_t g^{\si\de}}+\nab_\al(\widetilde{\d_t g^{\mu\nu}}\Ga_{\mu\nu,\be})+\nab_\be(\widetilde{\d_t g^{\mu\nu}}\Ga_{\mu\nu,\al}),
\end{align*}
which combined with \eqref{T^1-d_tg_terms} yields the $T^1$-equation \eqref{eqn-T1}.
\end{proof}

\begin{proof}[The equation for $T^2$] This has the form
\begin{equation}       \label{eqn-T2}
    \left\{\begin{aligned}
         &\begin{aligned}
         \nab^{A,\al}T^{2,\si}_\al=& ig^{\si\mu}\psi T^3_\mu-i\la^{\be\si}T^3_\be+g^{\si\de}\la^{\al\be} (-\nab_\al T^1_{\be\de}+\frac{1}{2}\nab_\de T^1_{\al\be})\\
         & -T^{1,\al\be}(\nab^A_\be \la^\si_\al+\Ga^\mu_{\al\be} \la^\si_\mu)+T^{1,\si\mu}\nab^A_\mu\psi,
         \end{aligned}\\
         &\begin{aligned}
         \nab^A_\al T^{2,\si}_\be-\nab^A_\be T^{2,\si}_\al=&\frac{1}{2} g^{\si\ga}[-\la^\mu_\be (\nab_\al T^1_{\mu\ga}+\nab_\mu T^1_{\al\ga}-\nab_\ga T^1_{\al\mu})\\
         &+\la^\mu_\al (\nab_\be T^1_{\mu\ga}+\nab_\mu T^1_{\be\ga}-\nab_\ga T^1_{\be\mu})]-iT^3_\al \la^\si_\be+iT^3_\be \la^\si_\al.
         \end{aligned}
         \end{aligned}\right.
\end{equation}

We compute the divergence of $T^2$ in \eqref{eqn-T2} first. Applying $\nab^{A,\al}$ to $T^{2,\si}_\al$, we have 
\begin{align*}
    \nab^{A,\al}T^{2,\si}_\al
    =& [\nab^{A,\al},\d_t^B-V^\ga \nab^A_\ga ]\la^\si_\al+[\d_t^B-V^\ga \nab^A_\ga,\nab^{A,\si}]\psi+\nab^{A,\si}(\d_t^B-V^\ga \nab^A_\ga) \psi\\
    &+\nab^{A,\al}(\la_\al^\ga \Im(\psi \bar{\la}^\si_\ga))-i\nab^{A,\al}\nab^A_\al \nab^{A,\si}\psi\\
    &+\nab^{A}_\ga \psi \nab^\ga V^\si -\nab^{A,\si}\la_{\al\ga}\nab^\al V^\ga+\la^{\al\ga}\nab_\al\nab_\ga V^\si-\la^\si_\ga \De_g V^\ga.
\end{align*}
Three of the terms  on the right-hand side are written as
\begin{align*}
    &[\nab^{A,\al},\d_t^B-V^\ga\nab^A_\ga]\la^\si_\al-\nab^{A,\si}\la_{\al\ga}\nab^\al V^\ga+\la^{\al\ga}\nab_\al\nab_\ga V^\si\\
    =&g^{\al\be}(\nab_\be\d_t \la^\si_\al-\d_t \nab_\be \la^\si_\al)+ig^{\al\be}(\nab_\be B-\d_t A_\be)\la^\si_\al-\d_t g^{\al\be} \nab^A_{\be} \la^\si_\al\\
    &+\la^{\al\ga}\nab_\al\nab_\ga V^\si-2\nab^{A,\si}\la_{\al\ga}\nab^\al V^\ga-V^\ga [\nab^{\al},\nab_\ga]\la^\si_\al-iV_\ga \bmF^{\al\ga}\la^\si_\al\\
    =&-\d_t g^{\al\be} (\nab^{A,\si} \la_{\al\be}+\Ga^\mu_{\al\be} \la^\si_\mu)-\d_t\Ga^\si_{\be\mu}\la^{\be\mu}+\la^{\al\ga}\nab_\al\nab_\ga V^\si\\
    &-i(\d_t A_\be-\nab_\be B)\la^{\be\si}-iV_\ga \bmF^{\al\ga}\la^\si_\al
    -2\nab^{A,\si}\la_{\al\ga}\nab^\al V^\ga-V^\ga [\nab^{\al},\nab_\ga]\la^\si_\al
\end{align*}
We can further use $T^1$ to rewrite the last two terms on the first line above as
\begin{align*}
    &-\d_t \Ga^\si_{\al\be}\la^{\al\be}+\la^{\al\ga}\nab_\al\nab_\ga V^\si\\
    =&-\d_t g^{\si\de}\Ga_{\al\be,\de}\la^{\al\be}-g^{\si\de}\d_t(\d_\al g_{\be\de}-\frac{1}{2}\d_\de g_{\al\be})\la^{\al\be}+\la^{\al\ga}\nab_\al\nab_\ga V^\si\\
    =& g^{\si\de}\la^{\al\be}(\d_t g_{\mu\de}\Ga_{\al\be}^\mu-\d_{\al}\d_t g_{\be\de}+\frac{1}{2}\d_\de \d_t g_{\al\be})+\la^{\al\ga}\nab_\al\nab_\ga V^\si\\
    =& g^{\si\de}\la^{\al\be} (-\nab_\al \d_t g_{\be\de}+\frac{1}{2}\nab_\de \d_t g_{\al\be})+\la^{\al\ga}\nab_\al\nab_\ga V^\si\\
    =& \la_{\mu\nu} (\nab^\mu T^{1,\nu\si}-\frac{1}{2}\nab^\si T^{1,\mu\nu})\\
    &+\la^{\al\be}[-2\nab_\al\Im(\psi\bar{\la}^{\si}_\be)+\nab^\si \Im(\psi \bar{\la}_{\al\be})-[\nab_\al,\nab^{\si}]V_{\be}]
\end{align*}
and the following term as
\begin{align*}
    -i(\d_t A_\be-\nab_\be B)\la^{\be\si}-iV_\ga \bmF^{\al\ga}\la^\si_\al=& -i\la^{\be\si}T^3_\be-i\la^{\be\si}\Re(\la_\be^\ga \overline{\nab^A_\ga \psi}).
\end{align*}
Similarly, we compute the second commutator by
\begin{align*}
    [\d_t^B-V^\ga\nab^A_\ga,\nab^{A,\si}]\psi+\nab^{A}_\ga \psi \nab^\ga V^\si=&\d_t g^{\si\mu} \nab^A_\mu \psi+ig^{\si\mu}\psi T^3_\mu+i\psi\Re(\la^{\si\ga}\overline{\nab^A_\ga\psi})\\
    &+\nab^A_\ga \psi (\nab^\ga V^\si+\nab^\si V^\ga).
\end{align*}
Hence, using $T^{2,\al}_\al$ and the $V$ equation \eqref{Ellp-X} we reorganize the expression of $\nab^{A,\al}T^{2,\si}_\al$ and obtain
\begin{align*}
    \nab^{A,\al}T^{2,\si}_\al=&ig^{\si\mu}\psi T^3_\mu-i\la^{\be\si}T^3_\be+g^{\si\de}\la^{\al\be} (-\nab_\al T^1_{\be\de}+\frac{1}{2}\nab_\de T^1_{\al\be})\\
    &-\d_t g^{\al\be} (\nab^A_\be \la^\si_\al+\Ga^\mu_{\al\be} \la^\si_\mu)
    +\la^{\al\be}[-2\nab_\al\Im(\psi\bar{\la}^{\si}_\be)+\nab^\si \Im(\psi \bar{\la}_{\al\be})]\\
    &-i\la^{\be\si}\Re(\la_\be^\ga \overline{\nab^A_\ga \psi})
    +\d_t g^{\si\mu} \nab^A_\mu \psi+i\psi\Re(\la^{\si\ga}\overline{\nab^A_\ga\psi})\\
    &-\nab^{A,\si}(\la^\ga_\al\Im(\psi\bar{\la}^\al_\ga))
    +\nab^{A,\al}(\la_\al^\ga \Im(\psi \bar{\la}^\si_\ga))\\
    &-i{\Ric^\si}_\de \nab^{A,\de}\psi-\nab_\al \bmF^{\si\al}\psi-2\bmF^{\si\al}\nab^A_\al\psi\\
    &+\nab^{A}_\ga \psi(\nab^\ga V^\si+\nab^\si V^\ga)  -2\nab^{A,\si}\la_{\al\ga}\nab^\al V^\ga\\
    &-2\la^\si_\ga \nab_\al \Im(\la^{\al\ga}\bar{\psi})+\la^\si_\ga \widetilde{\d_t g^{\al\be}}\Ga_{\al\be}^\ga.
\end{align*}
Using  $T^{2,\al}_\al$ and the $V$-equation \eqref{Ellp-X}, we have 
\begin{align*}
    &\la^{\al\be}[-2\nab_\al\Im(\psi\bar{\la}^{\si}_\be)+\nab^\si \Im(\psi \bar{\la}_{\al\be})]-i\la^{\be\si}\Re(\la_\be^\ga \overline{\nab^A_\ga \psi})\\
    &+i\psi\Re(\la^{\si\ga}\overline{\nab^A_\ga\psi})+\nab^{A,\al}(\la_\al^\ga \Im(\psi \bar{\la}^\si_\ga))-i{\Ric^\si}_\de \nab^{A,\de}\psi-\nab_\al \bmF^{\si\al}\psi-2\bmF^{\si\al}\nab^A_\al\psi\\
    &+2\la^\si_\ga \nab^\al \Im(\psi \bar{\la}^\ga_\al)-\nab^{A,\si}(\la^\ga_\al\Im(\psi\bar{\la}^\al_\ga))\\
    =& 2\nab^A_\ga\psi\Im(\psi\bar{\la}^{\si\ga})-2\nab^{A,\si}\la_{\al\be}\Im(\psi\bar{\la}^{\al\be})
\end{align*}
Combining  these two expressions, we obtain
\begin{align*}
    \nab^{A,\al}T^{2,\si}_\al=& ig^{\si\mu}\psi T^3_\mu-i\la^{\be\si}T^3_\be+g^{\si\de}\la^{\al\be} (-\nab_\al T^1_{\be\de}+\frac{1}{2}\nab_\de T^1_{\al\be})\\
    & -T^{1,\al\be}(\nab^A_\be \la^\si_\al+\Ga^\mu_{\al\be} \la^\si_\mu)+T^{1,\si\mu}\nab^A_\mu\psi
    +\nab^{A,\si} T^{2,\al}_\al.
\end{align*}

\medskip

Next we compute the \emph{curl} of $T^2$ in \eqref{eqn-T2}. By $T^2$ we have
\begin{align*}
&\nab^A_\al T^{2,\si}_\be-\nab^A_\be T^{2,\si}_\al\\
=&[\nab^A_\al,\d_t^B-V^\ga\nab^A_\ga]\la^\si_\be-[\nab^A_\be,\d_t^B-V^\ga\nab^A_\ga]\la^\si_\al+\la^\ga_\be \nab_\al\Im(\psi\bar{\la}^\si_\ga)-\la^\ga_\al \nab_\be\Im(\psi\bar{\la}^\si_\ga)\\
&-i[\nab^A_\al,\nab^A_\be]\nab^{A,\si}\psi+\la^\ga_\be \nab_\al\nab_\ga V^\si-\la^\ga_\al \nab_\be\nab_\ga V^\si\\
&-\la^\si_\ga [\nab_\al,\nab_\be]V^\ga-\nab^A_\ga\la^\si_\al \nab_\be V^\ga+\nab^A_\ga\la^\si_\be \nab_\al V^\ga.
\end{align*}
We use $T^1$ to rewrite six of the terms on the right-hand side as 
\begin{align}\nonumber
   &[\nab^A_\al,\d_t^B-V^\ga\nab^A_\ga]\la^\si_\be-[\nab^A_\be,\d_t^B-V^\ga\nab^A_\ga]\la^\si_\al+\la^\ga_\be \nab_\al\nab_\ga V^\si-\la^\ga_\al \nab_\be\nab_\ga V^\si\\\nonumber
   &-\nab^A_\ga\la^\si_\al \nab_\be V^\ga+\nab^A_\ga\la^\si_\be \nab_\al V^\ga\\\nonumber
   =&\frac{1}{2} g^{\si\ga}[-\la^\mu_\be (\nab_\al T^1_{\mu\ga}+\nab_\mu T^1_{\al\ga}-\nab_\ga T^1_{\al\mu})+\la^\mu_\al (\nab_\be T^1_{\mu\ga}+\nab_\mu T^1_{\be\ga}-\nab_\ga T^1_{\be\mu})]\\\nonumber
   &-iT^3_\al \la^\si_\be+iT^3_\be \la^\si_\al\\   \nonumber
   & -\la^\mu_\be\nab_\al \Im(\psi\bar{\la}^\si_\mu)+\la^\mu_\al\nab_\be \Im(\psi\bar{\la}^\si_\mu)\\   \tag{$I_1$}
   &+[\la^\mu_\be(-\nab_\mu\Im(\psi\bar{\la}^\si_\al)+\nab^\si \Im(\psi\bar{\la}_{\al\mu}))-\la^\mu_\al(-\nab_\mu\Im(\psi\bar{\la}^\si_\be)+\nab^\si \Im(\psi\bar{\la}_{\be\mu}))\\\nonumber
   &-i\Re(\la^\ga_\al\overline{\nab^A_\ga\psi})\la^\si_\be+i\Re(\la^\ga_\be\overline{\nab^A_\ga\psi})\la^\si_\al]\\\tag{$I_2$}
   &+[\frac{1}{2}\la^\mu_\be(R_{\al\mu\si\de}+R^\si_{\al\mu\de}+R^\si_{\mu\al\de})V^\de-\frac{1}{2}\la^\mu_\al(R_{\be\mu\si\de}+R^\si_{\be\mu\de}+R^\si_{\mu\be\de})V^\de\\\nonumber
   &-V^\ga R_{\al\ga\si\de}\la^\de_\be-V^\ga R_{\al\ga\be\de}\la^{\si\de}+V^\ga R_{\be\ga\si\de}\la^\de_\al+V^\ga R_{\be\ga\al\de}\la^{\si\de} ].
\end{align}
Then we use Bianchi identities and compatibility conditions to compute $I_1$ and $I_2$ by
\begin{align*}
I_1=i[\nab^A_\al,\nab^A_\be]\nab^{A,\si}\psi
\end{align*}
and 
\begin{align*}
I_2=V^\ga R_{\be\ga\si\de}\la^\de_\al+V^\ga R_{\be\ga\al\de}\la^{\si\de}=\la^\si_\ga [\nab_\al,\nab_\be]V^\ga.
\end{align*}
Hence, we obtain
\begin{align*}
\nab^A_\al T^{2,\si}_\be-\nab^A_\be T^{2,\si}_\al=&\frac{1}{2} g^{\si\ga}[-\la^\mu_\be (\nab_\al T^1_{\mu\ga}+\nab_\mu T^1_{\al\ga}-\nab_\ga T^1_{\al\mu})\\
&+\la^\mu_\al (\nab_\be T^1_{\mu\ga}+\nab_\mu T^1_{\be\ga}-\nab_\ga T^1_{\be\mu})]-iT^3_\al \la^\si_\be+iT^3_\be \la^\si_\al.
\end{align*}
This completes the derivation of \eqref{eqn-T2}.
\end{proof}

\begin{proof}[The equation for $T^3$] This has the form
\begin{equation*}
    \left\{\begin{aligned}
         &\nab^\al T^3_{\al}=-T^{1,\al\be}\d_\al A_\be,\\
         &\nab_\al T^3_\be-\nab_\be T^3_\al =\Im (T^{2,\si}_{\ \al}\bar{\la}_{\si\be}+\la^\si_\al \overline{T^{2}_{\si\be}}).
     \end{aligned}\right.
\end{equation*}
Applying $\nab^\al$ to $T^3_\al$, we then use the Coulomb condition $\nab^\al A_\al=0$ and the $B$-equation \eqref{Ellip-B} to get
\begin{align*}
    \nab^\al T^3_\al=& \nab^\al \d_t A_\al -\De_g B-\nab^\al \Re(\la^\si_\al\overline{\nab^A_\si \psi}+i\la^\si_\al \bar{\la}_{\si\ga}V^\ga)\\
    =& g^{\al\be}\d_\be \d_t A_\al +\d_t g^{\be\ga}\d_\be A_\ga-T^{1,\be\ga}\d_\be A_\ga=-T^{1,\al\be}\d_\al A_\be.
\end{align*}
The \emph{curl} of $T^3$ is obtained by \eqref{cpt-AiAj} directly.
\end{proof}
\end{proof}

\subsection{ The moving frame}
Here we undertake the task of reconstructing the frame $(F_\alpha, m)$. For this we use the system consisting of \eqref{strsys-cpf} and \eqref{mo-frame}, viewed as a linear ode. We recall these equations here:
\begin{equation}               \label{strsys-cpf-re}
\left\{\begin{aligned}
&\d_{\al}F_{\be}=\Gamma^{\ga}_{\al\be}F_{\ga}+\Re(\lambda_{\al\be}\bar{m}),\\
&\d_{\al}^A m=-\lambda^{\ga}_{\al} F_{\ga},
\end{aligned}\right.
\end{equation}
respectively
\begin{equation}              \label{mo-frame-re}
\left\{\begin{aligned}
&\d_t F_{\al}=-\Im (\d^A_{\al} \psi \bar{m}-i\la_{\al\ga}V^{\ga} \bar{m})+[\Im(\psi\bar{\la}^{\ga}_{\al})+\nab_{\al} V^{\ga}]F_{\ga},\\
&\d^{B}_t m=-i(\d^{A,\al} \psi -i\la^{\al}_{\ga}V^{\ga} )F_{\al}.
\end{aligned}\right.
\end{equation}

We start with the frame at time $t=0$, which 
already is known to solve \eqref{strsys-cpf-re}, and 
has the following properties:
\begin{enumerate}[label=(\roman*)]
    \item \emph{Orthogonality}, $F_\alpha \perp m$, $\<m,m\>=2$, $\<m,\bar{m}\>=0$ and consistency with the metric $g_{\al\be} = \< F_\al,F_\be \>$.  

    \item \emph{Integrability}, $\partial_\beta F_\alpha = \partial_\alpha F_\beta$.
    \item \emph{Consistency}  with the second fundamental form and the connection $A$:
\[
 \d_\al F_\be\cdot m=\la_{\al\be}, \qquad \< \partial_\alpha m,m\>= -2 i A_\alpha.
\]
    
\end{enumerate}

Next we extend this frame to times $t > 0$ by simultaneously solving the pair of equations \eqref{strsys-cpf-re} and \eqref{mo-frame-re}.
To avoid some technical difficulties, we first do this for regular solutions, i.e. $s > d/2 +2$, and then pass to the limit to obtain the frame 
for rough solutions.

\subsubsection{ The frame associated to smooth solutions}
The system consisting of \eqref{strsys-cpf-re} and \eqref{mo-frame-re}
is overdetermined, and the necessary and sufficient condition 
for existence of solutions is provided by Frobenius' theorem.
We now verify these compatibility conditions in two steps:

\bigskip

a) Compatibility conditions for the system \eqref{strsys-cpf-re} at fixed time. Here, by $C^2_{\al\be}=0$, $C^3_{\al\be}=0 $ and $C^7_{\al\be\mu\nu}=0$ we have
\begin{align*}
    &\d_{\al}(\Ga^\si_{\be\ga}F_\si+\Re(\la_{\be\ga}\bar{m}))-\d_{\be}(\Ga^\si_{\al\ga}F_\si+\Re(\la_{\al\ga}\bar{m}))
    = C^7_{\si\ga\al\be}F^\si=0,
\end{align*}
and 
\[
\d_\al(iA_\be m+\la^\si_\be F_\si)-\d_\be(iA_\al m+\la^\si_\al F_\si)=iC^3_{\al\be}m=0,
\]
as needed.

\bigskip

b) Between the system \eqref{strsys-cpf-re} and \eqref{mo-frame-re}.
By \eqref{strsys-cpf-re} and \eqref{mo-frame-re} we have
\begin{align*}
    \d_t(iA_\al m+\la^\si_\al F_\si)-\d_\al (iB m+i(\d^{A,\si}\psi-i\la^\si_\ga V^\ga)F_\si)
    = iT^3_{\al}m
    +T^{2\si}_\al F_\si
\end{align*}
and
\begin{equation}         \label{d_tal F}
\begin{aligned}
    &\d_\be[-\Im (\d^A_{\al} \psi \bar{m}-i\la_{\al\ga}V^{\ga} \bar{m})+[\Im(\psi\bar{\la}^{\ga}_{\al})+\nab_{\al} V^{\ga}]F_{\ga}]-\d_t[\Ga^\ga_{\be\al}F_\ga+\Re(\la_{\be\al}\bar{m})]\\
    =&-\Re[(g_{\si\al}T^{2\si}_\be+\la^\si_\be T^1_{\si\al})\bar{m}]
    -T^{1\ga\si}\Ga_{\be\al,\si}F_\ga-\frac{1}{2}(\d_\be T^1_{\al\si}+\d_\al T^1_{\be\si}-\d_\si T^1_{\be\al})F^\si.
\end{aligned}
\end{equation}
The first equality is obtained directly. For the second equality \eqref{d_tal F}, by \eqref{strsys-cpf} and \eqref{mo-frame} we compute this by 
\begin{align*}
    {\rm LHS}\eqref{d_tal F}
    =&-\Re[(g_{\si\al}T^{2\si}_\be+\la^\si_\be T^1_{\si\al})\bar{m}]
    +\nab_\be(\Im(\psi\bar{\la}_{\si\al})+\nab_\al V_\si) F^\si\\
    &+\Im (\nab^A_\al \psi\bar{\la}_{\si\be}-\nab^A_\si \psi \bar{\la}_{\al\be})F^\si-\tR_{\si\al\be\ga}V^\ga F^\si-\d_t \Ga^\ga_{\be\al} F_\ga.
\end{align*}
By $T^1$ we compute the last term by
\begin{align*}
    -\d_t \Ga^\ga_{\be\al} F_\ga
    =& -(T^{1\ga\si}-2\Im(\psi \bar{\la}^{\ga\si})-\nab^\ga V^\si-\nab^\si V^\ga)\Ga_{\be\al,\si}F_\ga\\
    & -\frac{1}{2}[\d_\be(T^1_{\al\si}+2\Im (\psi\bar{\la}_{\al\si})+\nab_\al V_\si +\nab_\si V_\al)]F^\si\\
    &-\frac{1}{2}[\d_\al(T^1_{\be\si}+2\Im (\psi\bar{\la}_{\be\si})+\nab_\be V_\si +\nab_\si V_\be)]F^\si\\
    &+\frac{1}{2}[\d_\si(T^1_{\be\al}+2\Im (\psi\bar{\la}_{\be\al})+\nab_\be V_\al +\nab_\al V_\be)]F^\si\\
    =& -T^{1\ga\si}\Ga_{\be\al,\si}F_\ga-\frac{1}{2}(\d_\be T^1_{\al\si}+\d_\al T^1_{\be\si}-\d_\si T^1_{\be\al})F^\si\\
    &+[-\nab_\be \Im (\psi \bar{\la}_{\al\si})-\Im (\nab^A_\al \psi \bar{\la}_{\be\si})+\Im (\nab^A_\si \psi\bar{\la}_{\be\al})\\
    &-\frac{1}{2}(\nab_\al \nab_\be+\nab_\be \nab_\al)V_\si -\frac{1}{2}[\nab_\be,\nab_\si]V_\al-\frac{1}{2}[\nab_\al,\nab_\si]V_\be]F^\si.
\end{align*}
Then by Bianchi identities and compatibility condition we have
\begin{align*}
&\frac{1}{2}([\nab_\be,\nab_\al]V_\si-[\nab_\be,\nab_\si]V_\al-[\nab_\al,\nab_\si]V_\be)-\tR_{\si\al\be\ga}V^\ga\\
=& \frac{1}{2}( R_{\be\al\si \ga}-R_{\be\si\al\ga}-R_{\al\si\be\ga}-2R_{\si\al\be\ga})V^\ga=0.
\end{align*}
From the above expressions the equality \eqref{d_tal F} follows.

\medskip

Once the compatibility conditions in Frobenius' theorem 
are verified, we obtain the frame $(F_\alpha,m)$
for $t \in [0,1]$. For this we can easily obtain the regularity
\begin{equation*}
\partial_x(F_\alpha, m) \in L^\infty H^{s+2}, \qquad \partial_t     (F_\alpha, m)
\in L^\infty H^{s+1}.
\end{equation*}
Finally, we show that the properties (i)-(iii) above also extend to all $t \in [0,1]$. 
The properties (ii) and (iii) follow directly from the equations \eqref{strsys-cpf-re} and \eqref{mo-frame-re} once 
the orthogonality conditions in (i) are verified.
For (i) we  denote 
\[
\tg_{00}=\<m,m\>,\quad \tg_{\al 0}=\<F_\al, m\>,\quad \tg_{\al\be}=\<F_\al,F_\be\>.
\]
Then by  \eqref{mo-frame} and $T^1_{\al\be}=0$,
we have
\begin{align*}
&\begin{aligned}
    \d_t \tg_{\al 0}=&-\frac{i}{2}(\overline{\d^A_\al \psi}+i \bar{\la}_{\al\ga}V^\ga)(\tg_{00}-2)-i(\overline{\d^{A,\si} \psi}+i \bar{\la}^\si_{\ga}V^\ga)(g_{\al\si}-\tg_{\al\si})\\
    &+\frac{i}{2}(\d^A_\al \psi+i \la_{\al\ga}V^\ga)\<\bar{m},m\>+(\Im(\psi\bar{\la}_\al^\ga)+\nab_\al V^\ga)\tg_{\ga 0}+iB \tg_{\al 0},
\end{aligned}
    \\
    &\d_t (\tg_{00}-2)=2\Im (\d^{A,\al}\psi-i \la^\al_\ga V^\ga) \tg_{\al 0},\\
    &\d_t \<m,\bar{m}\>= -iB\<m,\bar{m}\>-i (\d^{A,\al}\psi-i \la^\al_\ga V^\ga) \tg_{\al 0},\\
    &\begin{aligned}
    \d_t(g_{\al\be}-\tg_{\al\be})=&(\Im(\psi\bar{\la}_\al^\ga)+\nab_\al V^\ga)(g_{\be\ga}-\tg_{\be\ga})+(\Im(\psi\bar{\la}_\be^\ga)+\nab_\be V^\ga)(g_{\al\ga}-\tg_{\al\ga})\\
    &+\Im (\d^A_\al \psi \tg_{\be 0}-i\la_{\al\ga}V^\ga \bar{\tg}_{\be 0})+\Im (\d^A_\be \psi \tg_{\al 0}-i\la_{\be\ga}V^\ga \bar{\tg}_{\al 0}).
    \end{aligned}
\end{align*}
Viewed as a linear system of ode's in time, these equations allow us to propagate (i) in time.

\subsubsection{ The frame associated to rough solutions}
Here we use our approximation of rough solutions with smooth solutions
for the $\psi$ equation in order to construct the frame in the rough case. Precisely, 
given a small initial data $\psi_0 \in H^s$,
there exists a sequence $\{\psi_{0n}\}\in H^{s+2}$ such that $\|\psi_{0n}-\psi_0\|_{H^s}\rightarrow 0$. By Theorem \ref{LWP-MSS-thm}, the Schr\"odinger system \eqref{mdf-Shr-sys-2} coupled with \eqref{ell-syst} admits solutions $\psi_n$ with $\psi_n(0)=\psi_{0n}$ and 
\begin{align*}
    \|\psi_n\|_{H^{s+2}}\lesssim \|\psi_{0n}\|_{H^{s+2}},\quad \|\psi_n-\psi\|_{H^{s}}\lesssim \|\psi_{0n}-\psi_0\|_{H^{s}}\rightarrow 0.
\end{align*}

A-priori, we do not know whether the initial data $\psi_{0n}$ is associated to a frame at the initial time.
Hence we first use \eqref{strsys-cpf-re} to construct the frame $(F_\al^{(n)},m^{(n)})$ associated with $\psi_{0n}$ at $t=0$. At some point $x_0$, we choose $F^{(n)}_\alpha(x_0)$ and $m^{(n)}(x_0)$ so that they are orthogonal, and $\<m^{(n)},m^{(n)}\>=2$, $\<m^{(n)},\bar{m}^{(n)}\>=0$ and $\<F^{(n)}_\al,F^{(n)}_\be\>=g^{(n)}_{\al\be}$ hold. With this initial data,
we view \eqref{strsys-cpf-re} as a linear ode with continuous coefficients. As above, the necessary and sufficient condition for solvability, as provided by Frobenius' theorem, is a consequence of the relations 
$C^2=0$, $C^3 = 0$ and $C^7 = 0$, which are in turn a consequence of Theorem~\ref{t:ell-fixed-time}.

The above construction determines the frame $(F^{(n)}_\alpha,m^{(n)})$ up to symmetries (rigid rotations and translations). Hence, the frame $(F_\al^{(n)},m^{(n)})$ at $t=0$ is uniquely determined by the condition 
\[
\lim_{x\rightarrow \infty}(F_{\al}^{(n)},m^{(n)})(x,0)=\lim_{x\rightarrow \infty}(F_\al,m)(x,0).
\]
In this construction, the properties (i)-(iii) above also extend to all $x$. The properties (ii) and (iii) follow directly from equation \eqref{strsys-cpf-re} once the orthogonality conditions in (i) are verified. For (i) we use \eqref{strsys-cpf-re} to compute
\begin{align*}
&\d_\al \tg_{\be 0}=\Ga^\ga_{\al\be} \tg_{\ga 0}+\frac{1}{2}\la_{\al\be}\<\bar{m},m\>+\frac{1}{2}\bar{\la}_{\al\be}(\tg_{00}-2)+\bar{\la}^\ga_\al (g_{\be\ga}-\tg_{\be\ga})+iA_\al \tg_{\be 0},\\
&\d_{\al}(\tg_{00}-2)=-2\Re(\la^\ga_\al \tg_{\ga 0}),\\
&\d_\al \<m,\bar{m}\>=-2i A_\al \<m,\bar{m}\>-2\Re\la^\ga_\al \bar{\tg}_{\ga 0},\\
&\d_\al(g_{\be\ga}- \tg_{\be\ga})=\Ga^\si_{\al\be}(g_{\si\ga}-\tg_{\si\ga})+\Ga^\si_{\al\ga}(g_{\si\be}-\tg_{\si\be})+\Re(\bar{\la}_{\be\al}\tg_{\ga 0}+\bar{\la}_{\ga\al}\tg_{\be 0}).
\end{align*}
By ode uniqueness and the choice of the initial data, the desired properties for the frame are propagated spatially.

Once we have the frames $(F_\al^{(n)},m^{(n)})$ at $t=0$, we can invoke the smooth case analysis above,
using \eqref{mo-frame-re} and $\psi_n\in H^{s+2}$ to extend the frame $(F_\al^{(n)},m^{(n)})$ to $t>0$ with initial data $(F_\al^{(n)},m^{(n)})(x,0)$.

In order to obtain a limiting frame $(F_\alpha,m)$ we study the properties of the regular frames $(F_\alpha^{(n)}, m^{(n)})$ in three steps:
\medskip

\emph{a) Uniform bounds.} By \eqref{strsys-cpf-re}, \eqref{decomp-vector}, \eqref{psi-full-reg} and Sobolev embeddings we have
\begin{align*}
    \|\partial_x F^{(n)}_\al\|_{H^s}\lesssim &\| \Ga^{(n)} F^{(n)}_\ga+\la^{(n)} m^{(n)}\|_{H^s}\\
    \lesssim &\|\psi_n\|_{H^s}(|F_\al(\infty)|+| m(\infty)|+\|\partial_x( F^{(n)}_\al, m^{(n)})\|_{H^s})
\end{align*}
and 
\begin{align*}
    \|\partial_x m^{(n)}\|_{H^s}\lesssim &\| A^{(n)} m^{(n)}+\la^{(n)} F_\al^{(n)}\|_{H^s}\\
    \lesssim &\|\psi_n\|_{H^s}(|F_\al(\infty)|+| m(\infty)|+\|\partial_x ( F^{(n)}_\al, m^{(n)})\|_{H^s})
\end{align*}
Then, by the smallness of $\psi_n\in H^s$, we obtain
\[ 
\|\partial_x( F^{(n)}_\al, m^{(n)})\|_{H^s}\lesssim \|\psi_n\|_{H^s}.
\]
\medskip

\emph{b) Sobolev and uniform convergence at $t=0$.} Using an argument similar to that in \emph{a)}, by \eqref{strsys-cpf-re} and Theorem \ref{t:ell-fixed-time} \emph{b)} we have
\begin{align*}
     \|\partial_x (F^{(n)}_\al-F_\al, m^{(n)}-m)\|_{H^s}
     \lesssim &\|\psi_{0n}-\psi_0\|_{H^s}+\|\psi_0\|_{H^s}\|\partial_x (F^{(n)}_\al-F_\al, m^{(n)}-m)\|_{H^s}.
\end{align*}
By the smallness of $\psi_0$, this implies the $H^s$ convergence. The 
uniform convergence at $t=0$ also follows by Sobolev embeddings.

\medskip

\emph{c) a.e. convergence for $t > 0$.}  Here we use \eqref{mo-frame} as an ode in time. The coefficients
converge in $L^2_t$ for a.e. $x$, so the frames $(F^{(n)}_\alpha,m^{(n)})$ will also converge uniformly 
in time for a.e. $x$.  This can be rectified to uniform convergence in view of the uniform Sobolev bounds in (i).
This yields the desired limiting frames $(F_\alpha,m)$.

By \eqref{strsys-cpf-re} we also have
\begin{align*}
    &\|\partial_x (F^{(k)}_\al-F^{(l)}_\al, m^{(k)}-m^{(l)})\|_{L_t^\infty H^s}
     \lesssim \|\psi_{k}-\psi_l\|_{L_t^\infty H^s}\lesssim \|\psi_{0k}-\psi_{0l}\|_{H^s}.
\end{align*}
This shows that the limiting frame satisfies both equations \eqref{mo-frame-re} and \eqref{strsys-cpf-re}, as well the as the uniform bounds in (a).

\subsection{ The moving manifold \texorpdfstring{$\Sigma_t$}{}} Here we propagate the full map $F$ 
by simply integrating \eqref{sys-cpf}, i.e.
\begin{align*}
    F(t)=F(0)+\int_0^t -\Im(\psi \bar{m})+V^\ga F_\ga ds.
\end{align*}
Then by \eqref{strsys-cpf-re}, we have
\begin{align*}
    \d_\al F(t)=\d_\al F(0)+\int_0^t -\Im (\d^A_{\al} \psi \bar{m}-i\la_{\al\ga}V^{\ga} \bar{m})+[\Im(\psi\bar{\la}^{\ga}_{\al})+\nab_{\al} V^{\ga}]F_{\ga} ds,
\end{align*}
which is consistent with above 
definition of $F_\alpha$.

\subsection{ The (SMCF) equation for \texorpdfstring{$F$}{}}
Here we establish that $F$ solves \eqref{Main-Sys}. Using the relation  
$\la_{\al\be}=\d_\al \d_\be F \cdot m$
we have
\begin{align*} 
-\Im (\psi\bar{m})=&-\Im (g^{\al\be}\d_\al\d_\be F\cdot (\nu_1+i\nu_2)\ (\nu_1-i\nu_2))\\
=& (\De_g F\cdot \nu_1) \nu_2-(\De_g F\cdot \nu_2) \nu_1\\
=& J (\De_g F)^{\perp}=J\mathbf{H}(F).
\end{align*}
This implies that the $F$ solves \eqref{Main-Sys}.

\section*{Acknowledgments}
J. Huang would like to thank Prof. Lifeng Zhao for many
inspirations and discussions, and Dr. Ze Li for carefully reading the manuscript, helpful discussions and comments.

\end{document}